\documentclass[12pt]{article}
\RequirePackage[OT1]{fontenc}
\RequirePackage[numbers]{natbib}
\RequirePackage[colorlinks,citecolor=blue,urlcolor=blue]{hyperref}

\usepackage{latexsym} 
\usepackage{graphicx}
\usepackage{epsfig} 
\usepackage[latin1]{inputenc}  
\usepackage{float} 
\usepackage{paralist}
\usepackage{url}
\usepackage{tikz}
\usetikzlibrary{decorations.pathreplacing}
\usepackage{float} % für Fließumgebungen; Platzierung H verschiebt nicht
\usepackage{amssymb}
\usepackage{amsmath, amsthm}
\usepackage{amsfonts}
\usepackage{bbm}
\usepackage{relsize}
\usepackage{color}

\usepackage[left=1truein, right=1truein, top=0.7truein, bottom=0.9truein]{geometry}

\let\oldbibliography\thebibliography  
\renewcommand{\thebibliography}[1]{%
  \oldbibliography{#1}%
   \setlength{\itemsep}{1pt}% ZeilenabstÃ€nde im Literaturverzeichnis
}
\theoremstyle{plain}
\newtheorem{satz}{Theorem}[section]
\newtheorem{proposition}[satz]{Proposition}
\newtheorem{lemma}[satz]{Lemma}
\newtheorem{kor}[satz]{Corollary}

\newtheorem{bem}[satz]{Remark}

\newtheorem{notat}[satz]{Notation}

\renewcommand{\epsilon}{\varepsilon}

\newcommand{\M}{\mathrm{M}}

\def\3{\ss}
\def\l{\lbrace }
\def\r{\rbrace}
\def\t{\tilde}
\def\bl{\biggl}
\def\br{\biggr}

\def \en{\mathbb{N}}
\def \P{\mathbb{P}}
\def \E{\mathbb{E}}
\def \bs\setminus {\mathbb{P}}
\def \bs{I \!\! B}

\newcommand{\bea}{\begin{eqnarray*}}
\newcommand{\eea}{\end{eqnarray*}}
\newcommand{\be}{\begin{eqnarray}}
\newcommand{\ee}{\end{eqnarray}}

\newcommand{\Var}{ \mbox{Var} }
\newcommand{\Cov}{ \mbox{Cov} }

\newcommand*{\pkt}{\makebox[1.5ex]{\textbf{$\mathbin{\scalebox{1.2}{\ensuremath{
\cdot}}}$}}}
\newcommand*{\1}{\makebox[1.5ex]{\textbf{$\mathbin{\scalebox{1.1}{\ensuremath{
\mathbbm{1}}}}$}}}
\def \be{\mathbb{B}}
\def\bl{\biggl(}
\def\br{\biggr)}
\def\mcl{\mathcal{L}}

\def\mck{\mathcal{K}}

\newenvironment{bew}{\noindent $Proof$:}{\hspace*{\fill}
$\square$\\}
\newenvironment{bew1}{\smallskip $Proof$:}{}

\newcommand{\mvert}{\,|\,}

\begin{document}

% \begin{frontmatter}

\title{Convergence Rates for the Degree Distribution in a Dynamic Network Model}
% \runtitle{Degree Distribution in a Dynamic Network Model}

\author{\renewcommand{\thefootnote}{\arabic{footnote}} Fabian K\"uck\footnotemark[1]$^{\hspace*{3.7pt},\hspace*{0.5pt}}$\footnotemark[2] \ and Dominic Schuhmacher\footnotemark[1]$^{\hspace*{4pt},}$\footnotemark[3] \\[2mm] University of G\"ottingen}

\footnotetext[1]{Institute for Mathematical Stochastics, University of G\"ottingen,
Goldschmidtstra{\ss}e 7, 37077 G\"ottingen, Germany.} 
\footnotetext[2]{E-mail: fabian.kueck@mathematik.uni-goettingen.de}
\footnotetext[3]{E-mail: schuhmacher@math.uni-goettingen.de}
\maketitle

% \begin{aug}
% \author{\fnms{Fabian} \snm{K\"uck}\ead[label=e1]{fabian.kueck@mathematik.uni-goettingen.de}}
% \and
% \author{\fnms{Dominic} \snm{Schuhmacher}\ead[label=e2]{dominic.schuhmacher@mathematik.uni-goettingen.de}}
% 
% \address[a]{Institute for Mathematical Stochastics\\[-2mm]
% University of G\"ottingen
% \\[-2mm]
% Goldschmidtstra{\ss}e 7\\[-2mm] 37077 G\"ottingen\\[-2mm] Germany\\[-2mm]
% \printead{e1}\\[-2mm]
% \phantom{E-mail:\ }\printead*{e2}}
% 
% \affiliation{University of G\"ottingen}
% 
% \runauthor{F. K\"uck and D. Schuhmacher}
% 
% \end{aug}

%\pagenumbering{roman}

\begin{abstract}
% In the stochastic network model of Britton
% and Lindholm [Dynamic random networks in dynamic populations.
% \textit{Journal of Statistical Physics}, 2010], the number of individuals evolves 
% according to a linear
% birth and death process with per-capita birth rate $\lambda$ and per-capita 
% death rate $\mu<\lambda$. A random social index is assigned to each 
% individual at birth, which controls the rate at which connections to 
% other individuals are created. Britton and Lindholm give a
% somewhat rough proof for the convergence of the degree distribution
% in this model towards a mixed Poisson distribution. We derive a rate for this
% convergence giving precise arguments. In order to do so, we deduce the degree 
% distribution at finite time and derive an approximation result for mixed Poisson 
% distributions to com\-pute an upper bound
% for the total variation distance to the
% asymptotic degree distribution. We prove 
% several general results about 
% linear birth and death processes along the way. Most notably, we derive the 
% age distribution 
% of an individual picked uniformly at random at some finite time.
% has <= 174 words
In the stochastic network model of Britton
and Lindholm \cite{b10}, the number of individuals 
evolves according to a supercritical linear birth and death process, and a random social 
index is assigned to each individual at birth, which controls the rate at 
which connections to other individuals are created. We derive a rate for the 
convergence of the degree distribution in this model towards the mixed Poisson distribution determined by Britton and Lindholm based on heuristic arguments. In order to do so, we 
deduce the degree distribution at finite time and derive an approximation result 
for mixed Poisson distributions to com\-pute an upper bound
for the total variation distance to the
asymptotic degree distribution. 
\vspace*{3.5mm}

\noindent
% \textit{MSC 2010 subject classifications: Primary 60J80; secondary 60E05, 62E20.}\\[0mm]
\textit{Keywords: mixed Poisson distribution, dynamic random graph, small worlds.}
\end{abstract}

\section{Introduction}\label{sec1}
``Network Science'' is a relatively young, rapidly growing research area 
dealing with complex systems with an underlying graph structure (see e.g.\ 
the recent book by van der Hofstad~\cite{hof15}). One of the first random graph 
models is the well-known Erd\H{o}s-R\'enyi model, which is a static model, 
i.e.\ it describes a random network at a fixed time. Although this model shows 
interesting behaviour, further random graph models were needed for the 
description of real networks since some empirical networks have important 
properties that are not represented by Erd\H{o}s-R\'enyi Graphs. One of these 
is the power law property of the degree distribution. Preferential 
attachment models exhibit such a power law behaviour asymptotically and were 
popularized by Barab\'asi and Albert \cite{bara}. Pek\"oz, 
R\"ollin and Ross \cite{pek} established convergence rates for the degree 
distribution in 
a discrete-time preferential attachment model in terms of the total variation 
distance. Besides preferential models, the so-called fitness models have gained huge popularity in the recent years. In those models, the attachment does not only depend on the degree but also on a random intrinsic fitness that is determined at the birth of each node. The best-known fitness model was introduced in \cite{bb01}. In this model the attachment mechanism is a combination with preferential attachment. Pure fitness models were for example considered in \cite{cal02} and \cite{smo13}. The 
time-continuous random graph model that was introduced by Britton and 
Lindholm \cite{b10} and that we investigate here can be seen as pure fitness model that is particular realistic due to time continuity and possibly dying nodes as well as edges. Depending 
on the application, either preferential models or the model by Britton and 
Lindholm can be more realistic. Note that it can be 
shown that the degrees of nodes in the model by Britton and Lindholm can in 
particular exhibit power laws such that this model displays an interesting 
alternative mechanism for producing graphs with this property.
% An 
% alternative to preferential attachment models is the 
% time-continuous random graph model that was introduced by Britton and 
% Lindholm \cite{b10} and that we consider here. The main difference between this 
% model and the preferential attachment models is that in the former the 
% popularity of each node is determined at birth and stays the same for the 
% whole lifetime whereas the popularity in preferential attachment models is 
% proportional to the degree of the node, which changes over time. Depending 
% on the application, either model can be more realistic. Note that it can be 
% shown that the degrees of nodes in the model by Britton and Lindholm can in 
% particular exhibit power laws such that this model displays an interesting 
% alternative mechanism for producing graphs with this property.
\\

Let us give a \emph{loop-free} version of the definition of the original 
\emph{dynamic network model by Britton and Lindholm}; see \cite{b10} and 
\cite{b11}:\\
We examine a finite undirected graph without loops that develops over time. The 
node process $(Y_t)_{t\geq 0}$ is a linear birth and death process with initial 
value one. Thus each node gives birth at constant rate $\lambda$ and dies at 
constant rate $\mu$ independently from other nodes. We assume that $(Y_t)_{t\geq0}$ has right-continuous trajectories.
The process $(Y_t)_{t\geq0}$ and all other random variables that are defined in 
what follows to describe the dynamic random graph are defined on a common 
underlying probability space $(\Omega,\mathcal{A},\P)$. 

We assume $\lambda>\mu$, so that the node process $(Y_t)_{t\geq0}$ is a supercritical 
continuous-time Markov branching process.  From standard branching process 
theory, we obtain that the random variable ${ W:=\lim_{t\to\infty} Y_t e^{-(\lambda-\mu)t}}$ exists almost surely and satisfies $\P(W=0)= \mu/\lambda$ and $\mcl( W| W>0)= \mathrm{Exp}(\frac{\lambda-\mu}{\lambda})$ (see e.g.\ \cite{har}, page~319).

We equip every node $i$ with a positive random social index $S_i$, where the 
$S_i$ are i.i.d.\ with finite expectation and independent of all other random variables. 

This allows us to define the development of the edge set. At birth every node is 
isolated.
During its lifetime and \emph{as long as there is at least one other node}, 
node $i$ generates and destroys edges according to a birth and death process 
with constant birth rate $\alpha S_i$ and per-edge death rate $\beta$. Here 
$\alpha$ and $\beta$ are positive constants. The ``second'' node of each newly 
born edge is chosen uniformly at random from the set of all \emph{other} living 
nodes, and all of the edge processes (including the choices of the second 
nodes) are independent of each other and all other events. 

In addition to the direct destruction of edges in the above process, all edges 
connected to a certain node are removed when the node dies.
\begin{bem}
 The only difference to the definition by Britton and Lindholm is that we do not 
allow loops because these are not present in most applications. Note that the 
proofs become slightly simpler if we use the original model by Britton and 
Lindholm, essentially because times where $Y_t=1$ need not obtain special 
treatment. The upper bounds remain largely the same; see also 
Remark~\ref{rem:loopsdontmatter} for the pure birth case. 

Note that we still allow multiple edges. However, it can be shown that those 
are negligible in the sense that the probability that a randomly picked node has at 
least one multiple edge converges to zero at an exponential rate (see 
Appendix \hyperref[A3]{A3}). This allows us to formulate the main result also for the case where 
we ignore multiple edges (see Corollary \ref{without_multiple} below).
\end{bem}

We refer to the distribution of the number of edges incident to a node picked 
uniformly at random from all living nodes at time $t$ given the number of nodes 
is positive as \emph{degree distribution}, and denote it by $\nu_t$. In \cite{b10}, Britton 
and Lindholm give a rather heuristic argument for the convergence of the degree 
distribution
in the original model towards a mixed Poisson distribution $\nu$. 
It is the 
main purpose of this paper to give a rate in total variation distance rather than a mere convergence result, providing full proof for this rate and thereby also for the convergence. The distribution $\nu$ is given by
\begin{equation*}
  \nu = \mathrm{MixPo} \biggl( \frac{\alpha}{\beta+\mu} \bigl( S + \E(S) \bigr) 
\bigl( 1-e^{-(\beta+\mu) A}) \biggr),
\end{equation*}
where $A \sim \mathrm{Exp}(\lambda)$, $S$ has the social index distribution, and 
$A$ and $S$ are independent. Here $\mathrm{MixPo}$ denotes the mixed Poisson 
distribution. The following result is an immediate consequence of the main results proved in this article, 
Theorems~\ref{ss1} and~\ref{thm}.
\begin{satz} \label{superthm}
  Let $\E(S^2)<\infty$. Then we obtain for the degree distribution $\nu_t$ in 
the Britton--Lindholm Model without 
loops
  \begin{enumerate}[(a)]
    \item if $\mu=0$, then $d_{TV}(\nu_t, \nu) = 
O\bigl(\sqrt{t} 
e^{-\frac{1}{2}\lambda t}\bigr)$ as $t \to \infty$;
    \item if $\mu>0$, then $d_{TV}(\nu_t, \nu) = 
O\bigl(t^2 
e^{-\frac{1}{6}(\lambda-\mu) t}\bigr)$ as $t \to \infty$.
  \end{enumerate}
\end{satz}
\begin{bem}
%  From the proof of Theorem \ref{superthm}(a), we can conclude that the rate is 
% of the same order as $\E(Y_t^{-1/2}) $ if $\mu=0$. From standard branching 
% process theory, it can be deduced that this expectation does not decrease at a 
% faster rate than $e^{-\frac{1}{2}\lambda t}$. Thus the rate given in the 
% theorem is very close to the actual one.

A positive death rate $\mu$ leads to a higher variability in the degree distributions for finite $t$ since in particular 
the death of a highly connected node (hub) can have a large impact. Thus we would not expect the same rate as in the pure birth case. However, the actual factor in the exponential
rate may be larger than the one stated in the theorem.  
\end{bem}

Theorem \ref{superthm} has consequences for the case where we ignore multiple 
edges. Let $\tilde \nu_t$ be the distribution of the number of neighbours of a 
node picked uniformly at random from all living nodes at time $t$ given the 
number of nodes at time $t$ is positive. We refer to $\tilde \nu_t$ shortly as 
distribution of the number of neighbours. The convergence of this distribution to 
the asymptotic degree distribution $\nu$ is an immediate consequence of the 
following corollary, which is proved in Appendix \hyperref[A3]{A3}.
\begin{kor}\label{without_multiple}
 Let $\E(S^2)<\infty$. For the distribution of the number of neighbours in the 
Britton--Lindholm Model 
without loops, we have that $d_{TV}(\tilde \nu_t, \nu) = O(t^{2 }
e^{-\frac{1}{6}(\lambda-\mu)t})$ as $t \to \infty$.
\end{kor}
\begin{bem}
 Britton and Lindholm also introduced a modified model where the ``second'' node 
of each newly born edge is not picked uniformly. Instead the probability for 
each node being the ``second'' node is proportional to its social index. We 
expect that one can prove similar results for this modified model analogously to 
the results for the model that we treat here. Note that the modification can 
lead to a significantly higher ratio of multiple edges. Therefore the modified 
version is less interesting for many applications.
\end{bem}

The rest of the paper is organized as follows. Section~\ref{sec:ddfinite} gives the degree distribution at finite time, which is deduced by using 
results about the birth and death processes of the edges. In 
Section~\ref{sec:bounds}, we derive upper bounds for the total variation 
distance between finite-time and asymptotic degree distributions, treating the 
pure birth case and the general case separately, which leads to 
Theorem~\ref{superthm} above. In order to achieve this, we derive a universal bound for the total variation distance 
between two general mixed Poisson distributions; see Theorem \ref{s2}. In Appendix \hyperref[sec:lbdp]{A1}, we 
collect some well-known and less known facts about linear birth and death processes.  Appendix \hyperref[A2]{A2} contains technical lemmas required for the proof of our main theorem including more detailed facts about linear birth and death processes that might be of general interest. In Appendix \hyperref[A3]{A3}, we treat the negligibility of multiple edges. 
% For a few of the more elementary proofs throughout the paper, we refer to the technical report \cite{ks15}.

\section{Degree distribution at finite time}\label{sec:ddfinite}

If at least two nodes are alive, each node $i$ spawns edges to other nodes 
according to a birth and death process with constant birth rate $\alpha' = 
\alpha S_{i}$ and linear death rate with factor $\beta$. The one-dimensional 
distributions of this process are well-known. 
% For completeness a derivation can 
% be found in~\cite[Subsections 4.1--4.2]{ks15}.
%
\begin{proposition}\label{prop:odi_idp}
Let $(Z_t)_{t\geq 0}$ be a birth and death process with constant birth rate 
$\alpha'$ and linear death rate $\beta' = \beta n$ if the process is in state $n 
\in \en=\lbrace 1,2,\ldots\r$. If the process is started deterministically at $k 
\in\en$, we have
\begin{equation*}
   Z_t \sim \mathrm{Po} \bigl(\tfrac{\alpha'}{\beta} (1-e^{-\beta t}) \bigr) 
\ast \mathrm{Bin}\bigl(k,e^{-\beta t}\bigr)
\end{equation*}
for every $t \geq 0$, where $\ast$ denotes convolution.
\end{proposition}
In what follows, we derive the degree distribution at finite time $T$ based on 
this result. 

\subsection{The pure birth case}\label{ssec:ddfinite1}

We deal with the case $\mu=0$ first. Let the 
nodes be ordered by their birth times. Let $S_i$ be the social index of node $i$
and $A_i(T)$ its age at time $T$. For 
convenience, we define $A_2(T)=0$ if $Y_T=1$. Furthermore, given $Y_T=y_T$, let 
the random variable $J_T$ be uniformly distributed on $\lbrace 1, \ldots , y_T 
\rbrace$ and independent of the ages, the social indices and the rest of the 
path $(Y_t)_{0\leq t <T}$. We interpret $J_T$ as the index of a node that is 
randomly picked at time $T$ among all living nodes. 

For the time being, we condition on $(Y_t)_{0\leq t \leq T}=(y_t)_{0\leq t \leq 
T}$, ${(S_k)_{k\in\en}=(s_k)_{k\in\en}}$ and $J_T=j_T$. Let $T=a_1 > \ldots 
> 
a_{y_T}$ denote the corresponding ages of the individuals.

Firstly, we consider the number of edges created by node $j_T$. Assuming 
$y_T>1$, the edges created by $j_T$ form a birth and death process with constant 
birth rate $\alpha s_{j_T}$ and linear 
death rate with factor $\beta$, started in zero at time $T-a_{\max(j_T,2)}$ 
since no edges are created if there is only one living node. By 
Proposition~\ref{prop:odi_idp}, the number of edges alive at time $T$ that 
$j_T$ has created has distribution
\begin{equation} \label{eq:outgoing}
  \mathrm{Po}\biggl(\frac{\alpha s_{j_T}}{\beta}(1-e^{-\beta 
a_{\max(j_T,2)}})\biggr).
\end{equation}

We have to add to this the number of edges alive at time $T$ that \emph{other} 
nodes have created and connect to $j_T$. Consider a fixed node $i \in \lbrace 
1,\ldots,y_T \rbrace \setminus \lbrace j_T \rbrace$ and some time interval of 
the form $[T-a_l,T-a_{l+1})$ for $l \geq i \vee 2$. The number of edges that 
connect $i$ to $j_T$ and that survive until the end of this interval, i.e.\ 
until the birth time $T-a_{l+1}$ of node $l+1$, can be described by a birth and 
death process of the above type again. The birth rate is constant and equal to 
$\alpha s_i 
\frac{1}{l-1}$ because node $i$ creates edges at rate $\alpha s_i$ and 
$\frac{1}{l-1}$ is the probability that an edge that is created in the interval 
$[T-a_l,T-a_{l+1})$ is connected to $j_T$. The death rate is linear again with 
factor $\beta$. By Proposition~\ref{prop:odi_idp}, the number of edges created 
by $i$ in $[T-a_l,T-a_{l+1})$ that connect to $j_T$ and survive until 
$T-a_{l+1}$ is $\mathrm{Po}\bigl(\frac{\alpha s_i}{(l-1) \beta}(1-e^{-\beta 
(a_l-a_{
l+1})})\bigr)$-distributed.

We can extend the time interval by one birth time, i.e.\ we compute the 
distribution of the number of edges that $i$ creates in $[T-a_l,T-a_{l+2})$, 
connect to $j_T$, and survive until $T-a_{l+2}$ by conditioning on the number 
$Z$ of edges that $i$ creates in $[T-a_l,T-a_{l+1})$, connect to $j_T$ and 
survive until $T-a_{l+1}$. Given $Z=z$, by 
Proposition~\ref{prop:odi_idp} and the Markov property, the number of 
edges 
that $i$ creates in $[T-a_l,T-a_{l+2})$, connect to $j_T$, and survive until 
$T-a_{l+2}$ has distribution $\mathrm{Po} \bigl(\tfrac{\alpha s_i}{l \beta} 
(1-e^{-\beta  (a_{l+1}-a_{l+2})}) \bigr) \ast \mathrm{Bin}\bigl(z,e^{-\beta 
(a_{l+1}-a_{l+2})}\bigr)$. We already know from above that $Z \sim 
\mathrm{Po}\bigl(\frac{\alpha s_i}{(l-1) \beta}(1-e^{-\beta 
(a_l-a_{l+1})})\bigr)$. Thus if we do not condition on $Z=z$, we 
obtain the distribution 
\begin{align}\label{3.2}
 \mathrm{Po}\biggl(\frac{\alpha s_i}{l \beta}(1-e^{-\beta(a_{l+1}-a_{l+2})}) + 
e^{-\beta (a_{l+1} - a_{l+2})} \frac{\alpha s_i}{(l-1) 
\beta}(1-e^{-\beta(a_l-a_{l+1})})\biggr).
\end{align}

Starting at $l=\max(i,j_T)$ and iterating the procedure that leads to 
\eqref{3.2} until the whole 
interval $[T-a_{\max(i,j_T)}, T)$ is spanned, we see that the number of edges 
alive at time $T$ that connect to $j_T$ but have been created by other nodes has 
distribution 
\allowdisplaybreaks{\begin{align} 
 & \mathrm{Po}\biggl( \sum\limits_{\substack{{i=1}\\i \not= j_T}}^{y_T} 
\frac{\alpha s_i (1-e^{-\beta a_{y_T}})}{(y_T-1) \beta} + 
\sum\limits_{\substack{{i=1}\\i \not= j_T}}^{y_T} e^{-\beta a_{y_T}} 
\sum_{l=i\vee j_T}^{y_T - 1} \biggl( \prod_{k=1}^{y_t-l-1} 
e^{-\beta(a_{l+k}-a_{l+k+1})} \biggr)\frac{\alpha s_i (1- 
e^{-\beta(a_l-a_{l+1})})}{(l-1) \beta}\biggr)
 \notag\\
 & = \mathrm{Po}\biggl( \frac{\alpha}{\beta} \sum\limits_{\substack{{i=1}\\i 
\not= j_T}}^{y_T} \frac{s_i}{y_T-1}(1-e^{-\beta a_{y_T}}) + 
\frac{\alpha}{\beta} 
\sum\limits_{\substack{{i=1}\\i \not= j_T}}^{y_T} \sum_{l=i\vee j_T}^{y_T - 1} 
\frac{s_i}{l-1} (e^{-\beta a_{l+1}}-e^{-\beta a_l}) \biggr). \label{eq:incoming}
\end{align}}

Since the numbers of outgoing and incoming edges are conditionally independent, the desired 
degree distribution is obtained by convoluting \eqref{eq:outgoing} and 
\eqref{eq:incoming}. Lifting the conditioning on $(Y_t)_{0\leq t\leq T}$, $(S_k)_{k\in\en}$ and~$J_T$, we 
arrive at the following result.
\begin{satz} \label{thm:ddpure}
  For $\mu=0$, the degree distribution in the Britton--Lindholm model without 
loops is the $\mathrm{MixPo}(\Lambda_T)$ distribution, where 
\begin{align}
 \Lambda_T=&\frac{\alpha S_{J_T}}{\beta}(1-e^{-\beta A_{\max(J_T,2)}(T)}) + 
\frac{\alpha}{\beta} \sum\limits_{\substack{{i=1}\\i \not= J_T}}^{Y_T} 
\frac{S_i}{Y_T-1}(1-e^{-\beta A_{Y_T}(T)}) \notag
 \\
 &+ \frac{\alpha}{\beta} \sum\limits_{\substack{{i=1}\\i \not= J_T}}^{Y_T} 
\sum_{l=i\vee J_T}^{Y_T - 1} \frac{S_i}{l-1} (e^{-\beta A_{l+1}(T)}-e^{-\beta 
A_l(T)}). \label{lambda}
\end{align}
\end{satz}

\subsection{The general case} \label{ssec:ddfinite2}
For general $\mu$, the degree distribution can be determined similarly to 
the pure birth case.

Let $0=T_1< \ldots < T_{\mathfrak{B}_T+\mathfrak{D}_T}$ be the event 
times up to $T$. We still enumerate nodes according to their birth times. Denote by $0=T_1^+<\ldots < T_{\mathfrak{B}_T}^+$ their birth times and by $T_i^-$ the death 
time of the $i$-th node. Finally, given the subset of $\lbrace 
1,\ldots,\mathfrak{B}_T\rbrace$ that contains the indices of all living nodes at time $T$, 
let $J_T$ be uniformly distributed on this set and independent of all other 
random variables as before. Note that $S_{J_T}$ is (stochastically) independent 
of $((Y_t)_{0\leq 
t \leq T}, J_T)$.

We condition on $(Y_t)_{0\leq t \leq T}=(y_t)_{0\leq t \leq T}$ with $y_T>0$, 
the social indices $(S_k)_{k\in\en}=(s_k)_{k\in\en}$ and $J_T=j_T$ again. Let 
$b_T$ and $d_T$ denote the corresponding number of births and deaths up to time 
$T$, respectively. Furthermore, let $0=t_1<t_2<\ldots<t_{b_T+d_T}$ and 
$0=t^+_1<t^+_2<\ldots<t^+_{b_T}$ be the corresponding event times and birth 
times up to time $T$, respectively. 

Since deaths of other nodes reduce the number of edges created by node $j_T$, we 
cannot derive the distribution of the number of outgoing edges at time $T$ in 
the same way as for the pure birth process. However, we can derive the total 
number of edges incident to $j_T$ in a similar way as the number of incoming 
edges in the pure birth case. 
Consider a fixed node $i\not=j_T$ that is alive at time $T$ (provided there are 
any) and some time interval of the form $[t_l,t_{l+1})$ for $t_l\geq t_i^+\vee 
t_{j_T}^+$ and $t_{l+1}\leq T$. Note that the nodes $i$ and $j_T$ create edges 
with rate $\alpha s_i$ and $\alpha s_{j_T}$, respectively, and that the 
probability that an edge that is created by $i$ is connected to $j_T$ is 
$\frac{1}{y_{t_l}-1}$ and equal to the probability that an edge that is created 
by $j_T$ is connected to $i$. Thus the number of edges created between $i$ and 
$j_T$ that survive until time $t_{l+1}$ can be described by a birth and death 
process with constant birth rate $\alpha (s_i+s_{j_T}) \frac{1}{y_{t_l}-1}$ and 
linear death rate with 
factor $\beta$ like before. By Proposition~\ref{prop:odi_idp} we obtain that the 
number of edges that are created in $[t_l,t_{l+1})$ between $i$ and $j_T$ and 
survive until $t_{l+1}$ has distribution
\begin{equation*}
  \mathrm{Po}\biggl(\frac{\alpha (s_i+s_{j_T})}{(y_{t_l}-1)\beta} (1-e^{-\beta 
(t_{l+1} -t_l)}) \1_{\lbrace y_{t_l} > 1 \rbrace}\biggl).
\end{equation*}

Let the function $r=r_{(y_t)_{0\leq t\leq T}}:\lbrace 1,\ldots,b_T \rbrace 
\rightarrow \lbrace 1,\ldots,b_T+d_T \rbrace$ be defined such that $t_j^+= 
T_{r(j)}$ for all $j \in \lbrace 1,\ldots,b_T \rbrace $, i.e.\ $r$ maps birth 
number to event number.

Applying the same iterative procedure as in Subsection~\ref{ssec:ddfinite1}, 
starting at $l = r(i)\vee r(j_T)$ and continuing until the whole interval 
$[t_i^+\vee t_{j_T}^+,T)$ is spanned, we can see that the number of edges 
between $i$ and $j_T$ alive at time $T$ has distribution 
\begin{align}
 &\mathrm{Po}\biggl(\frac{\alpha (s_i+s_{j_T})}{(y_T-1)\beta} 
(1-e^{-\beta(T-t_{b_T+d_T})}) \1_{\lbrace y_T>1 \rbrace} 
 \notag \\
 &+ e^{-\beta(T-t_{b_T+d_T})}\hspace{-2.5pt}\sum\limits_{l=r(i)\vee r(j_T)}^{b_T+d_T-1} 
\biggl(\prod\limits_{k=1}^{b_T+d_T-l-1} e^{-\beta (t_{l+k+1} - t_{l+k})} 
\frac{\alpha (s_i+s_{j_T})}{(y_{t_l} 
-1)\beta}(1-e^{-\beta(t_{l+1}-t_l)})\biggr) 
\1_{\lbrace y_{t_l} > 1\rbrace } \biggr). \label{eq:between2}
\end{align}

Since the processes of edges between $i$ and $j_T$ are conditionally mutually independent for 
different $i$, we obtain the desired degree distribution by convoluting the 
distribution \eqref{eq:between2} for $i$ running in ${\{i' \in \{1,\ldots,b_T\} 
\colon i' \neq j_T, t_{i'}^{-} > T \}}$ and lifting the conditioning on $(Y_t)_{0\leq t\leq T}$, 
$(S_k)_{k\in\en}$ and~$J_T$.

\begin{satz} \label{thm:dd}
  The degree distribution in the Britton--Lindholm model without loops is the 
$\mathrm{MixPo}(\Lambda^*_T)$ distribution, where $\Lambda^*_T$ is a 
random variable with 
$\mathcal{L}(\Lambda^*_T)=\mathcal{L}(\Lambda_T|Y_T>0)$ and
\begin{align}
  \Lambda_T:=&\frac{\alpha}{\beta}\sum\limits_{\substack{{i=1}\\i \not= 
J_T}}^{\mathfrak{B}_T} \frac{ S_i+S_{J_T} }{(Y_T-1)} \1_{\lbrace T_i^- > T 
\rbrace}(1-e^{-\beta(T-T_{\mathfrak{B}_T+\mathfrak{D}_T})}) \1_{\lbrace Y_T>1 \rbrace} \notag
 \\
 &+\frac{\alpha}{\beta}\sum\limits_{\substack{{i=1}\\i \not= 
J_T}}^{\mathfrak{B}_T}\1_{\lbrace T_i^- > T \rbrace} \sum\limits_{l=r(i)\vee 
r(J_T)}^{\mathfrak{B}_T+\mathfrak{D}_T-1}   \frac{ S_i+S_{J_T}}{(Y_{T_l} -1)}(e^{-\beta (T - 
T_{l+1})}-e^{-\beta(T-T_l)})\1_{\lbrace Y_{T_l} > 1\rbrace }. \label{lambdag}
\end{align}
\end{satz}

\section{Bounds on the total variation distance between the finite-time and the 
asymptotic degree distributions}\label{sec:bounds}

Since in the pure birth case we obtain a much better bound with considerably 
less work, we treat the cases $\mu=0$ and $\mu>0$ separately in 
Subsections~\ref{ssec:boundpure} and~\ref{ssec:boundgen}, giving the ideas of 
the proofs. The proofs themselves are deferred to 
Subsection~\ref{ssec:boundproofs}.

\subsection{The pure birth case} \label{ssec:boundpure}
Let $\mu=0$.  Fix $T>0$, and let the random variable $\Lambda_T$ be defined as 
in Theorem~\ref{thm:ddpure}.
Furthermore, let $S_{J_\infty}$ and $A_{J_\infty}$ be independent random 
variables such that $S_{J_\infty}(\omega)\hspace{-0.5mm}=\hspace{-0.5mm}S_{J_T}(\omega)$ and 
${A_{J_\infty}(\omega)\hspace{-0.5mm} = \hspace{-0.5mm}F_{\infty}^{-1}(F_T (A_{J_T}(T,\omega)))}$ for all 
$\omega \in \Omega$, where $F_T$ and $F_\infty$ are the cumulative distribution 
functions of $A_{J_T}(T)$ and the $\mathrm{Exp}(\lambda)$ distribution, 
respectively. Note that we have that 
$\mathcal{L}(A_{J_T}) \rightarrow \mathrm{Exp}(\lambda)$ weakly (see  Proposition~\ref{age} in the appendix). The random 
variable $A_{J_\infty}$ is $\mathrm{Exp}(\lambda)$ distributed since $F_T$ is 
continuous and hence $F_{T}(A_{J_T}(T))$ is uniformly distributed on $[0,1]$. 
Moreover, since the exponential distribution stochastically dominates the 
truncated exponential distribution, we have $F_{\infty}(A_{J_T}(T,\omega)) \leq 
F_T(A_{J_T}(T,\omega))$ for all $\omega \in \Omega$. Since $F_\infty^{-1}$ is 
increasing, it follows that $A_{J_T}(T,\omega) \leq A_{J_\infty}(\omega)$ for 
all $\omega \in \Omega$. 

Let $S$ be a generic random variable that is distributed according to the 
distribution of the social indices and let
\begin{align}
 \M=\frac{\alpha S_{J_\infty}}{\beta} (1-e^{-\beta A_{J_\infty}})+ \frac{\alpha 
\E(S)}{\beta}(1- e^{-\beta A_{J_\infty}}).\label{mu}
\end{align}

We know from Theorem~\ref{thm:ddpure} that $\mathrm{MixPo}(\Lambda_T)$ is the 
degree distribution at time $T$ in the pure birth case. Theorem \ref{ss1} below 
implies that $\mathrm{MixPo}(\Lambda_T)$ converges at rate of just a bit 
slower than $e^{-\frac{1}{2}\lambda T}$ to $\mathrm{MixPo}(\M)$ as $T 
\rightarrow 
\infty$, which is the asymptotic degree distribution already stated in Section 
3.2 of \cite{b10}. Note that the theorem below is much more powerful as it gives 
an exact distance bound for finite~$T$.
\begin{satz}\label{ss1}
Let $\sigma_S$ be the standard deviation of $S$. Then for $T \geq \frac{\log(2)}{\lambda}$, we have
\begin{align*}
&d_{TV}(\mathrm{MixPo}(\Lambda_T),\mathrm{MixPo}(\M))\leq \frac{\sqrt{32} 
\alpha}{\sqrt\lambda}\sigma_S\sqrt{{ T }} e^{-\frac{1}{2}\lambda T}  
 +{4 \alpha} \E(S)\bl { T}+\frac{\lambda}{\beta(\beta+\lambda)}\br e^{-\lambda 
T}.
\end{align*}
\end{satz}

The main idea of the proof is as follows.
Theorem \ref{s2} below is a simple but crucial result that allows to bound 
$d_{TV}(\mathrm{MixPo}(\Lambda_T),\mathrm{MixPo}(\M))$ by 
$\E(|\Lambda_T-\M|)$. In order to bound $\E(|\Lambda_T-\M|)$ further, we 
use that the expected value of the second summand of the right-hand side of 
\eqref{lambda} becomes small as $T \rightarrow \infty$ since the age 
$A_{Y_T}(T)$ of the youngest individual at time $T$ converges quickly to 0, and 
compare the other summands of the right-hand sides of \eqref{lambda} and 
\eqref{mu}. 

For the comparison of the last summands in \eqref{lambda} and \eqref{mu}, 
respectively, we note that the average of the social indices becomes close to 
$\E(S)$ by the Law of Large Numbers. Then we use again that the age 
$A_{Y_T}(T)$ of the youngest individual at time $T$ converges quickly to 0 and 
that $A_{J_T}(T)$ converges quickly to $A_{J_\infty}$ (see Proposition~\ref{age} in the appendix).

Finally, the expected absolute value of the difference between the first summand 
of \eqref{lambda} and the first summand of \eqref{mu} again becomes small since 
$A_{J_T}(T)$ converges quickly to $A_{J_\infty}$.
\begin{bem} \label{rem:loopsdontmatter}
For the original 
model of Britton and Lindholm with loops, we can adapt the proof of Theorem 
\ref{ss1} in such a way that the 
upper bound remains exactly the same. 
\end{bem}

\subsection{The general case} \label{ssec:boundgen}
Let $S_{J_\infty}(\omega)=S_{J_T}(\omega)$ as in the pure birth case. Furthermore, let 
$A_{J_\infty}=Z$, where $Z$ is the random variable from Corollary~\ref{gage}. 
From the proof of this corollary (see~\cite{age}, proof of Corollary~2.4), we 
obtain that we may write
$$A_{J_\infty}(\omega) = \1_{\lbrace J_T (\omega)<Y_T 
(\omega)\rbrace}F_\infty ^{-1}(F_*(A_{J_T}(T,\omega)) + \1_{\lbrace 
J_T(\omega)=Y_T(\omega)\rbrace} \tilde Z(\omega),$$
where $F_\infty$ is the cumulative distribution function of the 
$\mathrm{Exp}(\lambda)$ 
distribution as before, $\tilde Z \sim F_\infty$ independent of everything and 
$F_*(t)=1-\frac{e^{-\lambda t}-e^{-(\lambda-\mu)T}e^{-\mu 
t}}{1-e^{-(\lambda-\mu)T}}$. Note that, given $Y_T>0$, we have $A_{J_\infty}\sim \mathrm{Exp}(\lambda)$.

Recall the definition of the parameter random variable $\Lambda_T$ from 
\eqref{lambdag} and let $\Lambda^*_T$ be a random variable with 
${\mathcal{L}(\Lambda^*_T)= \mathcal{L}(\Lambda_T | Y_T>0)}$ as before. 
Moreover, set 
\begin{align}
 \M=\frac{\alpha S_{J_\infty}}{\beta+\mu} (1-e^{-(\beta+\mu) A_{J_\infty}})+ 
\frac{\alpha \E(S)}{\beta+\mu}(1- e^{-(\beta+\mu) A_{J_\infty}})\label{mu2}
\end{align}
and let $\M^*$ be a random variable with $\mathcal{L}(\M^*)= \mathcal{L}(\M | 
Y_T>0)$.

From Theorem~\ref{thm:dd}, we know that $\mathrm{MixPo}(\Lambda^*_T)$ is the 
degree distribution at time $T$ in the general case. Theorem~\ref{thm} below 
implies that $\mathrm{MixPo}(\Lambda^*_T)$ converges at a rate of just a bit 
over $e^{-\frac{1}{6}(\lambda-\mu)T}$ to the $\mathrm{MixPo}(\M^*)$ 
distribution, 
which is the asymptotic degree distribution stated in Section 3.2 of \cite{b10} as $T \rightarrow \infty$. Note again that this theorem is much more powerful 
since it gives an exact bound for the total variation distance for finite 
$T$.
\begin{satz}\label{thm}
 Let $\sigma_S<\infty$ be the standard deviation of $S$. Then for $T\geq 
\frac{2\log(4( {\lambda}(\lambda-\mu)^{-1}))}{\lambda-\mu}$, we 
have\fontsize{10.6}{11}
\begin{align*}
&d_{TV}(\mathrm{MixPo}(\Lambda^*_T),\mathrm{MixPo}(\M^*))
 \\
 &\leq \alpha\Biggl( 
\biggl(\frac{5\sqrt{6}}{2}\frac{\lambda}{\lambda-\mu}+\frac{2}{5}
+\biggl(\beta+\frac{\mu}{2}\br \bl \frac{229}{5 
(\lambda-\mu)}+\frac{2}{\lambda+\mu}     \br\biggr) 
\E(S)+\frac{27}{10}\sqrt{2} \sigma_S \Biggr)  (\lambda-\mu)T^2 
e^{-\frac{1}{6}(\lambda-\mu)T} 
\\
 & +  \alpha\beta\Biggl(\biggl(\frac 
1 2+\frac{\mu}{4\beta}\biggr)\frac{\mu}{\lambda}+\biggl(\frac{ 8 } { \mu } 
+\frac { 4 } { \beta} \biggr)\frac { \lambda^3(\lambda+\mu)}{(\lambda-\mu)^3}  
 + \biggl(6T+\frac{\sqrt{6}}{2}+\frac{3}{\beta}(\mu 
T+3)+\frac{5\lambda}{\beta^2}  \biggr)\lambda\Biggr) \E(S) T^2
e^{-(\lambda-\mu)T}\notag
 \\
 &+ 3\sqrt{2}\alpha\lambda \sigma_S T^2
e^{-(\lambda-\mu)T}.
\end{align*}\normalsize
Note that the right-hand side is of the order
$$O(T^2) e^{-\frac{1}{6}(\lambda-\mu)T}$$
as $T\to \infty$.
\end{satz}
\vspace{3mm}
The main idea of the proof of this theorem is the same as in the pure birth 
case: We make use of Theorem \ref{s2} to obtain $\E(|\Lambda_T-\M|\ |Y_T>0)$ 
as an upper bound for 
$d_{TV}(\mathrm{MixPo}(\Lambda^*_T),\mathrm{MixPo}(\M^*))$ and establish a 
further bound for this expected value. In order to do so, we use that the 
expected value of the first summand of the right-hand side of \eqref{lambdag} 
converges quickly to zero since the time $T-T_{\mathfrak{B}_T+\mathfrak{D}_T}$ since the last event 
before $T$ 
converges quickly to 0, and compare the remaining summand of the right-hand 
side of \eqref{lambdag} with the right-hand side of \eqref{mu2}. For this 
comparison, we make vital use of the fact that the average of the social indices 
of the nodes living at time $T$ is close to $\E(S)$ by the Law of Large Numbers 
again and that, given $T_l$ for some large $l\in\en$, the percentage of the nodes 
living at time $T_l$ that survive up to time $T$ is approximately 
$e^{-\mu(T-T_l)}$ by the Law of Large Numbers (see Lemma~\ref{nl1} in the appendix).

A further important ingredient is that the reciprocal of the node process $(Y_t)_{t\geq 0} $
conditioned on survival is a supermartingale (see Lemma 
\ref{supermartingale} in the appendix), which makes it easy to deal with the expected value of 
its maximum. Finally, we also use that the age $A_{J_T}(T)$ of the randomly 
picked individual converges quickly to $A_{J_\infty}$, which was recently proven in \cite{age} (see Corollary \ref{gage} in the appendix for the precise result we apply). 
\\

The fact that nodes may die complicates the procedure considerably since the 
population size after a fixed number of events is random in this case and 
additional dependencies have to be treated (e.g.\ the inter-event times depend 
on the random population size at the previous event time).
\\

In order to cope with additional dependencies on the random index $J_T$, we note that the probability that the node 
that is randomly picked is not older than $\frac{T}{2}$ at time $T$ decreases exponentially in $T$ (see Lemma \ref{ne}(i) 
in the appendix). Thus we essentially only have to consider the time 
interval $[\frac{T}{2},\infty)$ instead of the interval $[T_{r(J_T)},\infty)$, where $T_{r(J_T)}$ is the birth time of the randomly picked individual. The choice of $\frac{T}{2}$ as the left endpoint of the interval makes sure that we always have a large number of individuals in the time interval with high probability. We define the random number $\mathcal{K}(T)$ as index of the last event time before $\frac{T}{2}$ now such that $r(J_T)>\mathcal{K}(T)$ is equivalent to $T_{r(J_T)} \geq \frac{T}{2}$ .%TODO
\begin{notat}\label{def2}
 Let $\mathcal{K}(T):=\max\lbrace k: T_k< \frac{T}{2}\rbrace$, so that we have
${Y_{T_{\mathcal{K}(T)}}=Y_{\frac{T}{2}}}$ almost surely.
\end{notat}

\subsection{Proofs} \label{ssec:boundproofs}

For the proofs of both Theorem~\ref{ss1} and Theorem~\ref{thm} we are interested in the total variation distance between two mixed Poisson distributions, say $\mathrm{MixPo}(\tilde \Lambda)$ and the $\mathrm{MixPo}(\tilde \M)$. Theorem 2.1 in \cite{yan} gives us an upper bound for the total variation distance between two Poisson distributions. By conditioning we can generalize this result to mixed Poisson distributions.
\begin{satz}\label{s2}
Let $\tilde \Lambda$ and $\tilde \M$ be positive real valued random variables. Then we have 
for the total variation distance between the mixed Poisson distributions 
$\mathrm{MixPo}(\tilde \Lambda)$ and $\mathrm{MixPo}(\tilde \M)$:
$$d_{TV}\bigl(\mathrm{MixPo}(\tilde \Lambda),\mathrm{MixPo}(\tilde \M)\bigr)\leq 
\E\bigl(\min\bigl(\bigl|\sqrt{\tilde \Lambda}-\sqrt{\tilde \M}\bigr|,\bigl|\tilde \Lambda-\tilde \M\bigr|\bigr)\bigr).$$
\end{satz}
\begin{proof}
 Let $X$ and $Y$ be $\mathrm{MixPo}(\tilde \Lambda)$ and $\mathrm{MixPo}(\tilde \M)$ 
distributed, respectively. Then it follows
 \allowdisplaybreaks{\begin{align*} 
d_{TV}\bigl(\mathrm{MixPo}(\tilde \Lambda),\mathrm{MixPo}(\tilde \M)\bigr)&=\sup\limits_{A\subset
\en_0} \bigl| \P(X\in A)-\P(Y\in A)\bigr|
  \\ &=  \sup\limits_{A\subset\en_0} \bigl|\E(\P(X\in A|\tilde \Lambda))-\E(\P(Y\in
A|\tilde \M))\bigr| 
  \\ &\leq\sup\limits_{A\subset\en_0} \E \bigl(|\P(X\in A|\tilde \Lambda)-\P(Y\in A|\tilde \M)| \bigr)
  \\ &\leq  \E\biggl(\sup\limits_{A\subset\en_0}|(\P(X\in 
A|\tilde \Lambda))-(\P(Y\in
A|\tilde \M))|\biggr)
  \\ &= \E\bigl(d_{TV}(\mathcal{L}(X|\tilde \Lambda), \mathcal{L}(Y|\tilde \M)\bigr)
  \\ &\leq \E\bigl(\min\bigl(\bigl|\sqrt{\tilde \Lambda}-\sqrt{\tilde \M}\bigr|,\bigl|\tilde \Lambda-\tilde \M\bigr|\bigr)\bigr),
 \end{align*}}where the last line follows from Theorem 2.1 in \cite{yan}.
\end{proof}
\begin{bem}\label{k1}
Note that the mixed Poisson distribution depends on the parameter random variable only via its distribution; therefore Theorem \ref{s2} yields
 $$d_{TV}(\mathrm{MixPo}(\tilde \Lambda),\mathrm{MixPo}(\tilde \M))\leq 
\inf\limits_{\substack{{\hat \Lambda: \hat \Lambda 
\stackrel{\mathcal{D}}{=}\tilde \Lambda}\\{{\hat \M: \hat \M 
\stackrel{\mathcal{D}}{=} \tilde \M}}}}  \E(\min(|\hat \Lambda^{\frac 1 2 } -\hat 
\M^{\frac 1 2 }|,|\hat \Lambda-\hat \M|).$$
\end{bem}

\subsubsection*{Proof of Theorem \ref{ss1}}

From Theorem \ref{s2} follows
 {\allowdisplaybreaks\begin{align*}
  d_{TV}(\mathrm{MixPo}(\Lambda_T),\mathrm{MixPo}(\M)) &\leq   
\E(|\M-\Lambda_T|),\notag
  \end{align*}
  and $\E(|\M-\Lambda_T|)$ is smaller than or equal to
  \begin{align}
  &   \E\biggl| \frac{\alpha}{\beta}\biggl( S_{J_\infty}(1-e^{-\beta  
A_{J_\infty}})+\E(S)(1-e^{-\beta  A_{J_\infty}})-  S_{J_T}(1-e^{-\beta  
A_{\max(J_T,2)}(T)})\notag  \\ 
   &\hphantom{ \E\biggl| \frac{\alpha}{\beta}\biggl( } - 
\sum\limits_{\substack{{i=1}\\i \not= J_T}}^{Y_T} \frac{S_i}{Y_T-1}(1-e^{-\beta 
A_{Y_T}(T)}) -\sum\limits_{\substack{{i=1}\\i \not= J_T}}^{Y_T} 
\sum\limits_{l=i\vee J_T}^{Y_T-1} \frac{S_i}{l-1} (e^{-\beta 
A_{l+1}(T)}-e^{-\beta A_l(T)})\biggr)\biggr|
 \notag
 \\
  &\leq  \frac \alpha \beta  \E\biggl| \sum\limits_{\substack{{i=1}\\i \not= 
J_T}}^{Y_T} \frac{S_i}{Y_T-1}(1-e^{-\beta A_{Y_T}(T)}) \biggr| \notag\\
  & \ \ \  +\frac \alpha \beta \E\biggl| \sum\limits_{\substack{{i=1}\\i \not= 
J_T}}^{Y_T} 
\sum\limits_{l=i}^{Y_T-1} \frac{(S_i-\E(S))} {l-1} \1_{\lbrace J_T\leq 
l\rbrace} 
(e^{-\beta A_{l+1}(T)}-e^{-\beta A_l(T)})\biggr| \notag\\
  & \ \ \  + \frac \alpha \beta \E\biggl|  \E(S) (1-e^{-\beta  A_{J_\infty}})- 
\sum\limits_{\substack{{i=1}\\i \not= J_T}}^{Y_T} \sum\limits_{l=i\vee 
J_T}^{Y_T-1} \frac{\E(S)}{l-1}(e^{-\beta A_{l+1}(T)}-e^{-\beta A_l(T)}))  
\biggr| \notag  \\
   & \ \ \   +\frac \alpha \beta \E\biggl|  S_{J_\infty}(1-e^{-\beta  
A_{J_\infty}})  -  S_{J_T}(1-e^{-\beta  A_{\max(J_T,2)}(T)}) \biggr| ,    
\label{a}
 \end{align}}where we use the convention $\frac{0}{0}:=0$.
 \\
 
For the first line of the right-hand side, we have
\begin{align}
 \E\biggl| \sum\limits_{\substack{{i=1}\\i \not= J_T}}^{Y_T} 
\frac{S_i}{Y_T-1}(1-e^{-\beta A_{Y_T}(T)}) \biggr|
 &=\E\biggl(\E\biggl(\sum\limits_{\substack{{i=1}\\i \not= J_T}}^{Y_T} 
\frac{S_i}{Y_T-1}(1-e^{-\beta A_{Y_T}(T)}) \biggr| Y_T, J_T\biggr)\biggr)\notag
 \\
 &= \E\biggl(\sum\limits_{\substack{{i=1}\\i \not= J_T}}^{Y_T} 
\frac{\E(S_i)}{Y_T-1}(1-\E(e^{-\beta A_{Y_T}(T)}| Y_T)) \biggr), \label{c1}
\end{align}
where the last equality holds since, given $Y_T$, the age $A_{Y_T}(T)$ and 
$J_T$ are independent.

Given $Y_T=y_T>1$, the age $A_{y_T}(T)$ is the minimum of all ages, i.e.\ the minimum of $y_T-1$ i.i.d.\ 
truncated exponentially distributed random variables (see Proposition \ref{age} in the appendix). 
Since the minimum of $y_T-1$ independent $\mathrm{Exp}(\lambda)$ distributed 
random variables is $\mathrm{Exp}((y_T-1) \lambda)$ distributed, the 
distribution of $A_{y_T}(T)$ is stochastically dominated by the 
$\mathrm{Exp}((y_T-1) \lambda)$ distribution. In general, we have for random 
variables $\hat X$ and $\hat Y$ with $\hat X \leq_{st} \hat Y$ that $f(\hat X) 
\leq_{st} f(\hat Y)$ and consequently $\E(f(\hat X)) \leq \E(f(\hat Y))$ for 
every increasing function $f$. Thus the right-hand side of \eqref{c1} is smaller 
than or equal to
\begin{align*}
 \E\biggl(\E(S) 
\frac{1}{\lambda}\frac{\beta}{Y_T-1+\frac{\beta}{\lambda}}\1_{\l 
Y_T>1\r}\biggr)\leq 
\frac{\beta}{\lambda} \E(S) \E\biggl(\frac{1}{Y_T-1} \1_{\l Y_T>1\r}\biggr).
\end{align*}
Note that
\begin{align*}
 &\E\biggl(\frac{1}{Y_T-1} \1_{\l Y_T>1\r}\biggr)  = \sum_{n=2}^{\infty} 
\frac{p_n(T) }{n-1}= \frac{\lambda}{\lambda e^{\lambda T}} \sum_{n=2}^{\infty} 
\frac{1}{n-1} \biggl(\frac{\lambda e^{\lambda T}-\lambda}{\lambda e^{\lambda 
T}}\biggr)^{n-1} \hspace{-1.3mm}\\& = \frac{\lambda}{\lambda e^{\lambda T}} 
\sum_{n=1}^{\infty} 
\frac{1}{n} \biggl(\frac{\lambda e^{\lambda T}-\lambda}{\lambda e^{\lambda 
T}}\biggr)^n \leq \frac{\lambda}{\lambda e^{\lambda T}-\lambda} \sum_{n=1}^{\infty} 
\frac{1}{n} \biggl(\frac{\lambda e^{\lambda T}-\lambda}{\lambda e^{\lambda 
T}}\biggr)^n = \sum_{n=1}^{\infty} \frac{p_n(T) }{n} =\E\bl\frac{1}{Y_T}\br,
\end{align*}
where $p_n(T)$ is the probability mass function of $Y_T$ (see Appendix \hyperref[sec:lbdp]{A1}). Thus the right-hand side of \eqref{c1} is smaller than or equal to
\begin{align*}
 \frac{\beta}{\lambda} \E(S) \E\biggl(\frac{1}{Y_T} \biggr).
\end{align*}

For the second line of the right-hand side of \eqref{a}, we obtain
\allowdisplaybreaks{\begin{align}
 &\E\biggl| \sum\limits_{\substack{{i=1}\\i \not= J_T}}^{Y_T} 
\sum\limits_{l=i}^{Y_T-1} \frac{(S_i-\E(S))} {l-1} \1_{\lbrace J_T\leq 
l\rbrace} 
(e^{-\beta A_{l+1}(T)}-e^{-\beta A_l(T)})\biggr| \notag
 \\
 &=\E\biggl| \sum\limits_{l=2}^{Y_T-1} \frac 1 {l-1} 
\sum\limits_{\substack{{i=1}\\i \not= J_T}}^{l} {(S_i-\E(S))} \1_{\lbrace 
J_T\leq 
l\rbrace} (e^{-\beta A_{l+1}(T)}-e^{-\beta A_l(T)})\biggr| \notag
  \\
 &=\E\biggl(\E\biggl(\biggl| \sum\limits_{l=J_T \vee 2}^{Y_T-1} \frac 1 {l-1}  
\sum\limits_{\substack{{i=1}\\i \not= J_T}}^{l} {(S_i-\E(S))} (e^{-\beta 
A_{l+1}(T)}-e^{-\beta A_l(T)})\biggr|\ \ \biggm| Y_T, J_T\biggr)\biggr) \notag
   \\
 &\leq \E\biggl( \sum\limits_{l=J_T\vee 2}^{Y_T-1} \E\biggl|\frac 1 {l-1}  
\sum\limits_{{i=2}}^{l} {(S_i-\E(S))}\biggr|  \E\biggl(e^{-\beta 
A_{l+1}(T)}(1-e^{-\beta (A_l(T)-A_{l+1}(T))})|Y_T\biggr)\biggr) \notag
     \\
 &\leq \E\biggl( \sum\limits_{l=J_T\vee 2}^{Y_T-1} \frac 1 {{l-1}} \sqrt{  
\sum\limits_{{i=2}}^{l} {\E\biggl((S_i-\E(S))^2\biggr)}}   
\E(1-e^{-\beta (A_l(T)-A_{l+1}(T))}|Y_T)\biggr)\notag
    \\
 &\leq \E\biggl( \sum\limits_{l=J_T\vee 2}^{Y_T-1} \frac 1 {\sqrt{l-1}} 
\sigma_S 
  \E(1-e^{-\beta (A_l(T)-A_{l+1}(T))}|Y_T)\biggr), \label{a1}
\end{align}}where the fourth line holds since the second sum in the third line is (stochastically) independent of 
$J_T$, the social indices are independent of all other random variables and, 
given $Y_T$, the index $J_T$ is independent of the social indices and ages, and 
the second last line is obtained by applying $\E|Z-\E(Z)|\leq \sqrt{\Var(Z)}$ 
to 
$Z=\frac 1 {l-1} \sum_{{i=2}}^{l} S_i$.

Given $Y_T$ and $A_{l+1}(T)$, the difference $A_l(T)-A_{l+1}(T)$ is the minimum of $l-1$ i.i.d.\ truncated exponentially distributed ages (see Proposition \ref{age} in the appendix) for $l\geq 2$. Thus it is stochastically dominated by an $\mathrm{Exp}((l-1) \lambda)$ distributed random variable $Z_l$, where $Z_l$ can be assumed to independent of $A_{l+1}(T)$. This implies for $a>0$:
\begin{align*}
 &\P(A_l(T)-A_{l+1}(T) \leq a | Y_T) 
 \\ &= \int\limits_0^T \P(A_l(T)-A_{l+1}(T) \leq a | Y_T, 
T-A_{l+1}(T) =x) \P(T-A_{l+1}(T)\in dx|Y_T)
 \\
 &\geq  \P(Z_l\leq a|Y_T).
\end{align*}
Thus, also given $Y_T$ alone, the difference $A_l(T)-A_{l+1}(T)$ is stochastically dominated by $Z_l$. By the same argument as for \eqref{c1}, it follows that \eqref{a1} is smaller than or equal to
\begin{align*}
 \E\biggl( \sum\limits_{l=2}^{Y_T-1}\sigma_S \frac{1}{\sqrt{l-1}}   
\frac{\beta}{(l-1) \lambda + \beta}\1_{\lbrace J_T\leq l \rbrace}\biggr)
  &= \E\biggl( \sum\limits_{l=2}^{Y_T-1}\sigma_S \frac{1}{\sqrt{l-1}}   
\frac{\beta}{(l-1)\lambda + \beta} \E(\1_{\lbrace J_T\leq l \rbrace}| 
Y_T)\biggr)
 \\
 &\leq \E\biggl( \sum\limits_{l=2}^{Y_T-1}\sigma_S \frac{1}{\sqrt{l-1}}   
\frac{\beta}{(l-1)\lambda}\frac{l}{Y_T}\biggr)
   \\
    &\leq \E\biggl( \sum\limits_{l=2}^{Y_T-1}\sigma_S \frac{1}{\sqrt{l-1}}   
\frac{\beta}{(l-1)\lambda}\frac{2(l-1)}{Y_T}\biggr)
   \\
 &=\frac{2\beta}{\lambda} \sigma_S \E\biggl( 
\frac{1}{Y_T}\sum\limits_{l=2}^{Y_T-1} \frac 1 {\sqrt{l-1}} \biggr)
 \\
 &\leq \frac{2\beta}{\lambda} \sigma_S \E\biggl(\frac{1}{Y_T} 
\sum\limits_{l=2}^{Y_T-1} \frac 2 {\sqrt{l-2}+\sqrt{l-1}} \biggr)
 \\
 &=  \frac{2\beta}{\lambda} \sigma_S \E\biggl(\frac{1}{Y_T} 
\sum\limits_{l=2}^{Y_T-1}  2 {(\sqrt{l-1}-\sqrt{l-2})} \biggr)
 \\
 &= \frac{4 \beta}{\lambda} \sigma_S \E\biggl(\frac{\sqrt{Y_T-2}}{Y_T}\biggr)
 \\
 &\leq \frac{4 \beta}{\lambda} \sigma_S \E\biggl(\frac{1}{\sqrt{Y_T}}\biggr).
\end{align*}

Since we arranged $A_{J_T}(T) \leq A_{J_\infty}$, for the third line of the 
right-hand side of \eqref{a} we obtain
\begin{align}
 &\E\biggl|  \E(S) (1-e^{-\beta  A_{J_\infty}})- 
\sum\limits_{\substack{{i=1}\\i \not= J_T}}^{Y_T} \sum\limits_{l=i\vee 
J_T}^{Y_T-1} \frac{\E(S)}{l-1}(e^{-\beta A_{l+1}(T)}-e^{-\beta A_l(T)}))  
\biggr|\notag
 \\
  &=\E\biggl|  \E(S) (1-e^{-\beta  A_{J_\infty}})- \E(S) 
\sum\limits_{l=J_T\vee 2}^{Y_T-1}\frac{1}{l-1}\sum\limits_{\substack{{i=1}\\i 
\not= J_T}}^{l}  (e^{-\beta A_{l+1}(T)}-e^{-\beta A_l(T)}))  \biggr|\notag
 \\
 &= \E\biggl|  \E(S) (1-e^{-\beta  A_{J_\infty}})-  \E(S) \sum\limits_{l= 
J_T\vee 2}^{Y_T-1} (e^{-\beta A_{l+1}(T)}-e^{-\beta A_l(T)}))  \biggr|\notag
  \\
 &= \E\biggl|  \E(S) (1-e^{-\beta  A_{J_\infty}})-  \E(S)  (e^{-\beta 
A_{Y_T}(T)}-e^{-\beta A_{\max(J_T,2)}(T)}))  \biggr|\notag
   \\
 &= \E\biggl(  \E(S) (1-e^{-\beta  A_{Y_T}(T)})+  \E(S)  (e^{-\beta 
A_{\max(J_T,2)}(T)}-e^{-\beta A_{J_\infty} })  \biggr) \notag
    \\
 &= \E(S) \E\bigg(  1-\E(e^{-\beta  A_{Y_T}(T)}|Y_T)\biggr)+  \E(S)  
(\E(e^{-\beta 
A_{\max(J_T,2)}(T) })-\E(e^{-\beta A_{J_\infty}})). \label{neu}
\end{align}
Since $A_{\max(J_T,2)}$ has the $\mathrm{Exp}(\lambda)$ distribution truncated at $T$ (see Proposition \ref{age} in the appendix), we have
\begin{align}
 \E(e^{-\beta A_{\max(J_T,2)}(T)})= \int\limits_0^T e^{-\beta x} \frac{\lambda 
e^{-\lambda x}}{1-e^{-\lambda T}}dx= \frac{\lambda }{1-e^{-\lambda T}} \int\limits_0^T 
e^{-(\beta+\lambda) x} dx= \frac{\lambda}{\beta+\lambda} 
\frac{1-e^{-(\beta+\lambda) T} }{1-e^{-\lambda T}},\label{pro}
\end{align}
whence follows that \eqref{neu} is smaller than or equal to 
\begin{align*}
  & \E(S) \E\biggl( \frac{\beta}{\lambda} \frac{1}{Y_t+\frac{\beta}{\lambda}} 
\biggr) + 
\E(S)\frac{\lambda}{\beta+\lambda}\biggl(\frac{1-e^{-(\beta+\lambda)T}}{1-e^{
-\lambda T}}-1 \biggr) \\
&\leq \frac{\beta}{\lambda} \E(S) 
\E\biggl(\frac{1}{Y_T}\biggr)+\E(S)\frac{\lambda}{\beta+\lambda}  
\frac{1}{e^{\lambda T}-1}.
\end{align*}

Since we arranged $S_{J_T} = S_{J_\infty}$ and $A_{J_T}(T) \leq A_{J_\infty}$, 
for the fourth line of the right-hand side of \eqref{a} we obtain
\begin{align}
  \E\biggl|  S_{J_\infty}(1-e^{-\beta  A_{J_\infty}})  -  S_{J_T}(1-&e^{-\beta 
A_{\max(J_T,2)}(T)}) \biggr|  =\E(S) (  \E(e^{-\beta 
A_{\max(J_T,2)}(T)}) - \E(e^{-\beta A_{J_\infty}}) ) .\label{a2}
\end{align}
Note that the right-hand side of \eqref{a2} is the same 
as the second summand of the right-hand side of \eqref{neu}, which is smaller than or equal to 
$\E(S) \frac{\lambda}{\beta+\mu} \frac{1}{e^{\lambda T}-1}$ (see above).
\\

Altogether, we may conclude that the right-hand side of \eqref{a} is bounded from above by
\begin{align*}
 \frac{2\alpha}{\lambda} \E(S) \E\biggl(\frac{1}{Y_T}\biggr)  
+\frac{4\alpha}{\lambda }\sigma_S\E\biggl( \frac 1 {\sqrt{Y_T}} 
\biggr)+\frac{2\alpha}{\beta}\E(S)\frac{\lambda}{\beta+\lambda}\frac{1}{e^{
\lambda T} -1}.
\end{align*}
Using the upper bounds for $\E(Y_T^{-1} )$ and $\E(Y_T^{-1/2})$ from Propositions~\ref{l1b} and~\ref{ko2} in the appendix, we obtain the statement of the 
theorem. \hspace*{\fill}
$\square$

\vspace{4mm}

\subsubsection*{Proof of Theorem~\ref{thm}}

In order to simplify notation, we introduce 
$\E^*(\pkt)=\E(\pkt|Y_T>0)$, $\P^*(\pkt)=\P(\pkt|Y_T>0)$ and the number of events
$\mathcal{M}_T=\mathfrak{B}_T+\mathfrak{D}_T$ up to time $T$.
\\

As before we use Theorem \ref{s2} to obtain
\begin{align*}
&d_{TV}(\mathrm{MixPo}(\Lambda^*_T),\mathrm{MixPo}(\M^*)) \leq 
\E^*|\Lambda_T- \M|.
\end{align*}
In order to find an upper bound for $\E^*|\Lambda_T- \M|$, we use the 
triangle 
inequality for the absolute value after plugging in the definitions of 
$\Lambda_T$ and $\M$ given by \eqref{lambdag} and \eqref{mu2}, which yields
   \allowdisplaybreaks{\begin{align*}
 & \E^*\biggl|\frac{\alpha(1-e^{-\beta(T-  T_{\mathcal{M}_T})})}{(Y_T-1) \beta} 
\1_{\lbrace  Y_T>1 \rbrace} \sum\limits_{\substack{{i=1}\\i \not= J_T}}^{\mathfrak{B}_T} 
(S_i+S_{J_T}) \1_{\lbrace  T_i^- >T \rbrace} \notag
  \\
  &\ \ \ \ \ \ +\frac \alpha \beta \sum\limits_{\substack{{i=1}\\i \not= 
J_T}}^{\mathfrak{B}_T} (S_i+S_{J_T}) \1_{\lbrace  T_i^->T \rbrace} \sum\limits_{l=r(i)\vee 
r(J_T)}^{\mathcal{M}_T-1} (e^{-\beta(T-  T_{l+1})}-e^{-\beta(T-  T_l)})\frac 1 
{Y_{  T_l}-1}\1_{\lbrace  Y_{T_l}>1 \rbrace}
  \\
  &\ \ \ \ \ \ - \frac \alpha {\beta+\mu} (\E(S)+S_{J_\infty}) 
(1-e^{-(\beta+\mu)  A_{J_\infty}})\biggr| .\notag
\end{align*}This is bounded from above by
\begin{align}
&\E^*\bl \frac{\alpha(1-e^{-\beta(T- T_{\mathcal{M}_T})})}{(Y_T-1) \beta} 
\sum\limits_{\substack{{i=1}\\i \not= J_T}}^{\mathfrak{B}_T} S_i \1_{\lbrace T_i^- >T 
\rbrace}\1_{\lbrace  Y_T>1 \rbrace} \br \label{kk4}
\\
&+\E^*\bl \frac{\alpha(1-e^{-\beta(T- T_{\mathcal{M}_T})})}{(Y_T-1) \beta} 
S_{J_T}\sum\limits_{\substack{{i=1}\\i \not= J_T}}^{\mathfrak{B}_T}  \1_{\lbrace T_i^- >T 
\rbrace}\1_{\lbrace  Y_T>1 \rbrace} \br \label{kk4b}
\\
&+\E^*\biggl| \frac \alpha \beta \sum\limits_{\substack{{i=1}\\i \not= 
J_T}}^{\mathfrak{B}_T} 
(S_i-\E(S))  \1_{\lbrace T_i^->T \rbrace}\sum\limits_{l=r(i)\vee 
r(J_T)}^{\mathcal{M}_T-1} (e^{-\beta(T- T_{l+1})}-e^{-\beta(T- T_l)})\frac 1 
{Y_{ T_l}-1} \1_{\lbrace  Y_{T_l}>1 \rbrace}\biggr| \label{kk1}
\\
\begin{split}
&+\E^*\biggl| \frac \alpha \beta \sum\limits_{l=r(J_T)}^{\mathcal{M}_T-1} \frac 
1 
{Y_{ T_l}-1}\1_{\lbrace  Y_{T_l}>1 \rbrace} \sum\limits_{\substack{{i=1}\\i 
\not= 
J_T}}^{r^{-1}(l)}(\E(S)+S_{J_T}) \1_{\lbrace T_i^->T \rbrace}(e^{-\beta(T- 
T_{l+1})}-e^{-\beta(T- T_l)}) \label{kk3}
\\
&\ \ \  \ \ \ \ - \frac \alpha {\beta+\mu} (\E(S)+S_{J_T})  
\sum\limits_{l=r(J_T)}^{\mathcal{M}_T-1} (e^{-(\beta+\mu)(T- 
T_{l+1})}-e^{-(\beta+\mu)(T- T_l)}) \biggr| 
\end{split}
\\
\begin{split}
&+ \E^*\biggl|\frac \alpha {\beta+\mu} (\E(S)+S_{J_T}) 
\sum\limits_{l=r(J_T)}^{\mathcal{M}_T-1} (e^{-(\beta+\mu)(T-  
T_{l+1})}-e^{-(\beta+\mu)(T-  T_l)}) 
\\
&\ \ \ \ \  \ \  - \frac \alpha {\beta+\mu} (\E(S)+S_{J_\infty}) 
(1-e^{-(\beta+\mu)  A_{J_\infty}}) \biggr|, \label{kk7}
\end{split}
\end{align}}where $r^{-1}(l)$ is the number of births that do not occur later than 
the $l$th event 
for all ${l\in \lbrace 1,\ldots, \mathcal{M}_T \rbrace}$, i.e.\ ${r^{-1}: 
\lbrace 
1,\ldots, \mathcal{M}_T \rbrace \rightarrow \lbrace 1,\ldots, \mathfrak{B}_T \rbrace, 
l\mapsto \sum_{i=1}^{\mathfrak{B}_T}  \1_{\lbrace r(i)\leq l\rbrace}}$.

In the following, we deduce upper bounds for \eqref{kk4}--\eqref{kk7}.

\paragraph{Upper bound for (\ref{kk4}) and (\ref{kk4b})} 
 We treat \eqref{kk4} similarly to the corresponding expression in the pure 
birth case, but condition on $\mathfrak{B}_T$, $\mathfrak{D}_T$, $J_T$ and the information which 
nodes survive up to time $T$, and obtain
\begin{align}
 \E^*\biggl( \frac{\alpha(1-e^{-\beta(T- T_{\mathcal{M}_T })})}{(Y_T-1) 
\beta}\1_{\lbrace Y_T>1 \rbrace}\ \sum\limits_{\substack{{i=1}\\i \not= 
J_T}}^{\mathfrak{B}_T} S_i  &\1_{\lbrace T_i^- >T \rbrace}  \biggr)\leq 
\frac{\alpha}{\beta}\E(S) \E^*(1-e^{-\beta(T- T_{\mathcal{M}_T })} 
).\label{t11}
\end{align}
Since we condition on $J_T$, the same expression is also obtained for 
\eqref{kk4b}. 
\\

 For $T\geq \frac{1}{\lambda-\mu}\log(2)$, we have that the expectation 
$\E^*(1-e^{-\beta(T-T_{\mathcal{M}_T})})$ is smaller than or 
equal to
\begin{align}
  \frac{2}{\lambda}\biggl(\lambda-\mu+\beta 
\biggl(\log\biggl(\frac{\lambda}{\lambda-\mu}\biggr) 
+(\lambda-\mu)T\biggr)\biggr)e^{-(\lambda-\mu)T}\label{31b2}
\end{align}
by Lemma \ref{nn1} in the appendix. Plugging this expression in the right-hand side of \eqref{t11} 
results in the following upper bound for the sum of (\ref{kk4}) and 
(\ref{kk4b}):
\begin{align*}
  \frac{4\alpha \E(S)}{\beta\lambda}\biggl(\lambda-\mu+\beta 
\biggl(\log\biggl(\frac{\lambda}{\lambda-\mu}\biggr) 
+(\lambda-\mu)T\biggr)\biggr)e^{-(\lambda-\mu)T}.
\end{align*}

\paragraph{Upper bound for (\ref{kk1})}
Let $R_{ T_l,T}$ be the number of nodes 
that are alive at time $T_l$ and survive up to time $T$. We compute 
\begin{align}
 &\E^*\biggl(\biggl| \frac \alpha \beta \sum\limits_{\substack{{i=1}\\i \not= 
J_T}}^{\mathfrak{B}_T} (S_i-\E(S))  \1_{\lbrace T_i^->T \rbrace}\sum\limits_{l=r(i)\vee 
r(J_T)}^{\mathcal{M}_T -1} (e^{-\beta(T- T_{l+1})}-e^{-\beta(T- T_l)})\frac 1 
{Y_{ T_l}-1}\1_{\lbrace Y_{T_l}>1 \rbrace } \biggr|\biggr)\notag
 \\
 &\leq  \E^*\biggl(\biggl| \frac \alpha \beta 
\sum\limits_{\substack{ l=r(J_T)\\ R_{ T_l,T}\geq 2}}^{\mathcal{M}_T 
-1}(e^{-\beta(T- T_{l+1})}-e^{-\beta(T- 
T_l)})\frac {R_{ T_l,T}-1} {Y_{ T_l}-1}\frac {\1_{\lbrace Y_{T_l}>1 \rbrace }} 
{R_{ T_l,T}-1}\sum\limits_{\substack{{i=1}\\i \not= J_T}}^{r^{-1}(l)}(S_i-\E(S))
\1_{\lbrace T_i^->T \rbrace}   \biggr|\biggr) \notag
  \\
 &\leq  \E^*\Biggl( \frac \alpha \beta \sum\limits_{\substack{ l=r(J_T)\\ R_{ 
T_l,T}\geq 2}}^{\mathcal{M}_T 
-1}(e^{-\beta(T- T_{l+1})}-e^{-\beta(T- T_l)})\frac {R_{ T_l,T}-1} {Y_{ 
T_l}-1}\1_{\lbrace Y_{T_l}>1 \rbrace }\notag
 \\[-2.5mm]
& \hphantom{\leq ^*\Biggl(} \cdot \Biggl(\frac 1 {R_{ T_l,T}-1} \biggl| 
\sum\limits_{\substack{{i=1}\\i \not= J_T}}^{r^{-1}(l)}(S_i-\E(S)) \1_{\lbrace 
T_i^->T \rbrace}\biggr| \ \ \Biggm|(Y_t)_{0\leq t\leq T}, J_T\Biggr)   
\Biggr). \label{kk5}
\end{align}
Note in the second line that we can restrict the sum to $R_{ T_l,T}\geq 2$ because $R_{ T_l,T}\geq 1$ by $l\geq r(J_T)$ and a summand 
with $R_{ T_l,T}= 1$ in the first line would be 0 anyway.

Since the sum in the last line of \eqref{kk5} has exactly $R_{ T_l,T}-1$ 
summands of the form $S_i-\E(S)$ and all other summands are zero, the 
right-hand 
side of \eqref{kk5} is smaller than or equal to
\begin{align}
  &\E^*\biggl(\biggl| \frac \alpha \beta \sum\limits_{\substack{ l=r(J_T)\\ R_{ 
T_l,T}\geq 2}}^{\mathcal{M}_T 
-1}(e^{-\beta(T- T_{l+1})}-e^{-\beta(T- T_l)})\frac {R_{ T_l,T}-1} {Y_{ 
T_l}-1}\1_{\lbrace Y_{T_l}>1 \rbrace }\frac {\sigma_S} {\sqrt{R_{ T_l,T}-1}} 
\biggr| \biggr) .\label{zw}
\end{align}
Since $e^{-\beta(T- T_{l+1})}-e^{-\beta(T- T_l)} \leq \beta (T_{l+1}-T_l)$ and 
$R_{T_l,T} \leq Y_{T_l}$, the expression \eqref{zw} is smaller than or equal to
\begin{align}
   \E^*\biggl( \alpha  \sum\limits_{l=r(J_T)}^{\mathcal{M}_T -1}  \frac 
{\sigma_S \1_{\lbrace Y_{T_l}>1 \rbrace }} {\sqrt{Y_{ T_l}-1}}( T_{l+1}- 
T_l)\biggr)   &\leq \E^*\biggl( \alpha  \sum\limits_{l=r(J_T)}^{\mathcal{M}_T 
-1}  \frac 
{\sigma_S} {\sqrt{Y_{ T_l}-\frac{Y_{T_l}}{2}}} \1_{\lbrace Y_{T_l}>1 \rbrace }( 
T_{l+1}- T_l)\biggr) \notag
  \\
 &\leq\sqrt{2}\alpha   \sigma_S T \E^*\biggl(  \max\limits_{r(J_T) \leq k \leq 
\mathcal{M}_T -1 } \frac 1 {\sqrt{Y_{ T_k}}}\biggr) \notag
   \\
 &\leq \sqrt{2}\alpha   \sigma_S T \E^*\biggl(\1_{\lbrace  {\mathcal{K}(T)} <
{r(J_T)} \rbrace}{  \max\limits_{r(J_T) \leq k \leq 
\mathcal{M}_T -1 } \frac 1 {\sqrt{Y_{ T_k}}}} \biggr) \notag 
  \\
 &\ \ \ +\sqrt{2}{\alpha} \sigma_S T  \E^*\biggl(\1_{\lbrace {\mathcal{K}(T)} 
\geq 
{r(J_T)} \rbrace} {  \max\limits_{r(J_T) \leq k \leq \mathcal{M}_T -1 } \frac 1 
{\sqrt{Y_{ T_k}}}}\biggr), \label{19}
\end{align}
where $\mathcal{K}(T)=\max\lbrace k: T_k\leq \frac{T}{2}\rbrace$ (see 
Definition \ref{def2}). We have
\begin{align}
 &\E^*\biggl(\1_{\lbrace {\mathcal{K}(T)} <
{r(J_T)} \rbrace}   \max\limits_{r(J_T) \leq k \leq \mathcal{M}_T -1 } \frac 1 
{\sqrt{Y_{ T_k}}}\biggr)\notag
 \\
 &\leq \E\biggl(\1_{ {\{\mathcal{K}(T)} < 
{r(J_T)} \rbrace}   \max\limits_{r(J_T) \leq k \leq \mathcal{M}_T -1 } \frac 1 
{\sqrt{Y_{ T_k}}} \ \biggm|Y_\infty >0\biggr) \P^*(Y_\infty >0)+1\cdot 
\P^*(Y_\infty =0),\label{t1}
\end{align}
where $Y_\infty:= \lim\limits_{t \rightarrow \infty} Y_t $. Note that $Y_\infty\in \lbrace 0,\infty\rbrace$ almost surely.

By Lemma \ref{austerben2} in the appendix, we have for the second summand of the right-hand side 
of \eqref{t1}
\begin{align*}
 \P^*(Y_\infty =0) &= \frac{\mu}{\lambda} e^{-(\lambda-\mu)T}.
\end{align*}
The first summand of the right-hand side of \eqref{t1} is smaller than or equal 
to 
\begin{align}
 \E\biggl(   \max\limits_{\mathcal{K}(T) < k } \frac 1 {\sqrt{Y_{ T_k}}} \ 
\biggm|Y_\infty >0\biggr)&=  \E\biggl(   \max\limits_{T/2 < t } \frac 1 {\sqrt{Y_{ t}}} \ 
\biggm|Y_\infty >0\biggr)\notag
\\
&\leq e^{-\frac 1 6 
(\lambda-\mu)T}+\E\biggl(\frac 1 {Y_{ 
T/2}}\biggm|Y_\infty>0\biggr) e^{\frac 1 3  (\lambda-\mu)T}\notag
 \\
&\leq \biggl(1+2\log\biggl(\frac{\lambda}{\lambda-\mu}\biggr)+
(\lambda-\mu) T\biggr) e^{-\frac{1}{6}(\lambda-\mu)T} \label{kk12}
\end{align}
for $T\geq \frac{2}{\lambda-\mu}\log(2)$ by combining Lemma \ref{hl1}, where 
$\delta=1/2$ and $\gamma=1/6$, and Lemma~\ref{n3}. Thus for 
$T\geq \frac{2}{\lambda-\mu}\log(2)$, the first summand of the 
right-hand side of \eqref{19} is smaller than or equal to
\begin{align}
 \sqrt{2}{\alpha} \frac{\mu}{\lambda} \sigma_S  T 
e^{-(\lambda-\mu)T} + \sqrt{2}\alpha \sigma_S 
\biggl(1+2\log\biggl(\frac{\lambda}{\lambda-\mu}\biggr)+(\lambda-\mu)T\biggr) T 
e^{-\frac{1}{6}(\lambda-\mu)T}.\label{31a}
\end{align}

By Lemma \ref{ne}(i), we obtain that for $T\geq \frac{1}{\lambda-\mu}\log(2)$, 
the second summand of the right-hand side of \eqref{19} is bounded from above by
\begin{align}
\sqrt{2}\alpha \sigma_S T \Biggl( e^{-\frac{1}{2}\lambda 
T}+\frac{2(\lambda-\mu)}{\lambda}\biggl(\log\biggl(\frac{\lambda
} {\lambda-\mu} 
\biggr)+(\lambda-\mu)T\biggr)e^{-(\lambda-\mu)T}\Biggr).\label{31b}
\end{align}
   Thus for $T\geq \frac{2}{\lambda-\mu}\log(2)$, the expression \eqref{kk1} 
is bounded from above by the sum of \eqref{31a} and \eqref{31b}.

\paragraph{Upper bound for (\ref{kk3})}
For \eqref{kk3}, we have
\begin{align}
 &\E^*\biggl| \frac \alpha \beta \sum\limits_{l=r(J_T)}^{\mathcal{M}_T -1} 
\frac 
1 {Y_{ T_l}-1} \1_{\lbrace Y_{ T_l}>1 \rbrace} 
\sum\limits_{\substack{i=1\\ i \not= J_T}}^{r^{-1}(l)}(\E(S)+S_{J_T}) 
\1_{\lbrace T_i^->T \rbrace}(e^{-\beta(T- T_{l+1})}-e^{-\beta(T- T_l)}) \notag
\\
&\ \ \ \ \ \ \ - \frac \alpha {\beta+\mu} (\E(S)+S_{J_T}) 
\sum\limits_{l=r(J_T)}^{\mathcal{M}_T -1} (e^{-(\beta+\mu)(T- 
T_{l+1})}-e^{-(\beta+\mu)(T- T_l)}) \biggr|\notag
\\
&=\E^*\biggl|  \frac {\alpha (\beta+\mu)} {\beta (\beta+\mu)} (\E(S)+S_{J_T}) 
\sum\limits_{l=r(J_T)}^{\mathcal{M}_T -1} \frac {R_{ T_l,T}-1} {Y_{ T_l}-1} 
\1_{\lbrace Y_{ T_l}>1 \rbrace} (e^{-\beta(T- T_{l+1})}-e^{-\beta(T- T_l)}) 
\notag
\\
&\ \ \ \ \ \ \ - \frac {\alpha \beta} {\beta (\beta+\mu)}  (\E(S)+S_{J_T}) 
\sum\limits_{l=r(J_T)}^{\mathcal{M}_T -1} (e^{-(\beta+\mu)(T- 
T_{l+1})}-e^{-(\beta+\mu)(T- T_l)}) \biggr|\label{kk5b}
\end{align}
since the second sum on the left-hand side of \eqref{kk5b} has exactly $R_{ 
T_l,T}-1$ non-zero summands. The right-hand side of \eqref{kk5b} is equal to
\begin{align}
&\E^*\biggl|\frac {\alpha } {\beta (\beta+\mu)}  (\E(S)+S_{J_T}) 
\sum\limits_{l=r(J_T)}^{\mathcal{M}_T -1} \biggl((\beta+\mu) \frac {R_{ 
T_l,T}-1} {Y_{ T_l}-1}\1_{\lbrace Y_{ T_l}>1 \rbrace} e^{-\beta(T- T_{l+1})} 
(1-e^{-\beta( T_{l+1} -  T_l)})\notag
\\
&\ \ \ \ \ - \beta e^{-(\beta+\mu) (T-  T_{l+1})}(1-e^{-(\beta+\mu)( T_{l+1}- T_l)})\biggr) \biggr|.\label{tag1}
\end{align}
Now we use $$1-e^{-\beta( T_{l+1} -  T_l)} = \beta( T_{l+1} -  T_l) - 
\sum_{k=2}^\infty  \frac{(-\beta( T_{l+1}-  T_l))^k}{k!}$$ and 
$$1-e^{-(\beta+\mu)( T_{l+1}- T_l)} = (\beta+\mu)( T_{l+1} -  T_l) - 
\sum_{k=2}^\infty \frac{(-(\beta+\mu)( T_{l+1}-  T_l))^k}{k!}$$ to obtain that 
\eqref{tag1} is smaller than or equal to
\begin{align}
& \E^*\biggl|\frac {\alpha (\E(S)+S_{J_T})} {\beta (\beta+\mu)}   
\sum\limits_{l=r(J_T)}^{\mathcal{M}_T -1} (\beta+\mu) \beta \biggl(\frac {R_{ 
T_l,T}-1} {Y_{ T_l}-1}\1_{\lbrace Y_{ T_l}>1 \rbrace}-e^{-\mu(T- T_{l+1})} 
\biggr) e^{-\beta(T- T_{l+1})} ( T_{l+1}- T_l)\biggr| \notag
\\
& +\E^*\biggl|\frac {\alpha } {\beta (\beta+\mu)}  (\E(S)+S_{J_T}) 
\sum\limits_{l=r(J_T)}^{\mathcal{M}_T -1} \biggl( \beta e^{-(\beta+\mu)(T- 
T_{l+1})} \sum_{k=2}^\infty \frac{(-(\beta+\mu)( T_{l+1}-  T_l))^k}{k!} \notag 
\\
&\hphantom{\E^*\biggl|\frac {\alpha } {\beta (\beta}  (\E(S)+S_{J_T}) 
\sum\limits_{l=r(J_T)}^{\mathcal{M}_T -1}  }- (\beta+\mu) \frac {R_{ 
T_l,T}-1} {Y_{ T_l}-1}\1_{\lbrace Y_{ T_l}>1 \rbrace} e^{-\beta(T- T_{l+1})} 
\sum_{k=2}^\infty  \frac{(-\beta( T_{l+1}-  T_l))^k}{k!}  \biggr)\biggr| \notag
\\
&\leq \E^*\biggl(\alpha    (\E(S)+S_{J_T}) 
\sum\limits_{l=r(J_T)}^{\infty} \1_{\lbrace  T_{l+1} \leq T \rbrace}  
\biggl|\frac {R_{ T_l,T}-1} {Y_{ T_l}-1}\1_{\lbrace Y_{ T_l}>1 
\rbrace}-e^{-\mu(T- T_{l+1})} \biggr|  ( T_{l+1}- 
T_l)\biggr) \notag
\\
& \ \ \ +\E^*\biggl(\frac {\alpha } {\beta+\mu}  (\E(S)+S_{J_T}) 
\sum\limits_{l=r(J_T)}^{\mathcal{M}_T -1}  \sum_{k=2}^\infty 
\frac{(-(\beta+\mu)( T_{l+1}-  T_l))^k}{k!}\biggr) \notag
\\
&\ \ \ + \E^*\biggl(\frac {\alpha } {\beta }  (\E(S)+S_{J_T}) 
\sum\limits_{l=r(J_T)}^{\mathcal{M}_T -1}\sum_{k=2}^\infty  \frac{(-\beta( 
T_{l+1}-  T_l))^k}{k!}  \biggr)\biggr). \label{k10}
\end{align}

Firstly, we consider the first summand of the right-hand side of \eqref{k10}. 
Since the social index $S_{J_T}$ is independent of all other random 
variables appearing in \eqref{k10}, this summand is equal to 
$$2\alpha \E(S) \E^*\biggl(
\sum\limits_{l=r(J_T)}^{\infty} \1_{\lbrace  T_{l+1} \leq T \rbrace}  
\biggl|\frac {R_{ T_l,T}-1} {Y_{ T_l}-1}\1_{\lbrace Y_{ T_l}>1 
\rbrace}-e^{-\mu(T- T_{l+1})} \biggr|  ( T_{l+1}- 
T_l)\biggr).$$
We derive
\begin{align}
 &\E^*\biggl( \sum_{l=r(J_T)}^\infty \1_{\lbrace  T_{l+1} < T \rbrace} \biggl| 
\frac {R_{ T_l,T}-1} {Y_{ T_l}-1}\1_{\lbrace Y_{ T_l}>1 \rbrace}-e^{-\mu(T- 
T_{l+1})} \biggr| ( T_{l+1} -  T_l)  \biggr) \notag
 \\
 &\leq \E^*\biggl( \sum_{l=\mathcal{K}(T)+1}^\infty \1_{\lbrace  T_{l+1} < T 
\rbrace} \biggl| \frac {R_{ T_l,T}-1} {Y_{ T_l}-1}\1_{\lbrace Y_{ T_l}>1 
\rbrace}-e^{-\mu(T- T_{l+1})} \biggr| ( T_{l+1} -  T_l)  \biggr) 
\notag
 \\
 &\hphantom{\leq}+ T \P^*(\mathcal{K}(T)\geq r(J_T)). \label{kk11}
\end{align}
For $\frac{1}{\lambda-\mu}\log(2)$, the second summand of the right-hand side 
of \eqref{kk11} is smaller than or 
equal to 
\begin{align}
 T  e^{-\frac{1}{2}\lambda 
T}+\frac{2(\lambda-\mu)}{\lambda}\biggl(\log\biggl(\frac{\lambda
} {\lambda-\mu} 
\biggr)+(\lambda-\mu)T\biggr)T e^{-(\lambda-\mu)T}\notag 
\end{align}
by Lemma \ref{ne}(i). The first summand of the right-hand side of \eqref{kk11} 
is equal to
\begin{align}
&\E^*\biggl( \sum_{l=\mathcal{K}(T)+1}^\infty \E\biggl( \1_{\lbrace  T_{l+1} < 
T 
\rbrace} \biggl| \frac {R_{ T_l,T}-1} {Y_{ T_l}-1}\1_{\lbrace Y_{ T_l}>1 
\rbrace}-e^{-\mu(T- T_{l+1})}\biggr|( T_{l+1} -  T_l)\hspace{1pt} 
\biggm|\mathcal{K}(T), 
Y_T>0\biggr) \biggr). \notag
\end{align}
For the inner expectation, we have
\begin{align}
 &\E\biggl( \1_{\lbrace  T_{l+1} < T 
\rbrace} \biggl| \frac {R_{ T_l,T}-1} {Y_{ T_l}-1}\1_{\lbrace Y_{ T_l}>1 
\rbrace}-e^{-\mu(T- T_{l+1})}\biggr|( T_{l+1} -  T_l)\hspace{1pt} 
\biggm|\mathcal{K}(T), 
Y_T>0\biggr)\notag
\\
&\leq\E\biggl( \1_{\lbrace  T_{l+1} < T 
\rbrace} \biggl| \frac {R_{ T_l,T}-1} {Y_{ T_l}-1}\1_{\lbrace Y_{ T_l}>1 
\rbrace}-e^{-\mu(T- T_{l+1})}\biggr|( T_{l+1} -  T_l)\hspace{1pt} 
\biggm|\mathcal{K}(T), 
Y_{T_l}>0\biggr) \notag \\ &\hphantom{\leq} \cdot \frac{\P(Y_{T_l}>0|\mck(T))}{\P(Y_T>0|\mck(T))}.\label{q}
\end{align}
For the fraction, we obtain for $l\geq \mathcal{K}(T)$
\begin{align}
\frac{\P(Y_{T_l}>0|\mathcal{K}(T))}{\P(Y_T>0|\mathcal{K}(T)) }\leq 
\frac{\P(Y_{T_{\mathcal{K}(T)}}>0|\mathcal{K}(T))}{\E(\P(Y_T>0|\mathcal{K}(T), 
Y_{T_{\mathcal{K}(T)}})|\mathcal{K}(T))}.\label{tag2}
 \end{align}
Note that, given $\mathcal{K}(T)$ and $Y_{T_{\mathcal{K}(T)}}$, we know that 
exactly $Y_{T_{\mathcal{K}(T)}}$ nodes are alive at time~$\frac{T}{2}$. Thus 
due to the Markov property for $(Y_t)_{t\geq0}$ and the formula for the 
extinction probability of a linear birth and death process with a general 
initial value given in Remark \ref{ext}, the conditional probability 
${\P(Y_T=0|\mathcal{K}(T), 
Y_{T_{\mathcal{K}(T)}})}$ is equal to 
$p_0(\frac{T}{2})^{Y_{T_{\mathcal{K}(T)}}}$, where ${p_0(\frac{T}{2})=\P(Y_{\frac{T}{2}}=0)}$. This implies that \eqref{tag2} is 
equal to 
 \begin{align*}
 &\frac{\P(Y_{T_{\mathcal{K}(T)}}>0|\mathcal{K}(T))}{\E(1-p_0(\frac{1}{2} 
T)^{Y_{T_{\mathcal{K}(T)}}}|\mathcal{K}(T))}\hspace{-1pt}\leq \hspace{-1pt}
\frac{\P(Y_{T_{\mathcal{K}(T)}}>0|\mathcal{K}(T))}{(1-p_0(\frac{1}{2} 
T))\P(Y_{T_{\mathcal{K}(T)}}>0|\mathcal{K}(T))}\hspace{-1pt}=\hspace{-1pt}\frac{\lambda 
e^{\frac{1}{2}(\lambda-\mu)T}-\mu}{(\lambda-\mu)e^{\frac{1}{2}(\lambda-\mu)T}}
\hspace{-1pt}\leq\hspace{-1pt}\frac{\lambda}{\lambda-\mu},
\end{align*}
where we applied the formula for the extinction probability of $(Y_t)_{t\geq t}$, which is stated in Appendix \hyperref[sec:lbdp]{A1}.

Thus the right-hand side of \eqref{q} is smaller than or equal to \fontsize{11.4}{11}
\begin{align}
 &\E\biggl(  \E\biggl(\1_{\lbrace  T_{l+1} < T \rbrace}\biggl|\frac {R_{ 
T_l,T}-1} 
{Y_{ T_l}-1}\1_{\lbrace Y_{ T_l}>1 \rbrace}-e^{-\mu(T- T_{l+1})}\biggr|( 
T_{l+1} 
-  T_l) \ \biggm| Y_{T_l}, T_l, 
\mathcal{K}(T)\biggr)\biggm|\mathcal{K}(T),Y_{T_l}>0\biggr)\notag
\\ &\cdot \frac{\lambda}{\lambda-\mu}.\label{tag4}
\end{align}\normalsize
% \begin{align}
%  &\E\biggl(  \E\biggl(\1_{\lbrace  T_{l+1} < T \rbrace}\biggl|\frac {R_{ 
% T_l,T}-1} 
% {Y_{ T_l}-1}\1_{\lbrace Y_{ T_l}>1 \rbrace}-e^{-\mu(T- T_{l+1})}\biggr|( 
% T_{l+1} 
% -  T_l) \ \biggm| Y_{T_l}, T_l, 
% \mathcal{K}(T)\biggr)\biggm|\mathcal{K}(T),Y_{T_l}>0\biggr) \notag
% \\[4mm]
% \raisebox{0mm}[0mm][0mm]{\hspace*{3mm}}\notag
% \\[-11mm]
% &\hspace*{144mm}\cdot \frac{\lambda}{\lambda-\mu}.\notag
% \\[-9mm]
% \, \label{tag4}
% \end{align}
% \vspace*{-8mm}

By the Markov property of $(Y_t)_{t\geq0}$, we may omit the conditioning on 
$\mathcal{K}(T)$ if we condition on $Y_{T_l}$ and $T_l$ for $l\geq 
\mathcal{K}(T)+1$. Conditioning in addition on $T_{l+1}$, we see that \eqref{tag4} is equal to
 \fontsize{11.5}{11}\begin{align}
& \E\biggl(\1_{\lbrace  T_{l+1} < T \rbrace} ( T_{l+1} 
-  T_l) \E\biggl(\biggl|\frac {R_{ T_l,T}-1} 
{Y_{ T_l}-1}\1_{\lbrace Y_{ T_l}>1 \rbrace}-e^{-\mu(T- T_{l+1})}\biggr| \ 
\biggm| 
Y_{T_l}, T_l, T_{l+1}\biggr)\biggm|\mathcal{K}(T),Y_{T_l}>0\biggr) \notag \\& \cdot \frac{\lambda}{\lambda-\mu}. \label{tag4a}
\end{align}\normalsize
% \begin{align}
% & \E\biggl(\1_{\lbrace  T_{l+1} < T \rbrace} ( T_{l+1} 
% -  T_l) \E\biggl(\biggl|\frac {R_{ T_l,T}-1} 
% {Y_{ T_l}-1}\1_{\lbrace Y_{ T_l}>1 \rbrace}-e^{-\mu(T- T_{l+1})}\biggr| \ 
% \biggm| 
% Y_{T_l}, T_l, T_{l+1}\biggr)\biggm|\mathcal{K}(T),Y_{T_l}>0\biggr) \notag
% \\[4mm]
% \raisebox{0mm}[0mm][0mm]{\hspace*{3mm}}\notag
% \\[-11mm]
% &\hspace*{144mm}\cdot \frac{\lambda}{\lambda-\mu}.\notag
% \\[-9mm]
% \, \label{tag4a}
% \end{align}
% \vspace*{-8mm}

By Lemma \ref{nl1}, summing over $l$ and taking the 
expectation yields the following upper bound for the first summand of the 
right-hand side of \eqref{kk11}:
\begin{align}
 &\E^*\biggl( 
\sum\limits_{l=\mathcal{K}(T)+1}^\infty\sqrt{6}\frac{\lambda}{\lambda-\mu}
\E\biggl( 
\1_{\lbrace T_{l+1}< T \rbrace} (T_{l+1}-T_l ) Y_{T_l}^{-\frac{1}{2}} 
\biggm|\mathcal{K}(T), Y_{T_l} >0\biggr) \biggr).  \label{t02}
  \end{align}
Obviously, \eqref{t02} is equal to
  \begin{align}
&\sqrt{6}\frac{\lambda}{\lambda-\mu}\E^*\biggl(\1_{\lbrace Y_\infty>0 
\rbrace}\sum\limits_{l=\mathcal{K}(T)+1}^\infty \E\biggl(\1_{\lbrace T_{l+1}< T 
\rbrace} (T_{l+1}-T_l ) Y_{T_l}^{-\frac{1}{2}} \biggm|\mathcal{K}(T), Y_{T_l} 
>0\biggr)\biggr) \notag
 \\
 &+\sqrt{6}\frac{\lambda}{\lambda-\mu}\E^*\biggl(\1_{\lbrace Y_\infty=0 
\rbrace}\sum\limits_{l=\mathcal{K}(T)+1}^\infty \E\biggl(\1_{\lbrace T_{l+1}< T 
\rbrace} (T_{l+1}-T_l ) Y_{T_l}^{-\frac{1}{2}} \biggm|\mathcal{K}(T), Y_{T_l} 
>0\biggr)\biggr).\notag
  \end{align}
We may conclude that for $T\geq \frac{1}{\lambda-\mu}\log(2)$, the first 
summand of \eqref{k10} is bounded from above by
  \begin{align}
2\alpha \E(S) \Biggl(&\sqrt{6}\frac{\lambda}{\lambda-\mu}\E^*\biggl(\1_{\lbrace 
Y_\infty>0 
\rbrace}\sum\limits_{l=\mathcal{K}(T)+1}^\infty \E\biggl(\1_{\lbrace T_{l+1}< T 
\rbrace} (T_{l+1}-T_l ) Y_{T_l}^{-\frac{1}{2}} \biggm|\mathcal{K}(T), Y_{T_l} 
>0\biggr)\biggr) \notag
 \\
 &+\sqrt{6}\frac{\lambda}{\lambda-\mu}\E^*\biggl(\1_{\lbrace Y_\infty=0 
\rbrace}\sum\limits_{l=\mathcal{K}(T)+1}^\infty \E\biggl(\1_{\lbrace T_{l+1}< T 
\rbrace} (T_{l+1}-T_l ) Y_{T_l}^{-\frac{1}{2}} \biggm|\mathcal{K}(T), Y_{T_l} 
>0\biggr)\biggr)\notag
\\
&+ T  e^{-\frac{1}{2}\lambda 
T}+\frac{2(\lambda-\mu)}{\lambda}\biggl(\log\biggl(\frac{\lambda
} {\lambda-\mu} 
\biggr)+(\lambda-\mu)T\biggr)T e^{-(\lambda-\mu)T}\Biggr).\label{n2} 
  \end{align}

 For the first outer expectation in \eqref{n2}, we have
 \begin{align*}
&\E^*\biggl(\1_{\lbrace Y_\infty>0 
\rbrace}\sum\limits_{l=\mathcal{K}(T)+1}^\infty 
\E\biggl(\1_{\lbrace T_{l+1}< T \rbrace} (T_{l+1}-T_l ) Y_{T_l}^{-\frac{1}{2}} 
\biggm|\mathcal{K}(T), Y_{T_l} >0\biggr)\biggr) 
 \\
 &\leq \ \  \E\biggl(\sum\limits_{l=\mathcal{K}(T)+1}^\infty 
\E\biggl(\1_{\lbrace 
T_{l+1}< T \rbrace} (T_{l+1}-T_l ) Y_{T_l}^{-\frac{1}{2}} \biggm|\mathcal{K}(T) 
\biggr) \biggm| Y_\infty >0\biggr) 
 \\
 &\leq \ \  \E\biggl( \max\limits_ { \mathcal{K}(T) \leq k}   Y_{ T_k}^{-\frac 
1 
2} \sum\limits_{l=\mathcal{K}(T)+1}^\infty \1_{\lbrace T_{l+1}< T \rbrace} 
(T_{l+1}-T_l ) \biggm| Y_\infty>0 \biggr)   
  \\
 &\leq \ \  \frac T 2 \E\biggl( \max\limits_ { \mathcal{K}(T)\leq k }   Y_{ 
T_k}^{-\frac 1 2} \biggm| Y_\infty>0 \biggr) .
\end{align*}
By inequality \eqref{kk12}, we obtain that for $T\geq 
\frac{2}{\lambda-\mu}\log(2)$, this expression is smaller than or equal to
\begin{align*}
\frac T 2 \biggl(1+2\log\biggl(\frac{\lambda}{\lambda-\mu}\biggr)+(\lambda-\mu) T \biggr) e^{-\frac{1}{6}(\lambda-\mu)T}.
\end{align*}

For the second outer expectation in \eqref{n2}, we obtain
\begin{align*}
 &\E\biggl(\1_{\lbrace Y_\infty=0 
\rbrace}\sum\limits_{l=\mathcal{K}(T)+1}^\infty 
\E\biggl(\1_{\lbrace T_{l+1}< T \rbrace} (T_{l+1}-T_l ) Y_{T_l}^{-\frac{1}{2}} 
\biggm|\mathcal{K}(T), Y_{T_l} >0\biggr)\biggm|Y_T>0\biggr)
 \\
 &\leq \frac T 2 \P(Y_\infty=0 |Y_T>0) = \frac T 2 \frac{\mu}{\lambda 
}e^{-(\lambda-\mu)T},
\end{align*}
where the last equality follows from Lemma \ref{austerben2}.
\\

In conclusion, for $T\geq \frac{2}{\lambda-\mu}\log(2)$, the first 
summand of the right-hand 
side of \eqref{k10} is bounded from above by \vspace{-2mm}
\begin{align*}
     &\sqrt{6}\alpha\frac{\lambda}{\lambda-\mu} \E(S) T 
\Biggl(\biggl(1+2\log\biggl(\frac{\lambda}{\lambda-\mu}\biggr)+(\lambda-\mu) T 
\biggr) e^{-\frac{1}{6}(\lambda-\mu)T}+  e^{-(\lambda-\mu)T} \Biggr)  
\\
&+2\alpha \E(S)\Biggl( Te^{-\frac{1}{2}\lambda 
T}+\frac{2(\lambda-\mu)}{\lambda}\biggl(\log\biggl(\frac{\lambda
} {\lambda-\mu} 
\biggr)+(\lambda-\mu)T\biggr)T e^{-(\lambda-\mu)T}\Biggr). 
\end{align*}

Since the social index $S_{J_T}$ is independent of all other random 
variables appearing in \eqref{k10}, the second summand of the right-hand side 
of \eqref{k10} is smaller than or equal to
\begin{align}
 &\frac{2\alpha}{\beta+\mu}\E(S) \E^*\biggl(\sum_{l=r(J_T)}^{\mathcal{M}_T -1} 
\sum_{k=2}^\infty \frac{(-(\beta+\mu)( T_{l+1}- T_l))^k}{k!}\biggr).\notag
\end{align}
Since $\sum_{k=2}^\infty \frac{(-x)^k}{k!}\leq \frac{x^2}{2}$ for $x\geq 0$, this expression is bounded from above by
\begin{align}
 & \frac{\alpha}{\beta+\mu}\E(S) \E^*\biggl(\sum_{l=r(J_T)}^{\mathcal{M}_T 
-1} (\beta+\mu)^2 ( T_{l+1}-  T_l)^2\biggr) \biggr).        \label{la1}
 \end{align}
For ${T\geq(\frac{2}{\lambda-\mu}\log(2)\vee\frac{2(\log(4 
\lambda)-\log(\lambda-\mu))}{\lambda+\mu}})$, Lemma \ref{nl} implies that 
\eqref{la1} and hence also the second summand of the right-hand side of \eqref{k10} is 
smaller than or equal to
\begin{align*}
\alpha(\beta+\mu)  \E(S)  \Biggl( &\frac{\mu}{\lambda}\frac{T^2}{4} 
e^{-(\lambda-\mu)T} +
\frac{60 \lambda^3(\lambda+\mu)}{(\lambda-\mu)^4} 
T^2 e^{-(\lambda+\mu)T}+ T^2 e^{-\frac{1}{2}\lambda 
T}
\\&+\frac{2(\lambda-\mu)}{\lambda}\biggl(\log\biggl(\frac{\lambda
} {\lambda-\mu} \biggr)+(\lambda-\mu)T\biggr)T^2e^{-(\lambda-\mu)T} \notag
 \\
 &+  \biggl(\frac 3 4 T^2+\frac{T}{2(\lambda+\mu)}\biggr) \biggl( 
1+2\log\biggl(\frac{\lambda}{\lambda-\mu}\biggr)+ (\lambda-\mu) T\biggr) 
e^{-\frac{1}{4} (\lambda-\mu)T}\Biggr).\end{align*}

For the third summand of the right-hand side of \eqref{k10}, we obviously obtain 
the same upper bound except that the factor $\beta+\mu$ is replaced by $\beta$.
\\

We may conclude that for ${T\geq(\frac{2}{\lambda-\mu}\log(2)\vee\frac{2(\log(4 
\lambda)-\log(\lambda-\mu))}{\lambda+\mu}})$, the expression 
\eqref{kk3} is smaller than or equal to
\begin{align*}
& \sqrt{6}\alpha\frac{\lambda}{\lambda-\mu} \E(S) T 
\Biggl(\biggl(1+2\log\biggl(\frac{\lambda}{\lambda-\mu}\biggr)+(\lambda-\mu) T 
\biggr) e^{-\frac{1}{6}(\lambda-\mu)T}+  e^{-(\lambda-\mu)T} \Biggr)  
\\
&+ 2\alpha \E(S)\Biggl(T  e^{-\frac{1}{2}\lambda 
T}+\frac{2(\lambda-\mu)}{\lambda}\biggl(\log\biggl(\frac{\lambda
} {\lambda-\mu} 
\biggr)+(\lambda-\mu)T\biggr)T e^{-(\lambda-\mu)T}\Biggr)
\\
&+   \alpha(2\beta+\mu)  \E(S)  \Biggl( \frac{\mu}{\lambda}\frac{T^2}{4} 
e^{-(\lambda-\mu)T} + 
\frac{60 \lambda^3(\lambda+\mu)}{(\lambda-\mu)^4} 
T^2e^{-(\lambda+\mu)T}+ T^2 e^{-\frac{1}{2}\lambda 
T}
\\&\hphantom{+   \alpha(2\beta+\mu)  \E(S)  
\Biggl(}+\frac{2(\lambda-\mu)}{\lambda}\biggl(\log\biggl(\frac{\lambda
} {\lambda-\mu} \biggr)+(\lambda-\mu)T\biggr)T^2e^{-(\lambda-\mu)T} \notag
 \\
 &\hphantom{+   \alpha(2\beta+\mu)  \E(S)  
\Biggl(}+  \biggl(\frac 3 4 T^2+\frac{T}{2(\lambda+\mu)}\biggr) \biggl( 
1+2\log\biggl(\frac{\lambda}{\lambda-\mu}\biggr)+ (\lambda-\mu) T\biggr) 
e^{-\frac{1}{4} (\lambda-\mu)T}\Biggr).
\end{align*}

\paragraph{Upper bound for (\ref{kk7})}
Recall that we arranged $S_{J_T}=S_{J_\infty}$. Since the sum in \eqref{kk7} telescopes, we obtain
\begin{align}
  &\E^* \biggl(\biggl|\frac \alpha {\beta+\mu} (\E(S)+S_{J_T}) 
\sum\limits_{l=r(J_T)}^{\mathcal{M}_T -1} (e^{-(\beta+\mu)(T-  
T_{l+1})}-e^{-(\beta+\mu)(T-  T_l)}) \notag
 \\[-1mm]
 &\hphantom{\E^* \biggl(\biggl|}- \frac \alpha {\beta+\mu} (\E(S)+S_{J_\infty}) 
(1-e^{-(\beta+\mu)  A_{J_\infty}}) \biggr| \biggr) \notag
 \\[1.5mm]
 &=  \E^* \biggl(\biggl|\frac \alpha {\beta+\mu} 
(\E(S)+S_{J_\infty})\biggl(e^{-(\beta+\mu)(T-  
T_{\mathcal{M}_T})}-e^{-(\beta+\mu)(T-  T_{r(J_T)})} -  (1-e^{-(\beta+\mu)  
A_{J_\infty}})\biggr) \biggr| \biggr) \notag
 \\[0mm]
 & \leq \frac {2\alpha} {\beta+\mu} \E(S) \biggl(\E^*\biggl(1- 
e^{-(\beta+\mu)(T-T_{\mathcal{M}_T })}  \biggr)   +  
\E^*\biggl(\Bigl|e^{-(\beta+\mu)A_{J_T}(T)}-e^{-(\beta+\mu) 
A_{J_\infty}}\Bigr| 
\biggr)  \biggr), \label{t12}
\end{align}
where the last inequality holds since $S_{J_\infty}$ is independent of all other 
random variables appearing in \eqref{t12}.

By using Lemma \ref{nn1} again, the first 
conditional expectation on the right-hand side 
of \eqref{t12} can be bounded from above by 
\begin{align*}
  \frac{2}{\lambda}\biggl(\lambda-\mu+(\beta+\mu) 
\biggl(\log\biggl(\frac{\lambda}{\lambda-\mu}\biggr) 
+(\lambda-\mu)T\biggr)\biggr)e^{-(\lambda-\mu)T}
\end{align*}
if $T\geq \frac{1}{\lambda-\mu}\log(2)$.

In order to find an upper bound for the second conditional expectation on the right-hand side of \eqref{t12}, we apply a recent result from \cite{age} about the age distribution in a linear birth and death process. This result is stated in Corollary \ref{gage} in the appendix and implies that this conditional expectation is bounded from above by
\begin{align*}
  \Biggl(\frac{2\lambda}{\beta+\mu+\lambda} 
 + \frac{2(\lambda-\mu)}{\lambda} 
\biggl(\log\biggl(\frac{\lambda}{\lambda-\mu}
\biggr)+(\lambda-\mu)T\biggr)\Biggr)e^ {
-(\lambda-\mu)T}
\end{align*}
if $T\geq \frac{1}{\lambda-\mu}\log(2)$.

Altogether, we obtain that the right-hand side of \eqref{t12} is smaller than or 
equal to
\begin{align*}
 & \frac{4\alpha}{\beta+\mu} \E(S) \Biggl( \frac{\lambda-\mu}{\lambda}+ 
\frac{\beta+\lambda}{\lambda} 
\biggl(\log\biggl(\frac{\lambda}{\lambda-\mu}\biggr)+(\lambda-\mu)T\biggr)+ 
\frac{\lambda}{\beta+\mu+\lambda} \Biggr)e^{-(\lambda-\mu) T}
\end{align*}
if $T\geq \frac{1}{\lambda-\mu}\log(2)$.
\\

\paragraph{Conclusion}
Combining the upper bounds obtained for \eqref{kk4}--\eqref{kk7}, we have 
\begin{align}
&\!d_{TV}(\mathrm{MixPo}(\Lambda^*_T),\mathrm{MixPo}(\M^*)) \notag\\[1mm]
&\leq \frac{4\alpha \E(S)}{\beta \lambda}\biggl(\lambda-\mu+\beta 
\biggl(\log\biggl(\frac{\lambda}{\lambda-\mu}\biggr) 
+(\lambda-\mu)T\biggr)\biggr)e^{-(\lambda-\mu)T}\notag
  \\
  &+ \sqrt{2}{\alpha} \frac{\mu}{\lambda} \sigma_S  T e^{-(\lambda-\mu)T} + 
\sqrt{2}\alpha \sigma_S 
\biggl(1+2\log\biggl(\frac{\lambda}{\lambda-\mu}\biggr)+(\lambda-\mu)T\biggr) T 
e^{-\frac{1}{6}(\lambda-\mu)T}\notag
\\
&+\sqrt{2}\alpha \sigma_S T \Biggl( e^{-\frac{1}{2}\lambda 
T}+\frac{2(\lambda-\mu)}{\lambda}\biggl(\log\biggl(\frac{\lambda
} {\lambda-\mu} 
\biggr)+(\lambda-\mu)T\biggr)e^{-(\lambda-\mu)T}\Biggr)\notag
\\
&+ \sqrt{6}\alpha\frac{\lambda}{\lambda-\mu} \E(S) T 
\Biggl(\biggl(1+2\log\biggl(\frac{\lambda}{\lambda-\mu}\biggr)+(\lambda-\mu) T 
\biggr) e^{-\frac{1}{6}(\lambda-\mu)T}+  e^{-(\lambda-\mu)T} \Biggr)  \notag
\\
&+ 2\alpha \E(S)\Biggl(T  e^{-\frac{1}{2}\lambda 
T}+\frac{2(\lambda-\mu)}{\lambda}\biggl(\log\biggl(\frac{\lambda
} {\lambda-\mu} 
\biggr)+(\lambda-\mu)T\biggr)T e^{-(\lambda-\mu)T}\Biggr)\notag
 \\
 &+ \alpha(2\beta+\mu )  \E(S)  \Biggl( \biggl(\frac 3 4 
T^2+\frac{T}{2(\lambda+\mu)}\biggr)  \notag
\\&\hphantom{+ \alpha(2\beta+\mu )  \E(S)  \Biggl(}  \cdot \biggl(1+2\log\biggl(\frac{\lambda}{\lambda-\mu}\biggr)+ (\lambda-\mu) T\biggr) 
e^{-\frac{1}{4} (\lambda-\mu)T} + \frac{T^2}{4}\frac{\mu}{\lambda} 
e^{-(\lambda-\mu)T} \Biggr)\notag
\\
&+\alpha(2\beta+\mu )  
\E(S) T^2\Biggl(\frac{2(\lambda-\mu)}{\lambda}\biggl(\log\biggl(\frac{\lambda
} {\lambda-\mu} \biggr)+(\lambda-\mu)T\biggr)e^{-(\lambda-\mu)T} \notag
\\ &\hphantom{+\alpha(2\beta+\mu )  
\E(S) T^2\Biggl(}+  
\frac{60\lambda^3(\lambda+\mu)}{(\lambda-\mu)^4} e^{-(\lambda+\mu)T}\notag+  
e^{-\frac{1}{2}\lambda T}\Biggr)\notag
 \\
 &+ \frac{4\alpha}{\beta+\mu} \E(S) \Biggl( \frac{\lambda-\mu}{\lambda}+ 
\frac{\beta+\lambda}{\lambda} 
\biggl(\log\biggl(\frac{\lambda}{\lambda-\mu}\biggr)+(\lambda-\mu)T\biggr)+ 
\frac{\lambda}{\beta+\mu+\lambda} \Biggr)e^{-(\lambda-\mu) T}. \label{final}
\end{align}
for ${T\geq(\frac{2}{\lambda-\mu}\log(2)\vee\frac{2(\log(4 
\lambda)-\log(\lambda-\mu))}{\lambda+\mu}})$.
The upper bound claimed is now obtained by elementary computations.
$\hfill \square$

\renewcommand{\thesatz}{A.\arabic{satz}}
\renewcommand{\theequation}{A.\arabic{equation}}
\setcounter{satz}{0}
\setcounter{equation}{0}

% \section*{Appendix A1: Upper bound for the total variation distance between two mixed Poisson distributions}\label{A1}

% 
% \section*{Appendix A2: The age distribution at finite time $T$}
% \label{A2}

\section*{Appendix A1: Linear birth and death processes and the age of a randomly picked 
individual}\label{sec:lbdp}
The node process $(Y_t)_{t\geq0}$ is a 
linear birth and death process with birth rate $\lambda$, death rate $\mu < 
\lambda$ and 
initial value one, i.e.\ $Y_0=1$. According to (8.15) and (8.46) in \cite{bai}, 
the 
one-dimensional distributions of such a process are given by the following 
probability mass functions:
\begin{align*}
 &p_0(t)=\mu \tilde p(t)
 \\&p_n(t)=(1-\mu \tilde p(t))(1-\lambda \tilde p(t))(\lambda \tilde 
p(t))^{n-1}, \ n \geq1,
\end{align*}
where $$\tilde p(t):=\frac{e^{(\lambda-\mu)t}-1}{\lambda e^{(\lambda-\mu)t}-\mu} 
=\frac{1}{\lambda} 
\frac{1-e^{-(\lambda-\mu)t}}{1-\frac{\mu}{\lambda}e^{-(\lambda-\mu)t}}.$$
\begin{bem}\label{ext}
Note that $p_0(t)$ is the probability that a linear birth and death process with 
initial value one goes extinct up to time $t$. Due to the branching property of 
a linear birth and death process, we have that $p_0(t)^m$ is the probability 
that a linear birth and death process with a general initial value $m$ goes 
extinct up to time $t$. By taking the limit $t \rightarrow \infty$, we obtain 
that the probability of eventual extinction is $(\mu/\lambda)^m$ if $\lambda>\mu$ (see e.g.\ 
(8.59) in \cite{bai}).
\end{bem}
By elementary computations using these probability mass functions, we obtain 
the 
following proposition (cf. (8.16), (8.17), (8.48) and (8.49) in \cite{bai}).
\begin{proposition}\label{mom}
 We have $$\E(Y_t)=e^{(\lambda-\mu)t} \text{ and } 
\Var(Y_t)=\frac{\lambda+\mu}{\lambda-\mu}(e^{2(\lambda-\mu)t}-e^{(\lambda-\mu)t}
). $$
\end{proposition}
From \cite{age} we know furthermore an expression and an upper bound for the conditional expectation of $1/Y_t$ given $Y_t > 0$ that is essentially of the anticipated order $e^{-(\lambda-\mu)t}$ as $t \to \infty$.
\begin{proposition}[{\protect \citealp[Lemma 3.1]{age}}]\label{l1b}
For any $t>0$, we have
 \begin{align*}E\biggl(\frac{1}{Y_t}\biggm|Y_t>0\biggr)&= \frac{\lambda-\mu}{\lambda e^{(\lambda-\mu)t} - \lambda} \log \biggl(\frac{\lambda e^{(\lambda-\mu)t} - \mu}{\lambda-\mu} \biggr) \\
 &\leq \frac{\lambda-\mu}{\lambda e^{(\lambda-\mu)t}-\lambda}\biggl(\log\biggl(\frac{\lambda}{\lambda-\mu}\biggr)+(\lambda-\mu)t\biggr).\end{align*}
\end{proposition} 
We may apply this bound in order to obtain an upper bound for $\E(Y_t^{-1/2} \mvert Y_t > 0)$. 
\begin{proposition}\label{ko2}
For $\lambda>\mu$ and $t\geq \frac{1}{\lambda-\mu}\log(2)$, we have
 $$\E\biggl(\frac{1}{\sqrt{Y_t}} \biggm|Y_t>0\biggr)\leq 
e^{-\frac{1}{2}(\lambda-\mu)t} 
\sqrt{\frac{2(\lambda-\mu)}{\lambda}\biggl(\log\biggl(\frac{\lambda}{\lambda-\mu
}\biggr)+(\lambda-\mu)t\biggr)} .$$
\end{proposition}
\begin{bew}
For $\lambda>\mu$ and $t\geq \frac{1}{\lambda-\mu}\log(2)$, we have
\begin{align*}
 \E\biggl(\frac{1}{\sqrt{Y_t}} \biggm|Y_t>0\biggr)&\leq 
\sqrt{\E\biggl(\frac{1}{{Y_t}}\biggm|Y_t>0\biggr)}
 \\
 &\leq \sqrt{\frac{\lambda-\mu}{\lambda e^{(\lambda-\mu)t} - \lambda} \biggl( \log\biggl(\frac{\lambda}{\lambda-\mu}\biggr)+(\lambda-\mu )t \biggr) }
 \\
 &\leq \sqrt{\frac{2(\lambda-\mu)}{\lambda e^{(\lambda-\mu)t}}\biggl(\log\biggl(\frac{\lambda}{\lambda-\mu}\biggr)+(\lambda-\mu)t \biggr)}
 \\
 &=e^{-\frac{1}{2}(\lambda-\mu)t} \sqrt{\frac{2(\lambda-\mu)}{\lambda}\biggl(\log\biggl(\frac{\lambda}{\lambda-\mu}\biggr)+(\lambda-\mu)t\biggr)} ,
\end{align*}
where the second line follows from Lemma 3.1 in \cite{age}.
\end{bew}

For the proof of our main theorem a finer analysis is required. Denote by $\mathfrak{B}_t$ 
and $\mathfrak{D}_t$ the numbers of births and deaths up to time $t$, respectively, where 
we set $\mathfrak{B}_0 = Y_0 = 1$ and $\mathfrak{D}_0 = 0$. By using a partial differential equation 
for the joint cumulant generating function of $\mathfrak{B}_t$ and $Y_t$ stated in 
\cite{cox}, we obtain the following formulae for the first and second joint 
moments.
\begin{proposition}[] \label{pr1}
We have
 \begin{align}
 &\E(\mathfrak{B}_t)=\frac{\lambda}{\lambda-\mu}e^{(\lambda-\mu)t}- 
\frac{\mu}{\lambda-\mu} \notag
 \\
&\Cov(\mathfrak{B}_t,Y_t)=\frac{\lambda(\lambda+\mu)}{(\lambda-\mu)^2}e^{2(\lambda-\mu)t}- 
\frac{2\lambda\mu}{\lambda-\mu} t 
e^{(\lambda-\mu)t}-\frac{\lambda^2}{(\lambda-\mu)^2}e^{(\lambda-\mu)t}\notag
 \\
 &\Var(\mathfrak{B}_t)= \frac{\lambda^2(\lambda+\mu)}{(\lambda-\mu)^3}e^{2(\lambda-\mu)t} 
-\frac{ 4\lambda^2\mu}{(\lambda-\mu)^2}t 
e^{(\lambda-\mu)t}+\biggl(\frac{2\lambda^2\mu}{(\lambda-\mu)^3}-\frac{
\lambda(\lambda+\mu)}{(\lambda-\mu)^2}\biggr) e^{(\lambda-\mu)t}. \notag
 \end{align}
\end{proposition}
\begin{bem}
 Note that $\E(Y_t)=\E(\mathfrak{B}_t-\mathfrak{D}_t)=\E(\mathfrak{B}_t)-\E(\mathfrak{D}_t)$ and 
${\frac{\E(\mathfrak{B}_t)-1}{\E(\mathfrak{D}_t)}=\frac{\lambda}{\mu}=\frac{\lambda}{\lambda+\mu}(\frac
{ \mu}{\lambda+\mu})^{-1}}$ is the ratio of the probabilities of a birth and a 
death at each event time. Furthermore, the sum
$\E(\mathfrak{B}_t)+\E(\mathfrak{D}_t)=\frac{\lambda+\mu}{\lambda-\mu}e^{(\lambda-\mu)T}-\frac{2 
\mu}{\lambda-\mu} $ is the expected number of events up to time $t$.
\end{bem}

In the rest of this section, we summarize the results about the age of an 
individual picked uniformly at random at a fixed time $T>0$ (given $Y_T > 0$). We 
briefly call the distribution of this age \emph{the age distribution of $(Y_t)_{t\geq0}$ 
at time $T$}. In 
the pure birth case, the age distribution has a simple form.
\begin{proposition}[Neuts and Resnick {\protect \cite[Theorem 1]{neu}}]\label{age}
Let $\mu=0$. The ages of the individuals at time $T$ that have been born after 
time zero are i.i.d.\ truncated exponentially distributed, more precisely they 
have distribution $\mathcal{L}(Z|Z \leq T)$, where $Z \sim 
\mathrm{Exp}(\lambda)$.
\end{proposition}

In the general case, we first state the conditional age distribution given the population size, which we know from \cite{age}. 
\begin{satz}[{\protect \citealp[Theorem~2.1]{age}}]\label{gagen}
Let $F_{y_T}$ denote the cumulative distribution function of the age of an individual 
picked uniformly at random at time $T$ given $Y_T=y_T$ for some $y_T>0$. Then $F_{y_T}$ is given by
 \begin{align*}
F_{y_T}(t)=& \frac{y_T-1}{y_T}\biggl(1-\frac{e^{-\lambda 
t}-e^{-(\lambda-\mu)T}e^{-\mu 
t}}{1-e^{-(\lambda-\mu)T}} \biggr)
\\ &+ \frac{1}{y_T}\biggl(\frac{\lambda 
(1-e^{-\mu t})-\mu(1-e^{-\lambda t}) }{\lambda -\mu}\1_{\lbrace t<T\rbrace} + \1_{\lbrace t=
T\rbrace}\biggr)
\end{align*}
for $t\in  [0,T]$.
\end{satz}
A simple computation yields then the unconditional age distribution (see \cite{age} for details).
\begin{kor}[cf.\ {\protect \citealp[Corollary~2.3]{age}}]\label{nk1}
The cumulative distribution function $F$ of the age of an individual picked 
uniformly at random at time $T$ is given by 
\begin{align*}
 F(t)= &\biggl(1-\frac{\lambda-\mu}{\lambda e^{(\lambda-\mu)T} - \lambda} \log 
\biggl(\frac{\lambda e^{(\lambda-\mu)T} - \mu}{\lambda-\mu} \biggr)\biggr) 
\biggl(1-\frac{e^{-\lambda t}-e^{-(\lambda-\mu)T}e^{-\mu 
t}}{1-e^{-(\lambda-\mu)T}} \biggr) 
  \\
  &+ \frac{\lambda-\mu}{\lambda e^{(\lambda-\mu)T} - \lambda} \log 
\biggl(\frac{\lambda e^{(\lambda-\mu)T} - \mu}{\lambda-\mu} \biggr) 
\biggl(\frac{\lambda 
(1-e^{-\mu t})-\mu(1-e^{-\lambda t}) }{\lambda -\mu}\1_{\lbrace t<T\rbrace} + 
\1_{\lbrace t=T\rbrace}\biggr)
\end{align*}
for $t\in  [0,T]$.
\end{kor}
The next corollary states that the age distribution converges exponentially fast 
to the $\mathrm{Exp}(\lambda)$ distribution in a certain sense for $\lambda>\mu$. It is an immediate consequence of Corollary \ref{nk1}.
\begin{kor}[{\protect \citealp[Corollary~2.4]{age}}]\label{gage}
Let $\lambda >\mu$, and let $A$ denote the age of an individual picked uniformly at random at time $T$. 
Then there exists a random variable $Z$ with 
${\mcl(Z|Y_T>0)=\mathrm{Exp}(\lambda)}$ such that
$$ \E\Bigl(\bigl| e^{-c A}- e^{-c Z} \bigr| \Bigm| Y_T>0\Bigr) 
\leq\frac{\lambda}{c+\lambda} \frac{1}{e^{(\lambda-\mu) T}-1} + 
\frac{\lambda-\mu}{\lambda 
e^{(\lambda-\mu)T}-\lambda}\biggl(\log\biggl(\frac{\lambda}{\lambda-\mu}
\biggr)+(\lambda-\mu)T\biggr)  $$
for any $c>0$.
\end{kor}

Another random quantity we need to control for our main bound is, at any fixed 
time $T$, the time since the last event has occurred. 
Let $0=T_1< T_2 < \ldots$ be the event times of $(Y_t)_{t\geq0}$.
Since $\mathfrak{B}_T+\mathfrak{D}_T$ is the number of events up to time $T$, the random 
variable $T-T_{\mathfrak{B}_T+\mathfrak{D}_T}$ describes the quantity we are interested in.
The following result states that the ${\mathrm{Exp}((Y_T-1)\lambda})$ 
distribution is a stochastic upper bound given $Y_T$ on $\{Y_T > 1\}$ and follows from Theorem \ref{gagen} above.
\begin{satz}[]\label{thm2}
 Given $Y_T=y_T$, the distribution of $T-T_{\mathfrak{B}_T+\mathfrak{D}_T}$ is stochastically dominated by 
the distribution with cumulative distribution function
  $$G(t)=\1_{\lbrace y_T>1 \rbrace } (1-e^{-(y_T-1)\lambda t}) + 
\1_{\lbrace y_T \leq 1 \rbrace } \1_{\lbrace t\geq T\rbrace}.$$
\end{satz}

%%%%%%%%%%%%%%%%%%%%%%%%%%%%%%%%%%%%%%%%%%%%%%%%%%%%%%%%%%%%%%%%%%%%%%%%%%%%%%

%%%%%%%%%%%%%%%%%%%%%%%%%%%%%%%%%%%%%%%%%%%%%%%%%%%%%%%%%%%%%%%%%%%%%%%%%%%%%%

%%%%%%%%%%%%%%%%%%%%%%%%%%%%%%%%%%%%%%%%%%%%%%%%%%%%%%%%%%%%%%%%%%%%%%%%%%%%%%

\section*{Appendix A2: Lemmas for the proof of Theorem~\ref{thm}}\label{A2}

\subsection*{Some lemmas of general interest}
Recall that we set $\E^*(\pkt)=\E(\pkt|Y_T>0)$ and $\P^*(\pkt)=\P(\pkt|Y_T>0)$.
We start with a number of results about the linear birth and death process 
$(Y_t)_{t\geq0}$ that could well be useful in other situations. We first compute the extinction 
probability given the process has survived up to time $T$. We write 
${Y_\infty = \lim\limits_{t \rightarrow \infty} Y_t\in \lbrace 0,\infty\rbrace}$ a.s.
\begin{lemma} \label{austerben2}
For the conditioned extinction probability given $Y_T>0$, we have 
 \begin{align*}
  \P^*(Y_\infty =0) &= \frac{\mu}{\lambda}e^{-(\lambda-\mu)T}. 
 \end{align*}
\end{lemma}\vspace{1mm}

\begin{bew} 
 By conditioning on the population size, we obtain
\begin{align}
 \P^*(Y_\infty =0) &= \E^*(\P(Y_\infty =0| Y_T))=\E^*\biggl( 
\biggl(\frac{\mu}{\lambda}\biggr)^{Y_T}\biggr)\label{aussterben1}
\end{align}
since $({\mu}/{\lambda})^{m}$ is the extinction probability of a linear birth and death process with birth rate $\lambda$, death rate $\mu$ and initial value $m$ (see Remark \ref{ext}). 

On the one hand, we have
\begin{align*}
 \E\biggl(\biggl( \frac{\mu}{\lambda}\biggr)^{Y_T}\biggr) &= \E\biggl( 
\biggl(\frac{\mu}{\lambda}\biggr)^{Y_T}\biggm| Y_T>0\biggr) \P(Y_T>0) + 
\E\biggl( \biggl(\frac{\mu}{\lambda}\biggr)^{Y_T}\biggm| Y_T=0\biggr) \P(Y_T=0)
 \\
 &= \E^*\biggl( \biggl(\frac{\mu}{\lambda}\biggr)^{Y_T}\biggr) 
(1-p_0(T))+p_0(T).
\end{align*}
On the other hand, we can make use of the known formula for the probability generating function of $Y_T$ (see III.5 in \cite{ath}) to obtain
\begin{align*}
 \E\biggl(\biggl( \frac{\mu}{\lambda}\biggr)^{Y_T}\biggr) =\frac{\mu}{\lambda}.
 \end{align*}
 This yields
 \begin{align*}
  \E^*\biggl( \biggl(\frac{\mu}{\lambda}\biggr)^{Y_T}\biggr) &= 
\biggl(\frac{\mu}{\lambda} - p_0(T)\biggr) (1-p_0(T))^{-1}
=\frac{\mu(\lambda-\mu)}{\lambda(\lambda e^{(\lambda-\mu)T}-\mu)} \frac{\lambda 
e^{(\lambda-\mu)T}-\mu}{(\lambda-\mu)e^{(\lambda-\mu)T}}\\&
\leq \frac{\mu}{\lambda} e^{-(\lambda-\mu)}.
 \end{align*}
\end{bew}
For the probability of $Y_T=1$ conditioned on $Y_T>0$, we have the following lemma.
\begin{lemma}[used for Lemma \ref{nn1}] \label{nnl}
We have for $T\geq \frac{1}{\lambda-\mu}\log(2)$
\begin{align*}
 \P^*(Y_T=1) \leq \frac{2(\lambda-\mu)}{\lambda} e^{-(\lambda-\mu)T}.
\end{align*}
\end{lemma}\vspace{1mm}

\begin{bew}
 Using the probability mass functions $p_n$ and the function $\tilde p$ from Appendix \hyperref[sec:lbdp]{A1}, for 
$T\geq \frac{1}{\lambda-\mu}\log(2)$, we obtain 
 \begin{align*}
 \P^*(Y_T=1) = \frac{p_1(T)}{1-p_0(T)}= 1-\lambda \tilde p(T) = 
\frac{\lambda-\mu}{\lambda e^{(\lambda-\mu)T}-\mu} \leq 
\frac{2(\lambda-\mu)}{\lambda e^{(\lambda-\mu)T}}.
\end{align*}
\end{bew}

We next consider sub- and supermartingale properties of conditioned processes.
\begin{lemma}[used for Lemma~\ref{supermartingale}]\label{submartingale}\
\vspace{-1mm}
\begin{enumerate}[(i)]
 \item $(Y_t)_{t \geq 0}$ conditioned on $Y_\infty = 0$ is a supermartingale.
 \item $(Y_t)_{t \geq 0}$ conditioned on $Y_\infty > 0$ is a submartingale.
\end{enumerate}
\end{lemma}\vspace{1mm}

\begin{bew1}\
\begin{enumerate}[(i)]
 \item Consider a subcritical linear birth and death process $(\tilde Y_t)_{t\geq0}$ with \emph{birth rate} $\mu$, \emph{death rate} $\lambda$ and initial value one. Then ($\tilde Y_t)_{t\geq0}$ has the same law as $(Y_t)_{t\geq0}$ conditioned 
on $Y_\infty =0$ (see e.g.\ page 78 in \cite{lam}). It is well-known that a subcritical linear birth and death process is a supermartingale, which yields the result.
\item 
Consider a process $(\hat Y_t)_{t\geq0}$ that has the law of $(Y_t)_{t\geq0}$ conditioned on $Y_\infty >0$. Note that $(\hat Y_t)_{t\geq0}$ inherits the Markov property from $(Y_t)_{t\geq0}$. Furthermore, it is well-known that $(Y_t)_{t\geq0}$ is a submartingale. Thus we obtain for $t>s\geq0$ and $y_s \in \en$
\begin{align*}
\E(\hat Y_t|\hat Y_s=y_s)&= \E(Y_t|Y_s=y_s,Y_\infty>0)
\\
&=\frac{\E(\1_{\{Y_\infty>0\}} Y_t|Y_s=y_s)}{\P(Y_\infty>0|Y_s=y_s)}
\\
&=\frac{\E(Y_t|Y_s=y_s) - \E(\1_{\{Y_\infty=0\}} Y_t|Y_s=y_s)}{\P(Y_\infty>0|Y_s=y_s)}
\\
&\geq \frac{y_s- \P(Y_\infty=0|Y_s=y_s)\E( \tilde Y_t|\tilde Y_s=y_s)}{\P(Y_\infty>0|Y_s=y_s)}
\\
&\geq \frac{y_s- \P(Y_\infty=0|Y_s=y_s) y_s}{\P(Y_\infty>0|Y_s=y_s)}
\\
&=y_s,
\end{align*}
where $(\tilde Y_t)_{t\geq0}$ is the supermartingale from (i), which also inherits the Markov property from $( Y_t)_{t\geq0}$.
Thus $(Y_{t})_{t\geq 0}$ conditioned on ultimate survival is a submartingale.$\hfill\square$
\end{enumerate}
\end{bew1}

The following is a simple consequence.
\begin{lemma}[used for Lemma~\ref{hl1}]\label{supermartingale}
 $(Y_{t}^{-1})_{k\in \en}$ conditioned on $Y_\infty > 0$ is a supermartingale.
\end{lemma}
\begin{bew}
 In general, for a submartingale $(Z_t)_{t\geq 0}$ with respect to a filtration $(\mathcal{F}_t)_{t\geq 0}$ with $Z_t\geq 1$ for all $t\geq0$, we have for $t>s\geq0$
$$\E\biggl( \frac 1 {Z_{t}}-\frac 1 {Z_s}\biggm| \mathcal{F}_s \biggr) = 
\E\biggl( \frac{Z_{s}-Z_{t}}{Z_s Z_{t}}\biggm|\mathcal{F}_s \biggr) \leq 
\E( {Z_{s}-Z_{t}}|\mathcal{F}_s ) \leq 0.$$
Thus $(Z_t^{-1})_{t\geq 0}$ is a supermartingale. Consequently, Lemma \ref{submartingale} implies that $(Y_{t}^{-1})_{t\geq 0}$ conditioned on ultimate survival is a supermartingale. 
\end{bew}
\vspace{1mm}

In order to cope with dependencies, we already introduced the random variable $\mathcal{K}(T)$ in Definition~\ref{def2}. For the same reason, we define the deterministic number $\kappa(T)$, which depends on $T$ in such a way that the probability for more than 
$\kappa(T)$ events up to time $T$ decreases exponentially in $T$ (see Lemma~\ref{ne}(ii) below). The quantity $\kappa(T)$ is solely used in Lemma~\ref{ne} and in the proof of Lemma \ref{nl} below. Essentially, we substitute the number of events $\mathfrak{B}_T+\mathfrak{D}_T$ up to time $T$ by 
$\kappa(T)$ since it is difficult to treat the dependencies between the number 
of events up to time $T$ and the event times. 
\begin{notat}\label{def2a}
 Let $\kappa(T):=\lfloor e^{\frac{3}{2}(\lambda+\mu)T} \rfloor$.
\end{notat} \vspace{1mm}
Lemma~\ref{ne} below states essentially that, for $T$ large, it is unlikely that the randomly picked individual was born before $\frac{T}{2}$ or that more than $e^{\frac{3}{2}(\lambda+\mu)T}$ events have occurred by time~$T$.

\begin{lemma}\label{ne}~ 
\vspace{-1mm}
\begin{enumerate}[(i)]
\item For the probability that fewer than $\mathcal{K}(T)$ events have occurred 
up to the birth time of the randomly picked node $J_T$ was born, we have 
\begin{align*}
 \P^*(r(J_T)\leq \mck(T)) &=\P^*\biggl(T_{r(J_T)}\leq \frac{T}{2}\biggr)
 \\
 &\leq e^{-\frac{1}{2}\lambda 
T}+\frac{\lambda-\mu}{\lambda 
e^{(\lambda-\mu)T}-\lambda}\biggl(\log\biggl(\frac{\lambda
} {\lambda-\mu} \biggr)+(\lambda-\mu)T\biggr).
 \end{align*}
 \item For $T\geq \frac{2(\log(4 \lambda)-\log(\lambda-\mu))}{\lambda+\mu}$, we have
 \begin{align*}
 \P^*(\kappa(T)\leq \mathcal{M}_T )\leq \frac{60
\lambda^3(\lambda+\mu)}{(\lambda-\mu)^4} e^{-(\lambda+\mu)T},
 \end{align*}
 where $\mathcal{M}_T=\mathfrak{B}_T+\mathfrak{D}_T$ is the number of events up to time $T$ as before.
\end{enumerate}
\end{lemma}
\begin{bew}
\begin{enumerate}[(i)]
 \item We have
\begin{align*}
 \P^*(r(J_T)\leq \mck(T))& = \P^*( T_{r(J_T)}\leq 
T_{\mathcal{K}(T)})=1-\P^*\biggl( T-T_{r(J_T)}\leq 
\frac{T}{2}\biggr)
\end{align*}
Note that $\P^*( T-T_{r(J_T)}\leq T/2 )$ is given by Corollary \ref{nk1}. Thus the result follows from this corollary by elementary computation.

% \begin{align*}
%  e^{-\frac{1}{2}\lambda 
% T}+\frac{\lambda-\mu}{\lambda 
% e^{(\lambda-\mu)T}-\lambda}\biggl(\log\biggl(\frac{\lambda
% } {\lambda-\mu} \biggr)+(\lambda-\mu)T\biggr).
%  \end{align*}

\item Using Proposition \ref{pr1}, we compute 
\begin{align}
 \kappa(T)-2\E(\mathfrak{B}_T) & \geq e^{\frac{3}{2}(\lambda+\mu)T} -1- 
\frac{2\lambda}{\lambda-\mu}e^{(\lambda-\mu)T}+\frac{2\mu}{\lambda-\mu}\notag
 \\
 &= e^{(\lambda-\mu)T} \biggl(e^{2\mu T} e^{\frac{1}{2}(\lambda+\mu)T}-\frac{2\lambda}{\lambda-\mu} \biggr) + \frac{3\mu-\lambda}{\lambda-\mu}.\label{ste}
\end{align}
Since $T\geq \frac{2(\log(4 \lambda)-\log(\lambda-\mu))}{\lambda+\mu}$, we have
\begin{align}
 \frac{1}{2} e^{2\mu T} e^{\frac{1}{2}(\lambda+\mu)T} \geq 
\frac{2\lambda}{\lambda-\mu}.\label{ste3}
\end{align}
Thus for $T\geq \frac{2(\log(4 \lambda)-\log(\lambda-\mu))}{\lambda+\mu}$, the 
right-hand side of \eqref{ste} is larger than or equal to
\begin{align}
 e^{(\lambda-\mu)T} \frac{1}{2} e^{2\mu T} 
e^{\frac{1}{2}(\lambda+\mu)T} + 
\frac{3\mu-\lambda}{\lambda-\mu}.\label{ste2}
\end{align}
Since for $T\geq \frac{2(\log(4 \lambda)-\log(\lambda-\mu))}{\lambda+\mu}$, we 
have
\begin{align}
 e^{(\lambda-\mu)T} &\geq 1+(\lambda-\mu)T\geq 
1+2\log\biggl(\frac{4\lambda}{\lambda+\mu}\br\geq1+2\log(4) \geq2,\label{ste6}
\end{align}
inequality \eqref{ste3} implies that 
for $T\geq \frac{2(\log(4 \lambda)-\log(\lambda-\mu))}{\lambda-\mu}$, the 
expression \eqref{ste2} is bounded from below by
\begin{align*}
 e^{(\lambda-\mu)T} \frac{3}{8} e^{2\mu T} 
e^{\frac{1}{2}(\lambda+\mu)T} + 
\frac{3\mu}{\lambda-\mu} \geq \frac{3}{8} 
e^{\frac{3}{2}(\lambda+\mu)T}.
\end{align*}
Thus for $T\geq \frac{2(\log(4 \lambda)-\log(\lambda-\mu))}{\lambda+\mu}$, we 
have 
\begin{align}
 \kappa(T)-2\E(\mathfrak{B}_T) \geq \frac{3}{8}  
e^{\frac{3}{2}(\lambda+\mu)T} >0. \label{ste4}
\end{align}
Consequently, we can apply Chebyshev's inequality and obtain
\begin{align}
 \P(\kappa(T)\leq \mathcal{M}_T | Y_T>0) &\leq \frac{\P(\kappa(T)\leq 
\mathcal{M}_T )}{\P(Y_T>0)} \notag \leq \frac{\P(\kappa(T) \leq 2 
\mathfrak{B}_T)}{\P(Y_T>0)} \notag
 \\
 &= \frac{1}{\P(Y_T>0)}\P(\kappa(T)-2 \E(\mathfrak{B}_T) \leq 2 \mathfrak{B}_T -2 \E(\mathfrak{B}_T) ) \notag
 \\
 &\leq \frac{1}{\P(Y_T>0)} \frac{4 \Var(\mathfrak{B}_T)}{(\kappa(T)-2\E(\mathfrak{B}_T))^2}, 
\label{ste5}
\end{align}
where $T\geq \frac{2(\log(4 \lambda)-\log(\lambda-\mu))}{\lambda+\mu}$. By 
$\P(Y_T>0)\geq \P(Y_\infty>0)=\frac{\lambda-\mu}{\lambda}$, Proposition 
\ref{pr1}  
and \eqref{ste4}, 
for $T\geq \frac{2(\log(4 \lambda)-\log(\lambda-\mu))}{\lambda+\mu}$, the 
right-hand side of \eqref{ste5} is smaller than or equal to
\begin{align}
 \frac{256}{9}\frac{ \lambda 
}{\lambda-\mu}\frac{\Var(\mathfrak{B}_T)}{e^{{3}(\lambda+\mu)T} }\leq \frac{30 \lambda 
}{\lambda-\mu} \biggl(\frac{\lambda^2(\lambda+\mu)}{(\lambda-\mu)^3} 
e^{-(\lambda+\mu)T}+\frac{2\lambda^2\mu}{(\lambda-\mu)^3}e^{-2(\lambda+\mu)T} 
\biggr).\label{ste7}
\end{align}
Since $e^{-(\lambda+\mu)T} \leq \frac{1}{2}$ by \eqref{ste6}, the right-hand 
side of \eqref{ste7} is smaller than or equal to
$$\frac{60 
\lambda^3(\lambda+\mu)}{(\lambda-\mu)^4} e^{-(\lambda+\mu)T}.$$
\vspace{-10mm}
\end{enumerate}
\end{bew}

\subsection*{Further lemmas}
In what follows we give some more specialized results that are used in the proof of Theorem~\ref{thm}.

The following lemma states that the time $T-T_{\mathcal{M}_T}$ since the last 
event becomes small quickly and is proved using Theorem \ref{thm2}.
\begin{lemma}\label{nn1}
 For $T\geq \frac{1}{\lambda-\mu}\log(2)$ and $c>0$, we have
 \begin{align}
   \E^*(1-e^{-c(T-T_{\mathcal{M}_T})})\leq 
\frac{2(\lambda-\mu)}{\lambda e^{(\lambda-\mu)T}}+  
\frac{2 c}{\lambda }\biggl( 
\log\biggl(\frac{\lambda}{\lambda-\mu}\biggr) +(\lambda-\mu)T\biggr)e^{-(\lambda-\mu)T}.\label{ste9}
 \end{align}
\end{lemma}
\begin{bew}
 Let $X$ have the cumulative 
distribution function
 $$G(t)=\1_{\lbrace Y_T>1 \rbrace } (1-e^{-(Y_T-1)(\lambda-\mu)t}) + 
\1_{\lbrace Y_T\leq 1 \rbrace } \1_{\lbrace t\geq T\rbrace}.$$
Conditionally on $Y_T$, we then have $T- T_{\mathcal{M}_T } \leq_{st} X$ by 
Theorem \ref{thm2}, which implies
\begin{align}
 \E^*(1-&e^{-c(T-T_{\mathcal{M}_T})})\leq \E^*(1-e^{-c X})\notag
 \\
 &=\E^*\biggl( (1-e^{-c T}) \P(X=T|Y_T)+ \E(1-e^{-c X}| X<T, Y_T) 
\P(X<T|Y_T) \biggr).\label{t7}
\end{align}
Since $\P(X=T|Y_T=1)=1$ and 
$\mathcal{L}(X|Y_T)=\mathrm{Exp}((Y_T-1)\lambda)$ on 
$\lbrace Y_T\geq 2 \rbrace$, we obtain that the right-hand side of 
\eqref{t7} is smaller than or equal to
\begin{align}
 \E^*\biggl(\1_{\lbrace Y_T=1 \rbrace} + \frac{c}{c+(Y_T-1)\lambda} 
\1_{\lbrace Y_T>1 \rbrace} \biggr) &\leq  \P^*(Y_T=1) + \frac{c}{\lambda} \E^*\biggl( \frac{1}{Y_T-1} \1_{\lbrace 
Y_T>1 
\rbrace} \biggr).\notag
\end{align}
For the conditional probability of $Y_T=1$, by Lemma \ref{nnl}, we have for 
$T\geq \frac{1}{\lambda-\mu}\log(2)$
\begin{align*}
 \P^*(Y_T=1) \leq \frac{2(\lambda-\mu)}{\lambda e^{(\lambda-\mu)T}}.
\end{align*}
 Furthermore, we have
 \begin{align}
  \E^*\biggl( \frac{1}{Y_T-1} \1_{\lbrace Y_T>1 \rbrace} \biggr) 
&=\frac{\lambda-\mu}{\lambda e^{(\lambda-\mu)T}-\mu} \sum\limits_{n=1}^\infty 
\frac{1}{n} \biggl(\frac{\lambda e^{(\lambda-\mu)T}-\lambda}{\lambda 
e^{(\lambda-\mu)T}-\mu}\biggr)^n \notag
\\
  &= \frac{\lambda-\mu}{\lambda e^{(\lambda-\mu)T} - \mu}\biggl(-\log\biggl(1-\frac{\lambda e^{(\lambda-\mu)T} -\lambda}{\lambda e^{(\lambda-\mu)T} -\mu} \biggr)\biggr) \notag
  \\
  &= \frac{\lambda-\mu}{\lambda e^{(\lambda-\mu)T} - \mu} \log \biggl(\frac{\lambda e^{(\lambda-\mu)T} - \mu}{\lambda-\mu} \biggr) \notag
  \\
  &\leq \frac{2}{ e^{(\lambda-\mu)T} } \biggl( \log\biggl(\frac{\lambda}{\lambda-\mu}\biggr)+(\lambda-\mu )T \biggr) \notag
 \end{align}
 for $T\geq \frac{1}{\lambda-\mu}\log(2)$, which completes the proof.
\end{bew}

%TODO
\vspace{1mm}
The next result follows from Lemma \ref{supermartingale} and a well-known inequality for supermartingales.
% The next result follows straightforwardly from Corollary \ref{supermartingale} and a well-known inequality for supermartingales.
\begin{lemma}\label{hl1}
For all $\delta, \gamma >0$, we have
 \begin{align*}
 \E\biggl( &\frac{1}{ \min \limits_{T/2\leq t }  
Y_{t}^\delta}\biggm| Y_\infty >0 \biggr)\leq e^{-\gamma(\lambda-\mu)T} + 
\E\biggl(\frac{1}{Y_{T/2}}\biggm| Y_\infty>0\biggr) 
e^{\frac{\gamma}{\delta}(\lambda-\mu)T} .
 \end{align*}
\end{lemma}

\begin{bew}
 Writing $Z_{T/2} = \max_{T/2 \leq t} Y_t^{-1}$, we obtain 
\begin{align}
 \E\bigl( Z_{T/2}^{\delta} \bigm| Y_\infty >0 \bigr) &\leq \E\Bigl( \E\bigl(\1_{\lbrace Z_{T/2}^{\delta} \leq e^{-\gamma(\lambda-\mu) 
T}\rbrace} Z_{T/2}^{\delta} \bigm| 
Y_{T/2} \bigr)\Bigm| Y_\infty>0\Bigr)\notag
\\
&{\ }\ +\E\Bigl( \E\bigl(\1_{\lbrace Z_{T/2}^{\delta} > e^{-\gamma(\lambda-\mu) T}\rbrace} 
Z_{T/2}^{\delta} \bigm| 
Y_{T/2} \bigr)\Bigm| Y_\infty>0\Bigr) \notag
\\[0.5mm]
& \leq e^{-\gamma(\lambda-\mu)T}  \notag+ \E\Bigl(\P\bigl( 
Z_{T/2} > 
e^{-\frac{\gamma}{\delta}(\lambda-\mu)T} \bigm| Y_{T/2} 
\bigr)\Bigm| Y_\infty>0\Bigr)\notag
\\[0.5mm]
&\leq e^{-\gamma(\lambda-\mu)T} +  
\E\bigl(Y_{T/2}^{-1}\bigm| Y_\infty>0\bigr) 
e^{\frac{\gamma}{\delta}(\lambda-\mu)T},\notag
 \end{align}
where the last line follows from Lemma \ref{supermartingale} and the the inequality (2) in Theorem 6.14 on page 99 in~\cite{yeh}.
\end{bew}
%
% \vspace{1mm}

We use the following lemma to bound the conditional expectation on the right 
hand side of Lemma~\ref{hl1} from above.
\begin{lemma}\label{n3}
 We have for $T\geq \frac{2}{\lambda-\mu}\log(2)$
 \begin{align*}
\E\biggl(\frac{1}{Y_{T/2}}\biggm| Y_\infty >0 \biggr)\leq
2\biggl(\log\biggl(\frac{\lambda}{\lambda-\mu}\biggr)+\frac{1}{2}(\lambda-\mu) 
T\biggr) e^{-\frac{1}{2}(\lambda-\mu)T}.
\end{align*} 
\end{lemma}
\begin{bew}
 For $T\geq \frac{2}{\lambda-\mu}\log(2)$, we obtain
\begin{align}
\E\bigl(Y_{T/2}^{-1} \bigm| Y_\infty >0 \bigr) 
&=\frac{\E\bigl(Y_{T/2}^{-1}\1_{\lbrace Y_\infty >0 
\rbrace}\bigr)}{\P(Y_\infty >0)} \leq \frac{\lambda}{\lambda-\mu}  \E\bigl(Y_{T/2}^{-1} 
\1_{\lbrace Y_{T/2}>0\rbrace}\bigr)\notag \\ &\leq \frac{\lambda}{\lambda-\mu} 
\E\bigl(Y_{T/2}^{-1}\bigm| 
Y_{T/2}>0\bigr) \notag
 \\
 &\leq \frac{\lambda}{\lambda e^{\frac{1}{2}(\lambda-\mu)  T}- \lambda 
}\biggl(\log\biggl(\frac{\lambda}{\lambda-\mu}\biggr)+\frac{1}{2}(\lambda-\mu)  
T\biggr) \notag
 \\
 &\leq \frac{1}{ \frac{1}{2} e^{\frac{1}{2}(\lambda-\mu)  T}+\frac{1}{2} 
e^{\frac{1}{2}(\lambda-\mu)  T}- 1 
}\biggl(\log\biggl(\frac{\lambda}{\lambda-\mu}\biggr)+\frac{1}{2}(\lambda-\mu)  
T\biggr) \notag
  \\
 &\leq 
2\biggl(\log\biggl(\frac{\lambda}{\lambda-\mu}\biggr)+\frac{1}{2}(\lambda-\mu) 
T\biggr) e^{-\frac{1}{2}(\lambda-\mu)T},\notag
\end{align}
where the second line follows from Proposition~\ref{l1b}.
\end{bew}

Recall that $R_{T_l,T}$ is the number of nodes that are alive at time $T_l$ and survive 
up to time~$T$. The following lemma gives us the first two conditional moments of $R_{ T_l,T}$.

\begin{lemma}[used for Lemma~\ref{nl1}]\label{no}
We have
 \begin{align*}
  \E(R_{ T_l,T}|Y_{T_l}, T_l, T_{l+1} )&=(Y_{T_l}-1) e^{-\mu(T-T_{l+1})},
  \\
  \E( R_{ T_l,T}^2 | Y_{T_l}, T_l, T_{l+1} )&= Y_{T_l} e^{-\mu(T-T_{l+1})} - 
Y_{T_l} e^{-2\mu(T-T_{l+1})} + Y_{T_l}^2 e^{-2\mu(T-T_{l+1})} 
 \\&\hspace{-5mm} - \P(Y_{T_{l+1}} =Y_{T_l} - 1| Y_{T_l}, T_l, T_{l+1}) \biggl(  2( 
Y_{T_l}-1) e^{-2\mu (T-T_{l+1})}+e^{-\mu(T-T_{l+1})} \biggr).
 \end{align*}
\end{lemma}

\begin{bew}
Firstly, we determine the conditional expectation of $R_{ T_l,T}$:
\begin{align}
 \E(R_{ T_l,T}|Y_{T_l}, T_l, T_{l+1} ) &= p^+  \E(R_{ T_l,T}|Y_{T_l}, T_l, 
T_{l+1} , Y_{T_{l+1}}=Y_{T_l}+1) \notag
 \\
 &\hphantom{=} + p^-  \E(R_{ T_l,T}|Y_{T_l}, T_l, T_{l+1} , 
Y_{T_{l+1}}=Y_{T_l}-1), 
\label{de}
\end{align}
where $p^+ :=\P(Y_{T_{l+1}}=Y_{T_l}+1| Y_{T_l}, T_l, T_{l+1})$ and 
$p^-:=\P(Y_{T_{l+1}}=Y_{T_l}-1| Y_{T_l}, T_l, T_{l+1})$.
% \footnote{We have $p^+= 
% \frac{\lambda}{\lambda+\mu}$ and $p^-= \frac{\mu}{\lambda+\mu}$. However, we do 
% not need these explicit expressions for our purposes.}

With
\begin{align}
 \E(R_{ T_l,T}|Y_{T_l}, &T_l, T_{l+1} , Y_{T_{l+1}}=Y_{T_l}+1) \notag \\& = \E(R_{ 
T_{l+1},T}-\1_{\lbrace T_{r^{-1}(l+1)}^->T\rbrace} | Y_{T_l}, T_l, 
T_{l+1} , Y_{T_{l+1}}=Y_{T_l}+1)  \notag
 \\
 &=(Y_{T_l}+1) e^{-\mu(T-T_{l+1})} -e^{-\mu(T-T_{l+1})}  = Y_{T_l} 
e^{-\mu(T-T_{l+1})}\label{v}
\end{align}
and 
\begin{align*}
 \E(R_{ T_l,T}|Y_{T_l}, T_l, T_{l+1} , Y_{T_{l+1}}=Y_{T_l}-1) &= \E(R_{ 
T_{l+1},T}|Y_{T_l}, T_l, T_{l+1} , Y_{T_{l+1}}=Y_{T_l}-1) .
 \\
 &=(Y_{T_l}-1) e^{-\mu(T-T_{l+1})}.
\end{align*}
Equation \eqref{de} implies
\begin{align}
  \E(R_{ T_l,T}|Y_{T_l}, T_l, T_{l+1} ) = Y_{T_l} e^{-\mu(T-T_{l+1})} - 
e^{-\mu(T-T_{l+1})} p^- .\notag
\end{align}
Secondly, we compute the conditional second moment of $R_{ T_l,T}$:
\begin{align}
 \E( R_{ T_l,T}^2 | Y_{T_l}, T_l, T_{l+1} )&=p^+   \E( R_{ T_l,T}^2 | Y_{T_l}, 
T_l, T_{l+1} , Y_{T_{l+1}}=Y_{T_l}+1)\notag
 \\ 
 &\hphantom{=} +p^-  \E( R_{ T_l,T}^2 | Y_{T_l}, T_l, T_{l+1} , 
Y_{T_{l+1}}=Y_{T_l}-1). 
\label{dp}
\end{align}
We treat the summands separately again. For the case where a birth occurs at 
time $T_{l+1}$, we have \fontsize{11.5}{11}
\begin{align}
 &\E( R_{ T_l,T}^2 | Y_{T_l}, T_l, T_{l+1} , Y_{T_{l+1}}=Y_{T_l}+1) \notag
 \\[1mm]
 &= \E((R_{ T_{l+1},T}-\1_{\lbrace T_{r^{-1}(l+1)}^->T\rbrace})^2 | Y_{T_l}, 
T_l, 
T_{l+1} , Y_{T_{l+1}}=Y_{T_l}+1) \notag
 \\[-0.5mm]
 &= \E\Bigl(R_{ T_{l+1},T}^2 -2 R_{ T_{l+1},T} \1_{\lbrace 
T_{r^{-1}(l+1)}^->T\rbrace} + \1_{\lbrace T_{r^{-1}(l+1)}^->T\rbrace} \Bigm| 
Y_{T_l}, T_l, T_{l+1} , Y_{T_{l+1}}=Y_{T_l}+1,Y_{T_l}>0\Bigr) \label{zp}
\end{align}\normalsize
and further
\begin{align}
 \E( R_{ T_{l+1},T}^2 | Y_{T_l}, T_l, T_{l+1} , Y_{T_{l+1}}=Y_{T_l}+1) 
&=\Var(R_{ T_{l+1},T} | Y_{T_l}, T_l, T_{l+1} , Y_{T_{l+1}}=Y_{T_l}+1) \notag
 \\
 &\hphantom{=} + (\E(R_{ T_{l+1},T} | Y_{T_l}, T_l, T_{l+1} , 
Y_{T_{l+1}}=Y_{T_l}+1))^2. 
\label{varl}
\end{align}
Given $Y_{T_{l+1}}$, let $\mathcal{L}_{l+1}$ be the set of the $Y_{T_{l+1}}$ 
nodes living at time $T_{l+1}$. Then by independence of various death times, we 
obtain for the conditional variance 
\begin{align*}
 &\Var(R_{ T_{l+1},T} | Y_{T_l}, T_l, T_{l+1} , Y_{T_{l+1}}=Y_{T_l}+1) \notag
 \\
 &= \Var\biggl(\sum\limits_{j\in \mathcal{L}_{l+1}} \1_{\lbrace T_j^->T 
\rbrace}\biggm| Y_{T_l}, T_l, T_{l+1} , Y_{T_{l+1}}=Y_{T_l}+1,Y_{T_l}>0\biggr)  
  \\
 &= (Y_{T_l}+1) (e^{-\mu(T-T_{l+1})}-e^{-2\mu(T-T_{l+1})}).
\end{align*}
For the second summand of \eqref{varl}, we obtain:
\begin{align*}
 \E(R_{ T_{l+1},T} | Y_{T_l}, T_l, T_{l+1} , Y_{T_{l+1}}=Y_{T_l}+1) = 
(Y_{T_l}+1) 
e^{-\mu(T-T_{l+1})}.
\end{align*}
Thus \eqref{varl} is equal to
\begin{align*}
(Y_{T_l}+1)\biggl(e^{-\mu(T-T_{l+1})}+ Y_{T_l} e^{-2 \mu(T-T_{l+1})}\biggr).
\end{align*}
For the remaining parts of \eqref{zp}, we have
\begin{align*}
 &\E( R_{ T_{l+1},T} \1_{\lbrace T_{r^{-1}(l+1)}^->T\rbrace}  | Y_{T_l}, T_l, 
T_{l+1} , Y_{T_{l+1}}=Y_{T_l}+1) 
 \\
 &= \P(T_{r^{-1}(l+1)}^->T | Y_{T_l}, T_l, T_{l+1} , Y_{T_{l+1}}=Y_{T_l}+1)  
 \\
 &\ \ \ \cdot \E(R_{ T_{l+1},T} | T_{r^{-1}(l+1)}^->T,Y_{T_l}, T_l, T_{l+1} , 
Y_{T_{l+1}}=Y_{T_l}+1)
 \\
 &= e^{-\mu(T-T_{l+1})} (1+ Y_{T_l} e^{- \mu (T-T_{l+1})})
\end{align*}
and 
\begin{align*}
 \E(\1_{\lbrace T_{r^{-1}(l+1)}^->T\rbrace } | Y_{T_l}, T_l, T_{l+1} , 
Y_{T_{l+1}}=Y_{T_l}+1) = e^{-\mu(T-T_{l+1})}.
\end{align*}
Thus \eqref{zp} implies
\begin{align*}
 &\E( R_{ T_l,T}^2 | Y_{T_l}, T_l, T_{l+1} , Y_{T_{l+1}}=Y_{T_l}+1)  
 \\
 &= (Y_{T_l}+1) \biggl(e^{-\mu(T-T_{l+1}} + Y_{T_l} e^{-2\mu(T-T_{l+1})}\biggr) 
- 2\biggl(e^{-\mu(T-T_{l+1}}+ Y_{T_l} e^{-2 \mu (T-T_{l+1})} \biggr) 
+e^{-\mu(T-T_{l+1})}
 \\
 &=Y_{T_l} e^{-\mu (T-T_{l+1})}- Y_{T_l} e^{-2\mu (T-T_{l+1})} + Y_{T_l}^2 
e^{-2\mu(T-T_{l+1})}.
\end{align*}
For the case where a death occurs at time $T_{l+1}$, we have
\begin{align*}
 &\E( R_{ T_l,T}^2 | Y_{T_l}, T_l, T_{l+1} , Y_{T_{l+1}}=Y_{T_l}-1) 
 \\
 &= \Var(R_{ T_{l+1},T}| Y_{T_l}, T_l, T_{l+1} , Y_{T_{l+1}}=Y_{T_l}-1) + 
(\E(R_{ 
T_{l+1},T}| Y_{T_l}, T_l, T_{l+1} , Y_{T_{l+1}}=Y_{T_l}-1))^2
 \\
 &= (Y_{T_l}-1) ( e^{- \mu (T-T_{l+1})} - e^{-2\mu(T-T_{l+1})}) + (Y_{T_l} -1)^2 
e^{-2\mu (T-T_{l+1})}
 \\
 &= (Y_{T_l} -1) e^{-\mu(T-T_{l+1})} + (Y_{T_l}^2 - 3Y_{T_l} +2) e^{-2\mu 
(T-T_{l+1})}.
\end{align*}
Thus from \eqref{dp} follows
\begin{align*}
 \E( R_{ T_l,T}^2 | Y_{T_l}, T_l,T_{l+1},Y_{T_l}>0 ) &= Y_{T_l} 
e^{-\mu(T-T_{l+1})} - Y_{T_l} e^{-2\mu(T-T_{l+1})} + Y_{T_l}^2 
e^{-2\mu(T-T_{l+1})} 
 \\&\ \ \ + p^- \biggl(  2(1- Y_{T_l}) e^{-2\mu (T-T_{l+1})}-e^{-\mu(T-T_{l+1})} 
\biggr) .
\end{align*}
\end{bew}

Knowing the conditional moments of $R_{ T_l,T}$, we can find an upper bound for 
a more complex conditional expectation involving $R_{ T_l,T}$ that appears in 
the proof of the main theorem. The following lemma allows us to control the proportion of nodes surviving up to time $T$ of the 
nodes
alive at time $T_l$ (if we reduce both numbers by one).

\begin{lemma}\label{nl1}
For $l\in\en$, we have 
 \begin{align*}
  &\E\biggl(\biggl|\frac {R_{ T_l,T}-1} {Y_{ T_l}-1}\1_{\lbrace Y_{ T_l}>1 
\rbrace}-e^{-\mu(T- T_{l+1})}\biggr| \ \biggm| Y_{T_l}, T_l, T_{l+1} 
\biggr) \leq \biggl(\frac{6}{Y_{T_l}}\biggr)^\frac{1}{2}. 
\end{align*}
\end{lemma}

 \begin{bew}
By Jensen's inequality, we obtain
\begin{align}
 &\E\biggl(\biggl|\frac {R_{ T_l,T}-1} {Y_{ T_l}-1}\1_{\lbrace Y_{ T_l}>1 
\rbrace}-e^{-\mu(T- T_{l+1})}\biggr| \ \biggm| Y_{T_l}, T_l, T_{l+1} 
\biggr)\notag
 \\
 &\leq \biggl(\E\biggl(\biggl(\frac {R_{ T_l,T}-1} {Y_{ T_l}-1}\1_{\lbrace Y_{ 
T_l}>1 \rbrace}-e^{-\mu(T- T_{l+1})}\biggr)^2 \ \biggm| Y_{T_l}, T_l, T_{l+1} 
\biggr)\biggr)^{\frac{1}{2}}\notag
 \\
 &= \biggl(\1_{\lbrace Y_{ 
T_l}>1 \rbrace}\frac{\E( (R_{ T_l,T}-1)^2 | Y_{T_l}, T_l, T_{l+1} )}{(Y_{ 
T_l}-1)^2} \notag
\\
&\hphantom{= \biggl(} -\frac{2}{Y_{ T_l}-1}\1_{\lbrace Y_{ 
T_l}>1 \rbrace} e^{-\mu(T- T_{l+1})} \E(R_{ T_l,T}-1|Y_{T_l}, 
T_l, T_{l+1}) + e^{-2\mu(T- T_{l+1})}\biggr)^{\frac{1}{2}}.\label{j}
\end{align}

Lemma \ref{no} implies 
\begin{align}
  \frac{2}{Y_{T_l}-1} \1_{\lbrace Y_{T_l} >1 \rbrace } e^{-\mu(T-T_{l+1})} 
\E(R_{T_l,T}-1| Y_{T_l},T_l,T_{l+1})=& 2 \cdot \1_{\lbrace Y_{T_l} >1 \rbrace } 
e^{-2\mu(T-T_{l+1})} \notag
 \\
 &- \frac{2}{Y_{T_l}-1} \1_{\lbrace Y_{T_l} >1 \rbrace } e^{-\mu(T-T_{l+1})} 
.\label{t01}
\end{align}
and
\begin{align}
 &\frac{1}{(Y_{T_l}-1)^2} \1_{\lbrace Y_{T_l} >1 \rbrace } \E((R_{T_l,T} -1)^2| 
Y_{T_l},T_l,T_{l+1}) \notag
 \\
 &= \frac{1}{(Y_{T_l}-1)^2} \1_{\lbrace Y_{T_l} >1 \rbrace } \biggl( 
\E(R_{T_l,T}^2| Y_{T_l},T_l,T_{l+1}) - 2 \E(R_{T_l,T}| Y_{T_l},T_l,T_{l+1}) 
+1\biggr)\notag
 \\
 &=\frac{1}{(Y_{T_l}-1)^2} \1_{\lbrace Y_{T_l} >1 \rbrace } \biggl( Y_{T_l} 
e^{-\mu(T-T_{l+1})} - Y_{T_l} e^{-2\mu(T-T_{l+1})} + Y_{T_l}^2 
e^{-2\mu(T-T_{l+1})} \notag
 \\[-1mm]
&\hphantom{=\frac{1}{(Y_{T_l}-1)^2}  
\biggl(} 
- \P(Y_{T_{l+1}} =Y_{T_l} - 1| Y_{T_l}, T_l, T_{l+1}) \biggl(  2(Y_{T_l}-1) 
e^{-2\mu (T-T_{l+1})}+e^{-\mu(T-T_{l+1})} \biggr)\biggr) \notag
 \\
 &\hphantom{=}- \frac{2}{Y_{T_l}-1} \1_{\lbrace Y_{T_l} >1 \rbrace } 
e^{-\mu(T-T_{l+1})}
+\frac{\1_{\lbrace Y_{T_l} >1 \rbrace }}{(Y_{T_l}-1)^2} \notag
 \\
 &
 \leq \frac{1}{(Y_{T_l}-1)^2} \1_{\lbrace Y_{T_l} >1 \rbrace } \biggl( Y_{T_l} 
e^{-\mu(T-T_{l+1})} + Y_{T_l}(Y_{T_l}-1) e^{-2\mu(T-T_{l+1})}  \biggr) \notag
 \\
 &\hphantom{=} -\frac{2}{Y_{T_l}-1} \1_{\lbrace Y_{T_l} >1 \rbrace }e^{-\mu(T-T_{l+1})}
+\frac{\1_{\lbrace Y_{T_l} >1 \rbrace }}{(Y_{T_l}-1)^2} \notag
 \\
 &= \frac{\1_{\lbrace Y_{T_l} >1 \rbrace }}{Y_{T_l}-1}  e^{-\mu(T-T_{l+1})} + 
\frac{\1_{\lbrace Y_{T_l} >1 \rbrace }}{(Y_{T_l}-1)^2}  e^{-\mu(T-T_{l+1})} + 
\frac{Y_{T_l}}{Y_{T_l}-1} \1_{\lbrace Y_{T_l} >1 \rbrace } e^{-2\mu(T-T_{l+1})}  
\notag
 \\
 &\hphantom{=}- \frac{2}{Y_{T_l}-1} \1_{\lbrace Y_{T_l} >1 \rbrace } 
e^{-\mu(T-T_{l+1})}+\frac{\1_{\lbrace Y_{T_l} >1 \rbrace }}{(Y_{T_l}-1)^2} 
\notag
 \\
 &\leq \1_{\lbrace Y_{T_l} >1 \rbrace } e^{-2\mu(T-T_{l+1})} 
+\frac{1}{Y_{T_l}-1}\1_{\lbrace Y_{T_l} >1 \rbrace } e^{-2\mu(T-T_{l+1})} + 
\frac{1}{(Y_{T_l}-1)^2}\1_{\lbrace Y_{T_l} >1 \rbrace }\notag
 \\
 &\leq \biggl(\frac{1}{Y_{T_l}-1}+ e^{-2\mu(T-T_{l+1})}\biggr)\1_{\lbrace 
Y_{T_l} >1 \rbrace } .  \label{t0}
\end{align}
By \eqref{t01} and \eqref{t0}, we can bound the right-hand side of \eqref{j} 
from above by
\begin{align}
 \biggl(\frac{3}{Y_{T_l}-1}\1_{\lbrace Y_{T_l} >1 \rbrace }-\1_{\lbrace Y_{T_l} 
>1 
\rbrace } e^{-2\mu(T-T_{l+1})} +e^{-2\mu(T-T_{l+1})}\biggr)^\frac{1}{2} &\leq 
\biggl(\frac{3}{Y_{T_l}-1}\1_{\lbrace Y_{T_l} >1 \rbrace } + \1_{\lbrace 
Y_{T_l} 
=1 \rbrace }\biggr)^\frac{1}{2}\notag
 \\
 &\leq \biggl(\frac{6}{Y_{T_l}}\biggr)^\frac{1}{2}. \label{tag3a}
\end{align}
\hspace*{2mm} \vspace*{-7mm}

\hspace*{2mm}
\end{bew}
%
%\vspace{1mm}

For the conditional expectation of the sum of the squared inter-event times 
since the birth of the randomly picked node, we 
have the following lemma, which is proved by using Lemmas~\ref{ne}, \ref{hl1} and \ref{n3}.
\begin{lemma}\label{nl}\vspace{1mm}
For ${T\geq(\frac{2}{\lambda-\mu}\log(2)\vee\frac{2(\log(4 
\lambda)-\log(\lambda-\mu))}{\lambda+\mu}})$, we have
 \begin{align*}
 \E^*\biggl(\sum_{l=r(J_T)}^{\mathcal{M}_T -1}  ( T_{l+1}- T_l)^2\biggr) 
\leq& \hspace{1mm} 
\frac{\mu}{\lambda}\frac{T^2}{4} 
e^{-(\lambda-\mu)T} +60 T^2 
\frac{\lambda^3(\lambda+\mu)}{(\lambda-\mu)^4} 
e^{-(\lambda+\mu)T}+ T^2 e^{-\frac{1}{2}\lambda 
T} \notag
\\&+\frac{2(\lambda-\mu)}{\lambda}\biggl(\log\biggl(\frac{\lambda
} {\lambda-\mu} \biggr)+(\lambda-\mu)T\biggr)T^2e^{-(\lambda-\mu)T} \notag
 \\
 &+  \biggl(\frac 3 4 T^2+\frac{T}{2(\lambda+\mu)}\biggr) \biggl( 
1+2\log\biggl(\frac{\lambda}{\lambda-\mu}\biggr)+ (\lambda-\mu) T\biggr) 
e^{-\frac{1}{4} (\lambda-\mu)T}.
\end{align*} 
\end{lemma}

\begin{bew}
 For the left-hand side, we deduce
\begin{align}
& \E^*\biggl(\sum_{l=r(J_T)}^{\mathcal{M}_T -1}  ( T_{l+1}- T_l)^2\biggr)\notag
\\
&\leq \E^*\biggl( \1_{\lbrace  T_{\kappa(T)} > T\rbrace} \1_{\lbrace  
T_{\mathcal{K}(T)} <  T_{r(J_T)}\rbrace} \sum_{l=1}^{\mathcal{M}_T -1} 
\1_{\lbrace  T_{l} \geq  T_{r(J_T)}\rbrace} ( T_{l+1} - T_{l}) \max 
\limits_{r(J_T)\leq j\leq \mathcal{M}_T -1} ( T_{j+1} - T_{j})\biggr) \notag
\\
&\ \ +\E^*( \1_{\lbrace  T_{\kappa(T)} \leq T \rbrace}) T^2 + \E^*(\1_{  
T_{\mathcal{K}(T)} \geq  T_{r(J_T)} \rbrace}) T^2. \label{2}
\end{align}
Note that, given $ (Y_{T_k})_{k\in\en}$ and $\mathcal{K}(T)$, the inter-event 
times $T_{j+1}-T_{j}$ are $\mathrm{Exp}((\lambda+\mu)Y_{T_j})$ distributed and 
independent for $j>\mathcal{K}(T)$. In order to derive an upper bound for the 
first conditional 
expectation on the right-hand side of \eqref{2}, we introduce a sequence of 
random variables 
$(U_j)_{j\in\en} $ such that, given $ (Y_{T_k})_{k\in\en}$ and $\mathcal{K}(T)$, $U_j \sim 
\mathrm{Exp}((\lambda+\mu) \min_{\mathcal{K}(T)< k} 
Y_{T_k})$ i.i.d. Then, given $ (Y_{T_k})_{k\in\en}$ and $\mathcal{K}(T)$, we have $ 
T_{j+1}- T_j \leq_{st} U_j$ for $ \mathcal{K}(T) < j \leq \kappa(T)$ and 
obtain
\begin{align}
&\E\biggl( \1_{\lbrace  T_{\kappa(T)} > T\rbrace} \1_{\lbrace  
T_{\mathcal{K}(T)} < 
 T_{r(J_T)}\rbrace} \sum_{l=1}^{\mathcal{M}_T -1} \1_{\lbrace  T_{l} \geq  
T_{r(J_T)}\rbrace} ( T_{l+1} - T_{l}) \max \limits_{r(J_T)\leq j\leq 
\mathcal{M}_T -1} ( T_{j+1} - T_{j})\biggm| Y_T>0\biggr)\notag
\\
&\leq \frac{T}{2}\E\biggl( \1_{\lbrace  T_{\kappa(T)} > T\rbrace} \1_{\lbrace  
T_{\mathcal{K}(T)} < 
 T_{r(J_T)}\rbrace} \max \limits_{r(J_T)\leq j\leq 
\mathcal{M}_T -1} ( T_{j+1} - T_{j})\biggm| Y_\infty>0\biggr) \notag \\ &\hphantom{\leq}+ 
\frac{T^2}{4}\P(Y_\infty=0|Y_T>0).\label{lll}
\end{align}
By Lemma \ref{austerben2}, the second summand is equal to
$$\frac{T^2}{4} \frac{\mu}{\lambda}e^{-(\lambda-\mu)T}.$$
The first summand of \eqref{lll} is bounded from above by 
\begin{align*}
&\frac{T}{2} \E\biggl(\E\biggl( \1_{\lbrace  T_{\kappa(T)} > T\rbrace} 
\1_{\lbrace  
T_{\mathcal{K}(T)} <  T_{r(J_T)}\rbrace} \max \limits_{\mathcal{K}(T)< j\leq 
\kappa(T)} ( T_{j+1} - T_{j})\biggm|\mathcal{K}(T),(Y_{T_k})_{k\geq 1} 
\biggr)\biggm| Y_\infty>0\biggr)
\\
&\leq \frac{T}{2} \E\biggl(\E 
\biggl(\max\limits_{1\leq j\leq \kappa(T)} U_j\biggm|\mathcal{K}(T),  
(Y_{T_k})_{k\geq 1}\biggr)\biggm| Y_\infty>0 \biggr) 
\\
&= \frac{T}{2} \E\biggl( \frac{1}{(\lambda+\mu) \min 
\limits_{\mathcal{K}(T)< 
k} Y_{T_k}} \sum_{l=1}^{\kappa(T)} \frac{1}{l}\biggm| Y_\infty>0 
\biggr) ,
\end{align*}
where the last equality follows from the formula for the expectation of 
the maximum 
of i.i.d.\ exponentially distributed random variables (see e.g.\ the 
introduction 
of 
\cite{exp}). 

Using the well-known upper bound for the harmonic sum yields
\begin{align}
 &\frac{T}{2} \E\biggl( \frac{1}{(\lambda+\mu) \min \limits_{\mathcal{K}(T)< 
k}{Y_{T_k}}} \sum_{l=1}^{\kappa(T)} \frac{1}{l}\biggm| Y_\infty>0 
\biggr)\notag
\\
&\leq \frac{T}{2(\lambda+\mu)} \E\biggl( \frac{1}{ \min \limits_{\mathcal{K}(T)< 
k}{Y_{T_k}}}\biggm| Y_\infty>0 
\biggr) (\log(\kappa(T)) + 1) \notag
 \\
 &\leq  \frac{T}{2(\lambda+\mu)}\E\biggl( \frac{1}{ \min \limits_{\mathcal{K}(T)< 
k}{Y_{T_k}}}\biggm| Y_\infty>0 \biggr)  
\biggl(\frac{3}{2}(\lambda+\mu)T + 1\biggr).\label{ab}
\end{align}

For the conditional expectation in \eqref{ab}, we obtain for $T\geq 
\frac{2\log(2)}{\lambda-\mu}$
\begin{align}
 \E\biggl( \frac{1}{ \min \limits_{\mathcal{K}(T)< 
k}  
Y_{T_k}}\biggm| Y_\infty>0 \biggr) &=  \E\biggl( \frac{1}{ \min 
\limits_{T/2\leq t}  Y_{t}}\biggm|  Y_\infty > 0 \biggr)\notag \\& \leq \biggl(1+2\log\biggl(\frac{\lambda}{\lambda-\mu}\biggr) + (\lambda-\mu)T \biggr) 
e^{-\frac{1}{4}(\lambda-\mu)T}, \label{new}
 \end{align}
 where the  last inequality follows from Lemma \ref{hl1} with $\delta = 1$ and 
$\gamma=1/4$ and 
Lemma~\ref{n3}.

Thus we can conclude that the first summand of the right-hand side of 
\eqref{lll} is smaller than or equal to
$$ \biggl(\frac 3 4 T^2+\frac{T}{2(\lambda+\mu)}\biggr) \biggl( 
1+2\log\biggl(\frac{\lambda}{\lambda-\mu}\biggr)+ (\lambda-\mu) T\biggr) 
e^{-\frac{1}{4} (\lambda-\mu)T}$$
for sufficiently large $T$.

For the last line of \eqref{2}, we can use the upper bounds from Lemma \ref{ne} 
and obtain that for ${T\geq(\frac{1}{\lambda-\mu}\log(2)\vee\frac{2(\log(4 
\lambda)-\log(\lambda-\mu))}{\lambda+\mu}})$, it is smaller than or 
equal to 
\begin{align}
    \frac{60\lambda^3(\lambda+\mu)}{(\lambda-\mu)^4} T^2
e^{-(\lambda+\mu)T}+ T^2 e^{-\frac{1}{2}\lambda T} 
+\frac{2(\lambda-\mu)}{\lambda}\biggl(\log\biggl(\frac{\lambda
} {\lambda-\mu} \biggr)+(\lambda-\mu)T\biggr)T^2e^{-(\lambda-\mu)T}.\label{52}
\end{align}

Altogether, we obtain the statement of the lemma.
\end{bew}

\section*{Appendix A3: Negligibility of multiple edges}\label{A3}
In this section, we show that multiple edges are negligible (Lemma 
\ref{negligible} below) and use this result to prove Corollary 
\ref{without_multiple} from the introduction.

First we consider the social index of a node that is connected to the node $J_T$ by an incoming edge.
\begin{lemma}\label{size-biased}
Given $Y_T=n$ for $n\in \en\setminus\{1\}$,
 the cumulative distribution function $F_{\t S}$ of the social index $\t S$ of a node $i_1$ that is connected to $J_T$ at time $T$ by an edge that was created by $i_1$ is given by
 $$F_{\t S}(s)=\int\limits_0^s\E\biggl( \frac{s_1}{\frac{s_1}{n-1}+ \frac{1}{n-1}\sum_{i=2}^{n-1} S_i }\biggr) \P^S(ds_1)= \E\biggl( \frac{S_1 \1_{[0,s]}(S_1)}{\frac{S_1}{n-1}+ \frac{1}{n-1}\sum_{i=2}^{n-1} S_i }\biggr) $$
 for $s\geq 0$.
\end{lemma}

\begin{bew} 
We condition on $Y_T=n$ for $n\in \en\setminus\{1\}$. Note that the conditional probability that $s_1$ is the social index of a node connected to $J_T$ at time $T$ by an incoming edge given $S_1=s_1,\ldots, S_n=s_n$ and $J_T=j_T\in \{2,\ldots,n\}$ is $s_1(\sum_{i\not=j_T} s_i)^{-1}$. Since the social indices are identically distributed, we thus obtain by Bayes' Theorem\fontsize{10.8}{11}
 {\allowdisplaybreaks[1]\begin{align}
  F_{\t S}(s) &= \biggl( \int\limits_{[0,s] \times [0,\infty)^{n-2}}\frac{s_1}{\sum_{i=1}^{n-1} s_i}  \P^S(ds_1) \ldots  \P^S(d s_{n-1})\biggr)  \biggl( \int\limits_{ [0,\infty)^{n-1}}\frac{s_1}{\sum_{i=1}^{n-1} s_i}  \P^S(ds_1) \ldots  \P^S(d s_{n-1})\biggr) ^{-1}\notag
  \\
  &= \biggl(\int\limits_0^s \E\biggl(\frac{S_1}{\sum_{i=1}^{n-1} S_i }\biggl|S_1=s_1\biggr) \P^S(ds_1) \biggr) \biggl(\E\biggl(\frac{S_1}{\sum_{i=1}^{n-1} S_i } \biggr)\biggr)^{-1}.\label{jgl}
 \end{align}}\normalsize
 Since we have
 $$1= \E\biggl(\sum_{j=1}^{n-1} \frac{S_j}{\sum_{i=1}^{n-1} S_i } \biggr) = \sum_{j=1}^{n-1}\E\biggl( \frac{S_j}{\sum_{i=1}^{n-1} S_i } \biggr) = (n-1) \E\biggl( \frac{S_1}{\sum_{i=1}^{n-1} S_i } \biggr),$$ the right-hand side of \eqref{jgl} is equal to
 \begin{align*}
  \int\limits_0^s\E\biggl((n-1) \frac{S_1}{\sum_{i=1}^{n-1} S_i }\biggl|S_1=s_1\biggr)\P^S(ds_1)=\int\limits_0^s\E\biggl( \frac{s_1}{\frac{s_1}{n-1}+ \frac{1}{n-1}\sum_{i=2}^{n-1} S_i }\biggr) \P^S(ds_1) .
 \end{align*}
\end{bew}
\begin{kor}\label{size_biased2}
 The expected value of the social index $\t S$ of a node $i_1$ that is connected to $J_T$ at time $T$ by an edge that was created by $i_1$ is bounded from above by
 $$ \frac{2\sigma_S^2}{c^2 \E(S) } +\frac{\E(S^2)}{\E(S)(1-c) }  $$
 for any $c\in(0,1)$.
\end{kor}
\begin{bew}
We condition on $Y_T=n\in \en\setminus\{1\}$.
 By Lemma \ref{size-biased}, the conditional expected value of the social index $\t S$ of a node $i_1$ that is connected to $J_T$ at time $T$ by an edge that was created by $i_1$ is then equal to
 $$ \int\limits_0^\infty s_1 \E\biggl( \frac{s_1}{\frac{s_1}{n-1}+ \frac{1}{n-1}\sum_{i=2}^{n-1} S_i }\biggr) \P^S(ds_1) =\E\biggl( \frac{S_1^2}{\frac{S_1}{n-1}+ \frac{1}{n-1}\sum_{i=2}^{n-1} S_i }\biggr), $$
 and we have
 \begin{align*}
  \E\biggl( \frac{S_1^2}{\frac{S_1}{n-1}+ \frac{1}{n-1}\sum_{i=2}^{n-1} S_i }\biggr)&\leq \E\biggl( \1_{\{\frac{1}{n-2}\sum_{i=2}^{n-1} S_i\leq \E(S)(1-c)\}}  (n-1) S_1\biggr)
  \\
  &\hphantom{\leq}+\E\biggl( \1_{\{\frac{1}{n-2}\sum_{i=2}^{n-1} S_i> \E(S)(1-c)\}} \frac{(n-1) S_1^2}{(n-2) \E(S)(1-c) }\biggr)
  \\
  &\leq
  \P\biggl(\frac{1}{n-2}\sum\limits_{i=2}^{n-1} S_i\leq \E(S)(1-c)\biggr) (n-1) \E(S_1)+ \frac{2\E(S_1^2)}{ \E(S)(1-c) },
 \end{align*}
 where we use the convention $\frac{0}{0}:=0$. By Chebyshev's inequality, we have 
\begin{align*}
 \P\biggl(\frac{1}{n-2}\sum\limits_{i=2}^{n-1} S_i\leq \E(S)(1-c)\biggr) &\leq  \P\biggl(\biggl|\frac{1}{n-2}\sum\limits_{i=2}^{n-1} S_i - \E(S)\biggr|\geq c\E(S) \biggr) \\ &\leq \frac{1}{(c \E(S))^2} \Var\biggl( \frac{1}{n-2}\sum\limits_{i=2}^{n-1} S_i \biggr) 
 \leq \frac{1}{n-2} \frac{\sigma_S^2}{(c \E(S))^2}, 
\end{align*}
which yields the desired result.

\end{bew}

The following lemma states that multiple edges are negligible.

\begin{lemma}\label{negligible}
 The probability that an individual picked uniformly at random at time $T$ has 
at least one 
multiple edge given the number of nodes is positive at time $T$ is of the order 
$O(T^{2} e^{-\frac{1}{6}(\lambda-\mu)T})$ as $T \to\infty$.
\end{lemma}
\begin{proof}
Let $D_T$ denote the degree of the randomly picked node $J_T$ at time $T$, 
i.e.\ the number of edges 
that are incident to the node picked uniformly at random at time $T$. 
Let $\rho_1<\ldots< \rho_{D_T}$ be the birth times of these edges. 
Condition on $J_T$, $D_T$, $\rho_1<\ldots< 
\rho_{D_T}$ and $(Y_t)_{0\leq t\leq T}$. Let $B_k$ be the event that $J_T$ creates an outgoing edge at time $\rho_k$ that is a multiple edge up 
to time $T$. Then the (conditional) probability of $B_k$ is smaller than or equal to
\begin{align*}
 \frac{D_T-1}{Y_{\rho_k}-1} .
\end{align*}

Note that $\bigcup_{k=1}^{D_T} B_k$ is the event that $J_T$ has at least one outgoing edge that is a multiple edge at time $T$. By subadditivity, we have
\begin{align*}
 &\P\biggl(\bigcup\limits_{k=1}^{D_T} B_k\biggm| J_T, D_T, \rho_1,\ldots,
\rho_{D_T},(Y_t)_{0\leq t\leq T}, Y_T>0 \biggr)
\\
&\leq \min\biggl( \sum\limits_{k=1}^{D_T} \P(B_k| J_T, D_T, \rho_1,\ldots,
\rho_{D_T},(Y_t)_{0\leq t\leq T}, Y_T>0 ),1 \biggr) 
\\
&\leq \min\biggl(\sum\limits_{k=1}^{D_T} \frac{D_T-1}{Y_{\rho_k}-1},1\biggr)
\\
&\leq \min\biggl(\frac{D_T^2}{\min\limits_{T-A_{J_T}(T)\leq t \leq T} Y_t-1},1\biggr).
\end{align*}
Taking the expectation, we obtain
{\allowdisplaybreaks[1]\begin{align}
 \P\biggl(\bigcup\limits_{k=1}^{D_T} B_k\biggm|  Y_T&>0 \biggr) \leq \E\biggl(\min\biggl(\frac{D_T^2}{\min\limits_{T-A_{J_T}(T)\leq t \leq T} Y_t-1},1\biggr)\biggm|  Y_T>0\biggr)\notag
 \\
&\leq \P(D_T\geq  e^{\frac{1}{12}(\lambda-\mu)T}|Y_T>0) + \P\biggl( 
\min\limits_{T-A_{J_T}(T)\leq t \leq T} Y_t-1 \leq  e^{\frac{1}{3}(\lambda-\mu)T}
\biggl|Y_T>0 \biggr) \notag
\\
 &\hphantom{\leq}+\E\biggl(\min\biggl(\frac{e^{\frac{1}{6}(\lambda-\mu)T}}{e^{\frac{1}{3}(\lambda-\mu)T}},1\biggr)\biggm|  Y_T>0\biggr) .\label{123}
\end{align}}

Writing $D_{\infty}$ for a random variable having the
asymptotic degree distribution $\mathrm{MixPo}(\M^{*})$ with $\M^*$ defined at the beginning of Subsection \ref{ssec:boundgen}, we obtain by conditioning on $\M^*$ that the second moment 
$\E(D_{\infty}^2)$ is equal to
$$\frac{2\alpha E(S)}{\lambda+\beta+\mu}+\frac{2\alpha 
\E((S+\E(S))^2)}{(\lambda+\beta+\mu)(\lambda+2(\beta+\mu))}\text{ (cf. 
Subsection 3.3 of \cite{b10})}.$$
Thus Theorem \ref{thm} and the Markov inequality imply 
\begin{align*}
 \P(D_T\geq  &e^{\frac{1}{12}(\lambda-\mu)T}|Y_T>0) 
\leq 
\E(D_\infty^2) e^{-\frac{1}{6}(\lambda-\mu)T}+ O(T^2 
e^{-\frac{1}{6}(\lambda-\mu)T})
\\&\ \ \ \ \ \ \leq 
\biggl(\frac{2\alpha E(S)}{\lambda+\beta+\mu}+\frac{2\alpha 
\E((S+\E(S))^2)}{(\lambda+\beta+\mu)(\lambda+2(\beta+\mu))}\biggr) e^{-\frac{1}{6}(\lambda-\mu)T}+ O(T^2 
e^{-\frac{1}{6}(\lambda-\mu)T})
\\
&\ \ \ \ \ \ =O(T^2 
e^{-\frac{1}{6}(\lambda-\mu)T}).
\end{align*}
For the second summand of the right-hand side of \eqref{123}, we obtain
\begin{align}
 \P\biggl( \min\limits_{T-A_{J_T}(T)\leq t \leq T} Y_t-1 \leq 
 e^{\frac{1}{3}(\lambda-\mu)T} \biggm|Y_T>0 \biggr) &\leq 
\P\biggl( 
\max\limits_{\frac T 2 \leq t \leq T} \frac 1 {Y_t} >   \frac{1}{e^{\frac{1}{3}(\lambda-\mu)T}+1} \biggm|Y_T>0 \biggr)\notag
\\
&\hphantom{\leq} + \P\biggl(A_{J_T}(T) > \frac{T}{2}\biggm| Y_T>0\biggr).\label{nl13}
\end{align}By Lemma \ref{ne}, the second summand of the right-hand side is of the order $O(e^{-\frac{1}{2}(\lambda-\mu)T})$ as $T\to\infty$.
With $Y_\infty=\lim\limits_{t\rightarrow\infty} Y_t$, we have{\allowdisplaybreaks[1]\begin{align*}
\P\biggl( 
\max\limits_{T/2 \leq t \leq T} \frac 1 {Y_t} >   \frac{1}{e^{\frac{1}{3}(\lambda-\mu)T}+1}, Y_\infty >0 \biggm|Y_T>&0 \biggr) = 
\frac{\P\biggl( 
\max\limits_{T/2 \leq t \leq T} \frac 1 {Y_t} >   \frac{1}{e^{\frac{1}{3}(\lambda-\mu)T}+1}, Y_\infty >0  \biggr)}{\P(Y_T>0)} 
\\
&\leq \frac{\P\biggl( 
\max\limits_{T/2 \leq t \leq T} \frac 1 {Y_t} >   \frac{1}{e^{\frac{1}{3}(\lambda-\mu)T}+1}, Y_\infty >0  \biggr)}{\P(Y_\infty>0)} 
\\
&\leq\P\biggl( 
\max\limits_{T/2 \leq t \leq T} \frac 1 {Y_t} >   \frac{1}{e^{\frac{1}{3}(\lambda-\mu)T}+1}\biggm|Y_\infty>0 \biggr).
\end{align*}}
Thus the first summand of the right-hand side of \eqref{nl13} is smaller than 
or equal to
\begin{align}
 &\P\biggl( \max\limits_{T/2 \leq t \leq T} \frac 1 {Y_t} >   \frac{1}{e^{\frac{1}{3}(\lambda-\mu)T}+1} \biggm|Y_\infty>0 \biggr)+ 
\P(Y_\infty=0|Y_T>0).\label{z2}
%  \\
%  &\leq\P\biggl( \max\limits_{\frac T 2 \leq t \leq T} \frac 1 {Y_t} >   \frac{1}{e^{\frac{1}{3}(\lambda-\mu)T}+1} \biggm|Y_\infty>0 \biggr)+  \frac{\mu}{\lambda} e^{-(\lambda-\mu)T}
\end{align}
The second summand is smaller than or equal to $\frac{\mu}{\lambda} e^{-(\lambda-\mu)T}$ by Lemma \ref{austerben2}. By Corollary~\ref{supermartingale} and the inequality (2) in Theorem 6.14 on page 99 in \cite{yeh}, the first summand of \eqref{z2} 
is bounded from above by
 \begin{align*}
  \E\biggl(\frac{1}{Y_{\frac{T}{2}}}\biggm| Y_\infty >0 \biggr) 
( e^{\frac{1}{3}(\lambda-\mu)T}+1)
  &=O(T e^{-\frac{1}{2}(\lambda-\mu)T}) ( e^{\frac{1}{3}(\lambda-\mu)T}+1)= O( T
e^{-\frac{1}{6}(\lambda-\mu)T}),
 \end{align*}
 where the first equality follows from Lemma \ref{n3}.

 We may conclude that the probability that $J_T$ has at least one outgoing edge that is a multiple edge at time $T$ is of the order $O( T e^{-\frac{1}{6}(\lambda-\mu)T})$ 
\\

Now we consider incoming edges. By conditioning on $J_T$, $(S_i)_{i\in\en}$ and $(Y_t)_{0\leq t\leq T}$, we obtain that the probability for the event $\t B_i^{(1)}$ that node $i$ creates an edge that connects $i$ to $J_T$ at time $T$ is smaller than or equal to\fontsize{11}{11}
\begin{align*}
 &\E\biggl(1-\exp\biggl(-\alpha S_{i} A_{J_T} \frac{1}{\min\limits_{T-A_{J_T}(T)\leq t \leq T} Y_t-1}\biggr) \biggl|Y_T>0\biggr)\leq \E\biggl(\alpha S_{i} A_{J_T} \frac{1}{\min\limits_{T-A_{J_T}(T)\leq t \leq T} Y_t-1} \biggl|Y_T>0\biggr)\\
& \leq \P\biggl( 
\min\limits_{T-A_{J_T}(T)\leq t \leq T} Y_t-1 \leq  e^{\frac{1}{3}(\lambda-\mu)T}
\biggm|Y_T>0 \biggr) + \P\biggl(A_{J_T} > \frac{T}{2}\ \biggm| Y_T>0\biggr)+\alpha \E(S) \frac{T}{2}e^{-\frac{1}{3}(\lambda-\mu)T}.
\end{align*}\normalsize
Since we showed above that the first two summands are of the order $O( T e^{-\frac{1}{6}(\lambda-\mu)T})$, the right-hand side is of the order $O( T e^{-\frac{1}{3}(\lambda-\mu)T})$.

We condition on $\t B_i^{(1)}$ now and denote the birth time of the edge $(i,J_T)$ corresponding to $\t B_i^{(1)}$ by $\eta_i$. Then we have for the conditional probability of the event $\t B_i^{(2)}$ that $i$ creates another edge $(i,J_T)$ in the time interval $(\eta_i,T] \subset (T-A_{J_T}(T),T]$ that survives up to time~$T$
\begin{align*}
 &\P(\t B_i^{(2)}|    \t B_i^{(1)},Y_T>0) \leq \E\biggl(1-\exp\biggl(-\alpha S_{i} A_{J_T} \frac{1}{\min\limits_{T-A_{J_T}(T)\leq t \leq T} Y_t-1}\biggr) \biggl| \t B_i^{(1)},Y_T>0\biggr) \\
 &\leq \E\biggl(\alpha \t S A_{J_T} \frac{1}{\min\limits_{T-A_{J_T}(T)\leq t \leq T} Y_t-1} \biggl|  Y_T>0\biggr) \\
& \leq \P\biggl( 
\min\limits_{T-A_{J_T}(T)\leq t \leq T} Y_t-1 \leq  e^{\frac{1}{3}(\lambda-\mu)T}
\biggm|Y_T>0 \biggr) + \P\biggl(A_{J_T} > \frac{T}{2}\ \biggm| Y_T>0\biggr)
\\ &\hphantom{\leq}+\alpha \E(\t S) \frac{T}{2}e^{-\frac{1}{3}(\lambda-\mu)T} ,
\end{align*}
where $\t S$ denotes the social index of a node connected to $J_T$ at time $T$ by an incoming edge. By Corollary~\ref{size_biased2}, the right-hand side is of the order $O( T e^{-\frac{1}{3}(\lambda-\mu)T})$.

For simplicity, we denote the $Y_T$ nodes alive at time $T$ by $1,\ldots, Y_T$ now. For the probability that $J_T$ has at least two incoming edges from the same node at time $T$, we then obtain by subadditivity
\begin{align*}
 \P\biggl( \bigcup\limits_{\substack{{i=1}\\i \not= J_T}}^{Y_T} \t B_i^{(1)}\cap \t B_i^{(2)} \biggm|Y_T&>0\biggr) \leq \E\biggl(\min\biggl(\sum\limits_{i=1}^{Y_T-1} \P\bigl( \t B_i^{(1)}\cap \t B_i^{(2)}\bigm| (Y_t)_{0\leq t\leq T}\bigr),1\biggr) \biggm|Y_T>0\biggr) 
 \\
 & =\E\bigl(\min\bigl((Y_T-1) \P\bigl( \t B_1^{(1)}\cap \t B_1^{(2)}\bigm| (Y_t)_{0\leq t\leq T}\bigr),1\bigr) \bigm|Y_T>0\bigr)%TODO 
 \\
 &\leq e^{\frac{7}{6} (\lambda-\mu)T} \P\bigl( \t B_1^{(1)}\cap \t B_1^{(2)}|Y_T>0\bigr)+ \P(Y_T-1> e^{\frac{7}{6} (\lambda-\mu)T}|Y_T>0)
\end{align*}
By the Markov inequality, the second summand of the right-hand side is of the order $O(e^{-\frac{1}{6} (\lambda-\mu)T})$. For the first summand, we have
\begin{align*}
 e^{\frac{7}{6} (\lambda-\mu)T} \P\bigl( \t B_1^{(1)}\cap \t B_1^{(2)}|Y_T>0\bigr) &= e^{\frac{7}{6} (\lambda-\mu)T} \P\bigl( \t B_1^{(1)} |Y_T>0\bigr) \P\bigl( \t B_1^{(2)} |\t B_1^{(1)},Y_T>0\bigr) 
 \\
 &= e^{\frac{7}{6}(\lambda-\mu)T} O( T e^{-\frac{1}{3}(\lambda-\mu)T}) O( T e^{-\frac{1}{3}(\lambda-\mu)T}) \hspace{-2pt}=\hspace{-1pt}O( T^2 e^{-\frac{1}{6}(\lambda-\mu)T}).
\end{align*}

Altogether, we obtain that the probability that $J_T$ has at least one multiple 
edge is of the order $O(T^2 e^{-\frac{1}{6}(\lambda-\mu)T})$.
\end{proof}

\subsubsection*{Proof of Corollary \ref{without_multiple}}
Recall that $\tilde \nu_t$ denotes the distribution of the number of 
neighbours, $\nu_t$ the degree distribution at time $t$ and $\nu$ the 
asymptotic degree distribution. Lemma \ref{negligible} implies that 
${d_{TV}(\tilde \nu_t,\nu_t)=O(t^{2} e^{-\frac{1}{6}(\lambda-\mu)t})} $ as $t \to \infty$. Furthermore, we know from Theorem 
\ref{superthm} that ${d_{TV}(\nu_t, \nu) = 
O\bigl(t^2 
e^{-\frac{1}{6}(\lambda-\mu) t}\bigr)}$ as $t \to \infty$. Thus the triangle 
inequality yields the desired result. \hspace*{\fill}
\qed
\vspace*{3mm}

% \newpage

% Ein Zitat bzw. Verweis auf \cite{bla} ist schnell gemacht \dots
% Ein Verweis auf \cite{bla, fasel} fast genauso schnell.
%\newpage

%\bibliographystyle{alpha} 

\end{document}